\documentclass[10pt]{amsart}

\usepackage{times,amsfonts,amsmath,amstext,amsbsy,amssymb,
  amsopn,amsthm,upref,eucal, amscd, graphicx}
\usepackage[T1]{fontenc}
\usepackage{color}
\usepackage{subcaption}
\usepackage{verbatim}

\usepackage{todonotes}

\newcommand{\uvec}[1]{{\protect\underline{#1}}}

\newtheorem{theorem}{Theorem}[section]
\newtheorem{lemma}[theorem]{Lemma}
\newtheorem{corollary}[theorem]{Corollary}
\newtheorem{proposition}[theorem]{Proposition}

\numberwithin{equation}{section}

\theoremstyle{definition}
\newtheorem{definition}[theorem]{Definition}

\newtheorem{remark}[theorem]{Remark}

\def\Cset{\mathbb{C}}

 \def\Rset{\mathbb{R}}
 \def\Zset{\mathbb{Z}}

\def\Iset{\mathbb{I}}
\def\Pset{\mathbb{P}}

\def\norm{||}

\def\va{ \varepsilon}
\def\wt{\widetilde}
\def\wh{\widehat}
\def\leq{\leqslant }
\def\geq{\geqslant}
\def\RA{\mathbb{R}^{\mathcal A}} 
\def\A{\mathcal{A}}
\def\iat{I_{\alpha}^t}

\def\iatn{I_{\alpha}^{t,(n)}}

\def\iab{I_{\alpha}^b}

\def\Iset{\mathbb{I}}
\def\Pset{\mathbb{P}}
\def\va{ \varepsilon}
\def\wt{\widetilde}
\def\wh{\widehat}
\def\leq{\leqslant }
\def\geq{\geqslant}
\def\RA{\mathbb{R}^{\mathcal A}} 
\def\A{\mathcal{A}}
\def\iat{I_{\alpha}^t}

\def\iatn{I_{\alpha}^{t,(n)}}

\def\iab{I_{\alpha}^b}

\def\RS{\mathbb{R}^{\Sigma}}
\def\RSO{\mathbb{R}_0^{\Sigma}}

\def\D{{\mathcal D}}

\def\iin{I^{(n)}}
\def\vU{\varUpsilon}
\def\lr{\longrightarrow}
\def\vX{\varXi}
\def\R{\mathcal R}
\def\C{\mathcal C}
\def\ug{\underline \gamma}
\def\T{\mathcal T}
\def\B{\mathcal B}
\def\SW{\mathcal {SW}}

\def\RS{\mathbb{R}^{\Sigma}}
\def\RSO{\mathbb{R}_0^{\Sigma}}

\def\cD{{\mathcal D}}

\def\iin{I^{(n)}}
\def\vU{\varUpsilon}
\def\lr{\longrightarrow}
\def\vX{\varXi}

\def\Green{}
\def\Blue{}

\begin{document}

\title{On Roth type conditions,  duality and central Birkhoff sums for i.e.m}

\author[S. Marmi]{Stefano Marmi}
\address{Scuola Normale Superiore \\ Piazza dei Cavalieri, 7 \\ 56126 \\ Pisa \\ Italy}
\email{stefano.marmi@sns.it}
\author[C. Ulcigrai]{Corinna Ulcigrai}
\address{Institut f\"ur Mathematik \\ Universit\"at Z\"urich \\ Winterthurerstrasse 190 \\ CH-8057 \\ Z\"urich \\ Switzerland \\ and  
School of Mathematics \\ University of Bristol \\ University Walk \\ Bristol \\ BS8~1TW \\ United Kingdom}
\email{corinna.ulcigrai@math.uzh.ch}
\author[J.-C. Yoccoz]{Jean-Christophe Yoccoz}
\footnote{Here we summarize the timeline connecting the various texts written by J.C.Y. reporting on our joint work and on which the writing of this article is based.  The collaboration between the three authors began during a stay of the three of us at the Mittag-Leffler Institute in April and June 2010 and again at K.T.H. in Stockholm in February 2011. It then continued with frequent meetings until the last year of life of Jean-Christophe (the last mathematical discussions of the coauthors on this research project date of July 2015, when C.U.~met J.C.Y. at CIRM, and September 2015 when S.M. visited J.C.Y. in Paris).

In August 2010 J.C.Y. wrote a first text (12 pages) containing the results obtained on dual Birkhoff sums with the title "On Birkhoff sums for i.e.m." (a further version of the same text was written in March 2011). This constitutes the heart of Section \ref{sec:dualRoth} of this paper. Notable progress was  made during a visit of S.M. and J.C.Y. to C.U. in Bristol in October 2014, when the homological interpretation emerged  and J.C.Y. wrote a new version of this draft introducing the notion of KZ-hyperbolic translation surfaces  (this homological interpretation is included in  Section \ref{sec:homological}). 
Motivated by the discussions in Bristol, the notion of absolute Roth type i.e.m. emerged and was developed during further meetings in of S.M. with J.C.Y. in Paris in December 2014, March 2015 and September 2015. This notion, together with the completeness of backward rotation numbers, is the object of the text "Absolute Roth Type and Backward Rotation Number" (14 pages) written by J.C.Y.
in April 2015. This text formed the {\Green basis} for Section \ref{sec:absoluteRoth} and parts of it are here included as two Appendixes (Appendix \ref{secbac} and Appendix \ref{app:proofs}).  
The final version of this manuscript was prepared by S.M. and C.U. who are fully and solely responsible 
{\Green for} any mistake or imprecision.

Jean-Christophe discussed publicly the results obtained in our collaboration  in his talk  ``Probl\`emes de petits diviseurs pour les \'echanges d'intervalles'' on May 20, 2015 given at the ``Journ\'ee Surfaces plates'' held at the Institut Galil\'ee of the University of Paris 13 (as Carlos Matheus informed us) and also during his talk ``Diophantine conditions for interval exchange map'' on September 28, 2015 given in Oxford as part of the  Worshop \emph{Geometry and Dynamics on Moduli Spaces} at the 2015 Clay Research conference.}

\begin{abstract}
We introduce two Diophantine conditions on rotation numbers of interval exchange maps (i.e.m) and translation surfaces: the \emph{absolute Roth type condition} is a weakening of the notion of Roth type i.e.m., while the \emph{dual Roth type} condition is a condition on the \emph{backward} rotation number of a translation surface.  We show that results on the cohomological equation previously proved in \cite{MY} for restricted Roth type i.e.m.  (on the solvability under finitely many obstructions and the regularity of the solutions) can be extended to restricted \emph{absolute} Roth type i.e.m. 
Under the dual Roth type condition, we associate to a class of functions with \emph{subpolynomial} deviations of ergodic averages (corresponding to relative homology classes) \emph{distributional} limit shapes, which are constructed in a similar way to the \emph{limit shapes} of Birkhoff sums associated  in \cite{MMY3} to functions which correspond to positive Lyapunov exponents.
\end{abstract}

\date{October 27, 2017}

\maketitle


\tableofcontents

\section{Introduction}\label{sec:intro}

Diophantine conditions play a central role in the study of the dynamics of rotations of the circle, diffeomorphisms of the circle and more in general area-preserving flows on tori. These conditions, which convey information on how well a rotation number $\alpha$ can be approximated by rationals (and hence on small divisors problems), are often expressed in terms of growth rates for the entries of the continued fraction expansion of $\alpha$. 

In the study of dynamics on surfaces, one often requires similar Diophantine conditions on interval exchange maps and linear flows. Passing from genus one to higher genus, a natural generalization of linear flows on tori is indeed provided by linear flows on \emph{translation surfaces} (see $\S$~\ref{sec:transl} for definitions); \emph{interval exchange maps}, which will be shortened throughout this paper by \emph{i.e.m.}, are piecewise isometries which, analogously to rotations in genus one, arise are Poincar{\'e} maps of linear flows (see $\S$\ref{sec:iem} for the definition). An algorithm which plays in this context an analogous role to the continued fraction expansion is the \emph{Rauzy-Veech induction}, first introduced \cite{Ra, Ve1} and used since then as an essential tool for
proving many results on the ergodic and spectral properties  of i.e.m., flows on surfaces and rational billiards, see for example \cite{Ve2, Zo3, MMY1, For3, AF, AtFo, Ulc_Wea,  Ulc_Abs,  MMY2,  Chaika, Bu, Rav, KKU}.

Diophantine conditions on i.e.m. can be prescribed imposing conditions on the behaviour of the Rauzy-Veech induction matrices and the related (extended) Kontsevich-Zorich cocycle. 
 In this spirit, in their work \cite{MMY1} on  the cohomological equation for i.e.m.,  S.~M.~, P.~Moussa and J.-C.~Y.,  define a Diophantine condition on i.e.m. which generalize the notion of Roth type rotation and under which they show that the cohomological equation can be solved under finitely many obstructions (see Theorem~\ref{thmHolder1} for a generalized statement). 
 After Forni's  celebrated paper  \cite{For1} on the cohomological
equation associated to linear flows on surfaces of higher genus, this was the first result giving an explicit diophantine condition sufficient to guarantee the existence of a solution to the cohomological equation.  The i.e.m. which satify this condition are called  \emph{Roth type} i.e.m. (or i.e.m. of Roth type) and have full measure (as shown in \cite{MMY1}, see also \cite{MMM} for the sketch of a simpler proof based on the results in \cite{AGY}). A reformulation and a strenghthening of the Roth type condition (namely, \emph{restricted} Roth type) were then defined in \cite{MMY2} also for \emph{generalized} i.e.m. and provide the Diophantine condition under which a linearization result is proved in \cite{MMY2}. 


  


Two new  Diophantine conditions related to the Roth type condition for i.e.m.   are introduced in this paper, for the applications that we explain in  \S~\ref{sec:intro_dual} and \S~\ref{sec:intro_absolute} below. More precisely (in $\S$~\ref{sec:absoluteRoth}) we introduce the notion of  \emph{absolute Roth type} i.e.m., thus defining a class of i.e.m. which include and generalize Roth type i.e.m.~but for which the results mentioned before on linearization and the cohomological equation still hold.  
 We then define the notion of \emph{dual Roth type} for a translation surface (or more precisely, for its suspension data or backward rotation number, see $\S$~\ref{sec:dualRoth}). 
This is a Diophantine condition which is \emph{dual} to the Roth type condition for i.e.m. in a sense which will be made precise further on (see $\S$~\ref{sec:dualRoth} and $\S$~\ref{sec:homological}). 
Let us now explain the motivation for introducing these  conditions and the results which we proved assuming them, starting with the notion of dual Roth type.  


\smallskip
\subsection{Dual Roth type and distributional limit shapes.} \label{sec:intro_dual}
Results on deviations of ergodic averages and ergodic integrals are a central part of the study of i.e.m.\ and translation flows, see for example the works by Zorich \cite{Zo3, Zo2}, 
Forni \cite{For2}, Avila-Viana \cite{AV} and Bufetov \cite{Bu} among others.  
Deviations of ergodic averages, i.e. the oscillations of the Birkhoff sum $S_n f(x):= \sum_{k=0}^{n-1} f(T^k x)$ of a function $f:[0,1] \to \mathbb{R}$ of zero average over the orbit of (typical) point $x \in [0,1] $  under a  i.e.m. $T: [0,1] \to [0,1]$ are of polynomial nature. In \cite{Zo3} Zorich shows for example that for a typical i.e.m. $T$ and any {\Green  mean-zero function} $f$ constant on the intervals exchanged by $T$, we have $S_n f(x)=O(x^\nu)$ for some power exponent $\nu<1$; more precisely 
 for a full measure set of  $T$ there exists $\nu=\nu(f)$ such that for all $x \in [0,1]$, 
\begin{equation}\label{nu}
\limsup_{n\to + \infty} \frac{\log  S_n f(x)}{\log n} = \nu.
\end{equation}
Remark also that if $f$ is a \emph{coboundary} for $T$ with bounded transfer function, i.e. $f=g\circ T - T$ where the \emph{transfer function} $g:[0,1] \to \mathbb{R}$ is bounded, then Birkhoff sums $S_nf (x)$ are uniformely bounded (and in particular $\nu=0$). 
The power exponent can be understood in terms of Lyapunov exponents of (a suitable acceleration) of the Kontsevich-Zorich cocycle associated to the Rauzy-Veech induction \cite{Zo1, Zo2}. In particular,   
 $\nu$ is a ratio of Lyapunov exponents and depends on the position of the piecewise constant function $f$ (identified with a vector of $\mathbb{R}^d$) with respect to the Oseledets filtration of the Kontsevich-Zorich cocycle. { \Green Using this interpretation, it follows  from the work of Forni \cite{For2} (see also \cite{AV}) that,  for a typical choice of function $f$, $\nu$ is positive; furthermore one also has $\nu<1$ as an immediate consequence of the work of Veech \cite{Ve5} (see also \cite{For2} for a  more general result)). }
A powerful result of similar nature for ergodic integrals of smooth area-preserving flows was  proved by Forni in \cite{For2}: the \emph{power spectrum} of ergodic integrals is related to Lyapunov exponents of the Kontsevich-Zorich cocycle and Forni's invariant distributions.

	


\smallskip
A finer analysis of the behaviour of Birkhoff sums or integrals, beyond the \emph{size} of oscillations, appears in the works \cite{Bu, MMY3}. In \cite{MMY3}, motivated by the study of wandering intervals in affine i.e.m., S.~M., P.~Moussa and J.C.Y. introduced an object called \emph{limit shape} and used it to describe the \emph{shape} of ergodic sums (see $\S$~3.4 and $\S$~3.7.3 in  \cite{MMY3}). Roughly speaking these are obtained by looking at suitably rescaled Birkhoff sums, where time is renormalized according to the leading Laypunov exponent of the Kontsevich-Zorich cocycle, whereas the range of the sum  is renormalized using one of the other positive exponents, according to the choice of $f$. After this double rescaling one obtains a sequence of shapes exponentially converging (in the Hausdorff metric)  to the graph of a H\"older function.  
 In \cite{Bu}, Bufetov studies limit theorems for ergodic integrals of translation flows and describe weak limit distributions in terms of objects that he calls \emph{H\"older cocycles} (or, in the context of Markov compacta, \emph{finitely-additive measures}) and turn out to be \emph{dual} to Forni's invariant distributions (see \cite{Bu} for details). We remark that limit shapes and H\"older cocycles, despite having been introduced independently, are intrinsically related: limit shapes are essentially \emph{graphs} of  H\"older cocycles along flow leaves. {\Green Let us remark that similar results can also be proved for horocycle flows on negatively curved surfaces, see \cite{BuFo} were the existence of H\"older cocycles is proved in this context.}

\smallskip
Both limit shapes and H\"older cocycles are associated to functions which display \emph{truly polynomial} deviations, i.e. for which the exponent $\nu$ in \eqref{nu} is strictly positive.
{ \Green More precisely, from the} work of  Forni {\Green \cite{For2}} and Avila-Viana \cite{AV}  it follows that for a typical i.e.m. $T$ with  $d$ exchanged  subintervals,  the extended Kontsevich-Zorich cocycle  has $g$ positive Lyapunov exponents, $g$ negative and $s-1$ zero ones, where $d=2g+s-1$ and $g$ and $s$ can be computed from the combinatorics of $T$ ($g$ is the genus and $s$ is the number of marked points of any translation surface which suspends $T$, see \S~\ref{sec:transl}). For typical i.e.m. $T$, functions which are \emph{coboundaries} with bounded transfer functions  
  and hence have \emph{bounded} Birkhoff sums, can be associated to  the \emph{stable} space of the Kontsevich-Zorich cocycle, which correspond to \emph{negative} Lyapunov exponents. 

We construct in this paper  objects similar to the limit shapes introduced in \cite{MMY3} for functions which display \emph{subpolynomial} deviations of ergodic averages, i.e. functions for which $\nu=\nu(f)$ in \eqref{nu} is equal to zero, but are not coboundaries.  An important example of this type of function arise when considering rotations of the circle (which correspond to i.e.m. with $d=2$) and a  mean zero function $\chi=\chi_{[0,\beta]}-\beta$, where $\chi_{[0,\beta]}$ is the characteristic function of the interval $[0,\beta]$. The function $\chi$ can be seen as a piecewise constant function on a i.e.m. with $d=3$. It is well known that in this case $S_n \chi$ display logarithmic deviations (which can be described for example in terms of the Ostrowski expansion of $\beta$ with respect to $\alpha$, see for example \cite{BI}).  {\Green Results on the subdiffusive behaviour of these Birkhoff sums were proved for example in \cite{Huveneers, ConzeIsolaLeBorgne, ConzeLeBorgne}} {\Green (see also  \cite{BBH} for a related result in the context of substitutions).}  
Celebrated results on limit distributions for the Birkhoff sums $S_n \chi$ under additional randomness were proved for example by Kesten in \cite{Ke1, Ke2} {\Blue (where the rotation number is randomized)} and by Beck in \cite{Beck1, Beck2} {\Blue (where time is randomized, {\Blue see also \cite{DS, PS, BrUl, DS3} for recent extensions of Beck's temporal limit distribution results and \cite{DS, DS2} for results in the context of horocycle flows).}}

More in general, we show that for a typical i.e.m. $T$ {\Blue (for any number of exchanged intervals $d\geq 3$)} the characteristic function $\chi_{[0,\beta]}$ can be \emph{corrected} by adding a function constant on the intervals exchanged by $T$ in order to display subpolynomial deviations (see Propostion \ref{prop:correction}) and hence these seem to be 
 natural functions for future investigations towards a generalizations of Kesten's and Beck's results. 
 Functions with subpolynomial deviations can more generally be associated to \emph{relative} homology classes in $H_1(M,\Sigma, \mathbb{R})$  where $M$ is a (typical) translation surface with conical singularities $\Sigma $ (see $\S$~\ref{sec:homological}).

 We prove in \S~\ref{sec:dualRoth} that for a full measure set of translation surfaces, to each function with subpolynomial deviations (or more in general, in \S~\ref{sec:homological}, to each relative homology class with subpolynomial deviations), we can associate \emph{limit distribution} on H\"older functions, which is constructed in a similar way to the limit shapes in \cite{MMY3} (see $\S$~\ref{sec:functions} and $\S$~\ref{sec:convergence} for details). More precisely, limit shapes in \cite{MMY3}   are H\"older functions defined as pointwise limit of a sequence of piecewise affine functions 
 obtained by plotting suitably rescaled graphs of Birkhoff sums. 
Convergence exploits positivity of the Lyapunov exponent, which implies that oscillations on different scales are of exponentially smaller orders. 
 We  define similarly a sequence of piecewise affine functions in our setup, but they fail to converge pointwise since all oscillations are of the same magnitude. We prove on the other hand that there is convergence in the sense of distributions, when integrating these graphs against H\"older continuous functions (see Theorem~\ref{thm:convergence}). 

The  Diophantine condition that we need to ensure distributional convergence of this sequence of piecewise affine functions turns out to be a Roth type-like condition, but not for the usual {\Green Kontsevich}-Zorich cocycle, but for a dual cocycle. This is the (weak) \emph{dual Roth type} condition that we define in \S~\ref{sec:dualRothcondition}. In the proof of convergence, we also need to introduce the notion of \emph{dual special Birkhoff sums} (see \S~\ref{sec:dualBS}). This is based on a \emph{duality} between horizontal and vertical flow and future and past of the Teichmueller geodesic flow. The same duality is also exploited crucially in the paper by Bufetov \cite{Bu}. This duality has also a homological intepretation, which we explain in Section \ref{sec:homological}. 

The construction of limit distributions can also be done at the level of homology bases, as explain in Section \ref{sec:homological}.  Distributional limit objects  can hence be associated to relative homology classes $\Upsilon \in H_1(M, \Sigma, \mathbb{R})$ which have non trivial image  $\partial \Upsilon$ by the boundary operator $\partial : H_1(M, \Sigma, \mathbb{R}) \to H_0 (M \backslash \Sigma, \mathbb{R})$ (see $\S$. In view of the connection between the limit shapes in \cite{MMY3} and Bufetov's H\"older cocycles \cite{Bu} and the duality of the latter and Forni's invariant distributions (explained in \cite{Bu}, see ADD), our distributional limit objects are presumably connected to invariant distributions in relative homology.


\subsection{Absolute Roth type}\label{sec:intro_absolute}

To motivate the notion of absolute Roth type i.e.m., consider the following example. 
Let $[0,\beta)$ be an interval contained in the domain of an i.e.m. $T_0$ and let $\chi $ be the   (possibly {corrected}, as explained above) characteristic function of  $[0,\beta)$. To study the Birkhoff sums $S_n \chi $ of $\chi$ under $T_0$, 
 it is convenient to \emph{add} the discontinuity $\beta$ of $\chi$ as \emph{marked point} (or fake discontinuity), so that the function $\chi$ can be considered as function which is  \emph{constant} over the subintervals exchanged by an i.e.m. $T$ of $d+1$ of intervals. It is natural for many questions  to \emph{fix} $\chi$ and \emph{vary} the i.e.m. $T_0$. One would hence like to exploit {\Green a} Diophantine condition on $T$ which only depends on the underlying $T_0$ and not on the position of the fake discontinuity $\beta$.   

Let us recall that the number of exchanged intervals $d_0$ coincides with the dimension of the \emph{relative} homology $H(S_0, \Sigma_0, \mathbb{R})$, where  $M_0$ is any translation surface which \emph{suspends} the i.e.m. $T_0$ (see \S~\ref{sec:transl} for definitions) and $\Sigma_0$ is the set of conical singularities of $M_0$. The \emph{Roth type} condition for i.e.m. is defined in terms of the \emph{Rauzy-Veech cocycle}, which acts on the relative homology $H(M_0, \Sigma_0, \mathbb{R})$. If  $T_0$ in the above example is such that $d_0=2g$, where $g$ is the genus of $M_0$, i.e. $d_0$ equals the dimension of the \emph{absolute} homology $H(M_0, \mathbb{R})$, the growth of $S_n \chi $, where $\chi$ is a \emph{corrected} characteristic function,  is given by the growth of a \emph{relative} homology class in $H(M, \Sigma, \mathbb{R})$, where $M$ suspends $T$ and has singularities $\Sigma$. One would hence like to introduce Diophantine conditions on $T$ which depends only on the \emph{absolute} homology of $M$.

These considerations led us to define the notion of \emph{absolute Roth type} i.e.m. (given in \S~\ref{sec:absoluteRoth}). This is a  Diophantine condition on an i.e.m. $T$ defined in terms of the Rauzy-Veech cocycle, and hence on the action on   $H(M, \Sigma, \mathbb{R})$, but which depends on how the cocycle acts on the \emph{absolute} homology $H(M, \mathbb{R})$, which can be identified with a subspace of $\mathbb{R}^d$ defined in terms of the  combinatorial data of $T$ (see \S~\ref{sec:absoluteRoth} for details). We stress that the absolute Roth type condition is \emph{weaker} than Roth type, i.e. Roth type i.e.m. satisfy the absolute Roth type condition (see Lemma \ref{lem:Roth_implies_absRoth} in \S~\ref{ssabs}). 

While we were thinking of this notion, Chaika and Eskin \cite{ChaikaEskin} proved a result on Oseledets genericity of the Kontsevich-Zorich cocycle which, as an application,  implies  that given any translation surface $S$, for a.e. direction $\theta$ the i.e.m. obtained as Poincar{\'e} maps of the linear flow in direction $\theta$ to a segment in \emph{good position} (see \S~\ref{sec:transl} for definitions) satisfy the absolute Roth type condition. This motivated us to prove that several of the results on the cohomological equation for i.e.m. of (restricted) Roth type, in particular the main results in \cite{MY}, are still valid under the weaker (restricted) absolute Roth type condition. These results are presented in \S~\ref{sec:results} (see also \S~\ref{sec:aedirection} and Appendix~\ref{app:proofs}{\Green)}. 
 



\subsection{Outline of the paper}
We first recall background material, in particular definitions and notations for i.e.m and Rauzy Veech induction (see Section \ref{sec:background}). 
The reader who is already familiar with this can skip to Section \ref{secnot} where we simply summarize the notation.

\smallskip
In Section \ref{sec:absoluteRoth} we first recall the Roth type condition introduced by \cite{MMY1} and we then define absolute Roth type. We then states the results, first proved for Roth type, that still hold under this weaker condition. We prove in this section two crucial estimates that shows that Roth type allows to control the lengths of inducing subintervals for the positive acceleration of Rauzy-Veech  induction. 
The proofs of the results which extend from Roth type to absolute Roth type can then be adapted with minor modifications, which, for completeness, we discuss in an appendix (see Appendix \ref{app:proofs}).  We then recall the result by Chaika and Eskin in \cite{ChaikaEskin} and get as a Corollary that the results on the cohomological equation holds for every surface in a.e. direction.

\smallskip
In Section \ref{sec:dualRoth} we introduce the second variation of Roth type condition studied in this paper, namely dual Roth type. We first introduce the notion of dual Birkhoff sums, which are used to give the definition.  We also prove that some classical results, in particular estimates on the growth of special Birkhoff sums of BV fucntions, also holds for special dual Birkhoff sums under this dual Roth type condition.

\smallskip
In Section \ref{sec:centraldistributions} we study a distributional analogue of limit shapes for functions with subpolynomial deviations of ergodic averages. We first show how a natural class of functions belonging to the  central space of the  can be obtained by correcting indicatrix functions. We then define the affine graphs which plot the behaviour of central Birkhoff sums. Under the Roth type DC that we introduced in the previous section, we when show that these graphs converge in the sense of distributions (see Theorem~\ref{thm:convergence}).

\smallskip
Finally, in Section \ref{sec:homological}, we give a homological {\Green interpretation} of distributional limit shapes. We define suitable dual bases of relative homology and show that the graphs which we studied in the previous section and their convergence have a homological interpretation.

\smallskip
The Appendix \ref{secbac} contain the involved proof that backward rotation numbers are infinitely complete. The Appendix \ref{app:proofs}, as already mentioned, contains a summary of the structure of the proof of the results on the cohomological equation with an indication of the changes one has to do to adapet them to absolute Roth type i.e.m.

\subsection{Acknowledgements}
 We would like to thank Jon Chaika, Alex Eskin and Carlos Matheus for several useful discussions.  
 {\Blue We also thank the anonimous referee for useful comments.}   
C.U. is partially supported by the ERC Starting Grant \emph{ChaParDyn} and also acknowledges the support of the \emph{Leverhulme Trust} through a \emph{Leverhulme Prize} and the Royal Society and the  Wolfson Foundation through a \emph{Royal Society Wolfson Research Merit Award}.    
S.M. acknowledges support from UnicreditBank through the \emph{Dynamics and Information Research Institute} at the Scuola Normale Superiore.
The authors would like thank the University of Bristol, Centro di Ricerca Matematica Ennio de Giorgi, Mittag-Leffler Institute and IHES for hospitality during the visits that made this collaboration possible.  The research leading
to these results has received funding from the European Research Council under the European Union Seventh
Framework Program (FP/2007-2013) / ERC Grant Agreement n.~335989.

\section{Backgound material and notations}\label{sec:background}
In this section we recall basic definitions and notation for interval exchange maps and Rauzy-Veech induction. We refer to the lecture notes \cite{Y1,Y2,Y4,Zo1} for further information.  
 The reader who is already familiar with { this material might want to skip it and} to refer to \S~\ref{secnot} which provides a brief summary of the notation used.


\subsection{Interval exchange maps}\label{sec:iem} 
Let $I \subset \Rset$ be a bounded open interval.   An \emph{interval exchange map} (i.e.m)  i.e.m. $T$ on $I$ is defined by the following data.  Let $\A$ be an alphabet with $d\geq 2$ symbols. 
Consider two partitions (modulo $0$) of $I$ into $d$ open intervals indexed by $\A$ (the \emph{top} and \emph{bottom} partitions):
$$
I=\sqcup_{\alpha \in \A} I_\alpha^t = \sqcup_{\alpha \in \A} I_\alpha^b
$$
such that $|I_\alpha^t|=|I_\alpha^b|$ for every $\alpha \in A$.  The map $T$ is defined on $\sqcup I_\alpha^t$ so that its restriction to each $I_\alpha^t$ is a translation onto the corresponding $I_\alpha^b$. The partitions are determined by combinatorial data on one side, lengths on the other side, as follows.     The combinatorial data is a pair $\pi=(\pi_t,\pi_b)$ of bijections from $\A$ onto
$\{1, \ldots ,d\}$ which indicates in which order the intervals are met in the top and in the bottom partition respectively. We always assume that the combinatorial data are {\it irreducible}: for $1\leq k<d$, we have
\begin{equation}\label{def:irreducible}
\pi_t^{-1}(\{1,\ldots ,k\})\not=\pi_b^{-1}(\{1,\ldots ,k\})\; .
\end{equation}
The length data prescribe the lengths  $(\lambda_\alpha)_{\alpha\in\A}$ of the partition elements, i.e.\  $|I_\alpha^t|=|I_\alpha^b|= \lambda_\alpha$ for $\alpha \in A$. 
The i.e.m.\ $T=T_{\pi,\lambda}$ determined by these data acts on $I:=(0,\sum_{\alpha\in\A}\lambda_{\alpha})$; the subintervals of the top partition are
$$
I_\alpha^t = I_\alpha^t (T): = \left( \sum_{\pi_t\beta<\pi_t\alpha}\lambda_\beta , \sum_{\pi_t\beta\le\pi_t\alpha}\lambda_\beta\right)
$$
and those in the bottom partition are
$$
I_\alpha^b=I_\alpha^b(T) : = \left( \sum_{\pi_b\beta<\pi_b\alpha}\lambda_\beta , \sum_{\pi_b\beta\le\pi_b\alpha}\lambda_\beta\right)
\; .
$$
We also write simply $I_\alpha$ for $I_\alpha^t$, $\alpha \in A$ and refer to them as the \emph{intervals exchanged by } $T$ (or simply exchanged intervals).    
The i.e.m $T$ sends $\iat$ on $\iab$ by the translation $x \mapsto x+\delta_{\alpha}$, where the 
{\it translation vector} $(\delta_\alpha)_{\alpha\in\A}$ is given by
$$
\delta_\alpha = \sum_\beta\Omega_{\alpha\beta}\lambda_\beta
$$
and  $\Omega_{\pi}  = (\Omega_{\alpha,\beta})_{\alpha,\beta \in \A}$ is the antisymmetric matrix associated by the combinatorial data $\pi$ defined by 
\begin{equation}\label{Omega_def}
\Omega_{\alpha \,\beta}=\left\{
\begin{array}{cc}
+1 & \text{if } \pi_t(\alpha)<\pi_t(\beta),\pi_b(\alpha)>\pi_b(\beta),\\
-1 & \text{if } \pi_t(\alpha)>\pi_t(\beta),\pi_b(\alpha)<\pi_b(\beta),\\
0 &\text{otherwise.}
\end{array} \right.
\end{equation}
The points $u^t_1<\cdots<u^t_{d-1}$ separating the $ \iat$ are called the {\it singularities } of $T$.  The points $u^b_1<\cdots<u^b_{d-1}$ separating the $ \iab$ are called the singularities  of $T^{-1}$. We also define $u_0,u_d$ so that $I=(u_0,u_d)$. We will also write $u_i^t(T)$ and $u_i^b(T)$ when we want to stress the dependence on $T$.  
 The action of $T$ on $\sqcup I^t_\alpha\subset I$ can be extended to $I$ by right-continuity.  A {\it connection} is a triple $(u_i^t,u_j^b,m)$, where $m$ is a nonnegative integer, such that $ T^m(u_j^b)=u_i^t$.

\subsection{Translation surfaces and suspension data}\label{sec:transl}

Let $M$ be a compact connected orientable topological surface of genus $g\ge 1$ and let $\Sigma = \{A_1, \ldots , A_s\}$ be a non-empty finite subset 
of $M$. 
A {\it structure of translation surface} on $(M, \Sigma)$ is a maximal atlas $\zeta$ for $M- \Sigma$ of charts by open sets of $\mathbb{C} \simeq \mathbb{R}^2$ which satisfies two properties:
any coordinate change between two charts is locally a translation of $\mathbb{R}^2$ and 
for every $1 \leq i \leq s$, there exists an integer $\kappa_i \geq 1$ and a ramified covering $\pi : (V_i , A_i) \rightarrow (W_i , 0)$ of degree $\kappa_i$ such that every injective restriction of $\pi$ is a chart of $\zeta$. Here $V_i$ and $W_i\subset\mathbb{R}^2$ are neighborhoods respectively of  $A_i$ and $0$. 

A translation surface structure  naturally provides a complex structure and an holomorphic 
 $1$-form $\omega$ which does not vanish on $M-\Sigma$ and has at each point $A_i$ a zero of order $\kappa_i -1$. It also provides a \emph{flat metric} on $M-\Sigma$: the metric exhibits a (true) singularity at each $A_i$ such that 
$\kappa_i >1$ and the total angle around each $A_i \in \Sigma$ is $2 \pi \kappa_i $. The genus $g$ of the surface and the order of the zeros of $\omega$ are related by 
the Riemann-Roch formula $2 g-2= \sum^s_{i=1} \; (\kappa_i-1)$. 

The geodesic flow of the flat metric on $M- \Sigma$ gives rise to a $1$-parameter family of constant unitary directional 
flows on $M- \Sigma$, called \emph{linear flows},  containing in particular a \emph{vertical flow} $\partial / \partial y$ and a \emph{horizontal flow} $\partial / \partial x$. 
An orbit of the vertical (horizontal) flow which ends (resp.~starts) at a point of $\Sigma$ is called an {\it ingoing } (resp. {\it outgoing}) {\it vertical (horizontal) separatrix}. A {\it vertical connection} (resp.~{\it horizontal connection}) is an orbit of the vertical (resp.~horizontal) flow which both starts and ends at a point of $\Sigma$. 

The group $GL(2,\mathbb{R})$ acts naturally on traslation surfaces: given the structure of translation surface $\zeta$ on $(M, \Sigma)$ and $g \in GL(2,\mathbb{R})$, one defines a new structure $g\cdot \zeta$ by postcomposing the charts of $\zeta$ by $g$. There are two distinguished one-parameter subgroups of $GL(2,\mathbb{R})$ whose actions will turn out of be important (for example in \S~\ref{sec:aedirection}): the Teichm\"uller geodesic flow $(g_t)_{t\in\mathbb{R}}$ and the group of rotations $(r_\theta)_{\theta\in {S_1}}$, which are given respectively by
$$g_t = \begin{pmatrix} e^t & 0 \\ 0 & e^{-t} \end{pmatrix},\quad  t\in\Rset; \qquad  r_\theta =  \begin{pmatrix} \cos\theta & \sin\theta \\ -\sin\theta  & \cos\theta \end{pmatrix}, \quad \theta \in S^1.
$$
The rotation $r_\theta $  has the effect of rotating the flat surface $(M, \Sigma, \zeta)$ by the angle $\theta$. 


\medskip

Given a  translation surface $(M, \Sigma, \zeta)$, an open bounded horizontal segment  $I$ is in {\it good position} if 
\begin{enumerate}
\item $I$ meets every vertical connexion; 
\item the endpoints of $I$ are distinct and each of them either belongs to $\Sigma$ or is connected to a point of $\Sigma$ by a vertical segment not meeting $I$.
\end{enumerate}

If there is no vertical connexion, or no horizontal connexion, then such segments always exist. One may even ask that the left endpoint of $I$ is in $\Sigma$ ([Yo4, Proposition 5.7, p.16]) In particular, one can always find $g \in GL(2,\Rset)$, preserving the vertical direction, and a segment in good position which is horizontal for $g.\zeta$. The Poincar{\'e} first return map of the vertical flow on a horizontal segment in good position is a i.e.m. with minimal possible number of exchanged intervals. 

\medskip

Conversely, starting from an interval exchange map $T$, one can construct, following Veech \cite{Ve2}, a translation surface $M_T$ for which $T$ appears as a return map on a standard interval and the vertical flow on $M$ becomes a suspension flow on $T_I$. Furthermore, if $M$ is a translation surface with no horizontal and no vertical saddle connections and 
 $I$ is a horizontal interval in good position, the return map $T_I$ of the vertical flow of $M$ on $I$ is an i.e.m. of minimal number of exchanged intervals  and the translation surface $(M,\Sigma, \kappa, \zeta)$ can be recovered from  $T_I$ and appropriate {\it suspension data}  via Veech's {\it zippered rectangles construction} 
(\cite{Ve2}, see e.g.~\cite{Y4}, Section 4).  We will limit ourselves here to explain a simple version of this construction. 

Let $\pi= (\pi_t,\pi_b)$ be a combinatorial data. 
A vector $\tau \in \Rset^{\A}$ is a \emph{suspension vector} (for $\pi$) if it satisfies the following inequalities
\begin{equation}\label{def:suspdatum}
 \sum_{\pi_t (\alpha) < k} \; \tau_\alpha > 0 \; ,  \hspace{1cm} \sum_{\pi_b (\alpha) < k} \; \tau_\alpha <0 \hspace{2cm} \hbox{for all} \; 1 < k \leq d. \;
\end{equation}
We denote by $\Theta_\pi$ the cone of all suspension vectors for $\pi$. This cone is non-empty exactly when $\pi$ is irreducible. In this case indeed, one can see that $\theta_\pi$ contains the canonical suspension vector
\begin{equation}\label{cansusp}
\tau^{can}_\alpha = \pi_b(\alpha) - \pi_t(\alpha) \quad , \quad \alpha \in \A \; .\end{equation}
Let $T$ be an  i.e.m.\ with irreducible combinatorial data $\pi= (\pi_t,\pi_b)$ acting on an interval $I(T) = (u_0(T),u_d(T))$ and let $\tau \in \Rset^{\A}$ be a suspension vector. We identify as usual $\Rset^2$ with $\Cset$ and set,  
  for $\alpha \in \A$, 
\begin{equation}\label{def:zeta}
\lambda_{\alpha} = |\iat |=|\iab |, \qquad 
\zeta_{\alpha}:=\lambda_{\alpha}
+\sqrt {-1} \, \tau_\alpha , \end{equation} 
and then, for $0 \leq i \leq d$
$$ U_i = u_0 + \sum_{\pi_t (\alpha) \leq i} \zeta_{\alpha}, \quad V_i = u_0 + \sum_{\pi_b (\alpha) \leq i} \zeta_{\alpha}.$$
Consider the \emph{top} polygonal line connecting the points $U_0, U_1, \dots, U_d$ and 
the \emph{bottom} polygonal line connecting the points $V_0, V_1, \dots, V_d$. Both lines have the same endpoints, i.e. $U_0 = V_0 =u_0(T)$, $U_d = V_d$, and that, from the suspension
 condition \eqref{def:suspdatum}, all intermediary points in the top (resp. bottom) line lie in the upper (resp. lower) half-plane, i.e.\ $\Im U_i >0>\Im V_i$ for 
$0<i<d$.

Assume that {\it the two lines do not intersect except from their endpoints}. This is in particular the case for $\tau=\tau^{can}$. Then we can construct a translation surface $M(\lambda, \pi, \tau)$ considering the closed polygon bounded by the two lines and identifying for each $\alpha \in \A$ the $\zeta_\alpha$ side of the top line with the $\zeta_\alpha$ side of the bottom line through the appropriate translation. This produces a translation surface, whose singularity set $\Sigma$ is  by definition the image of the vertices of the polygon. 

The above non-intersection assumption, unfortunately, is not valid for all values of the data, but one can still associate to data  $(\lambda, \pi, \tau)$ as above a translation surface $M(\lambda, \pi, \tau)$. Consider the vector $q= -\Omega_\pi \tau$ and observe observe that, for all $\alpha \in {\A}$, since
$$q_\alpha = \sum_{\pi_t (\beta) < \pi_t (\alpha)} \tau_\beta - \sum_{\pi_b (\beta) < \pi_b (\alpha)} \tau_\beta,$$
it follows  from the suspension data condition \eqref{def:suspdatum} that $q_\alpha>0$ for each $\alpha \in \A$. The surface $M(\lambda, \pi, \tau)$ the general case is then built by indenfitying the boundaries of the $d$ rectangles $R_\alpha:= I_\alpha^t (T) \times [0, q_\alpha]$, $\alpha \in \A$, where $I_\alpha^t (T) \subset I:=I(T)$ are the intervals exchanged by $T=T_{\pi,\lambda}$. The top horizontal side of  $R_\alpha$ is identified by a translation to  $I_\alpha^b(T)\subset I$. The vertical sides are then identified so that the images of the points $U_0, U_1, \dots, U_d$ defined above are exactly the singularities of the idenfitied surface. We refer the reader to \cite{Y4}, Section 4 for details. This construction is known as Veech's {\it zippered rectangles construction}. Notice that $T$ is the Poincar{\'e} map first return of the vertical flow on $M(\lambda, \pi, \tau)$ to $I$ and $q_\alpha$ is the return time of any point $x \in I_\alpha=I_\alpha^t(T)$. 

We will denote by $M_T$ the surface $M(\lambda, \pi, \tau^{can})$ obtained as suspension from the canonical suspension data $\tau^{can}$.   
For $M=M_T$ or more in general $M=M( \pi, \lambda, \tau)$, one can see that the cardinality $s$ of the singularity set $\Sigma$, the genus $g$ of $M$ and the number $d$ of intervals are related by 
$d=2g+s-1$. 
Both $g$ and $s$ only depend on $\pi$, the genus $g$ can be computed from the antisymmetrix matrix $\Omega$ defined in \eqref{Omega_def} above: the rank of 
$\Omega$ is $2g$.  Another way to compute $s$ (and thus $g$) consists in going around the
marked points, a procedure that we briefly recall in \S~\ref{ssboundary}.

\begin{remark}\label{relativeHidentification}
One can show that $\{ \zeta_{\alpha},\, \alpha \in \A\}$ provide a basis of the relative homology group $H_1(M,\Sigma,\Rset)$  and hence one can  identify $\RA$ with $H_1(M,\Sigma,\Rset)$ by the map $(x_\alpha)_\alpha \mapsto [\sum_\alpha x_\alpha \zeta_\alpha]$ (see \cite{Y1},\cite{Y4}). 
\end{remark}

\subsection{A step of Rauzy--Veech algorithm, Rauzy Veech diagrams and Rauzy-Veech matrices}\label{secRV}

Let $T$ be an i.e.m with no connection. We have then $u^t_{d-1} \ne u^b_{d-1}$. Set $\widehat u_d := \max(u^t_{d-1},u^b_{d-1})$, $\widehat I := (u_0,\widehat u_d)$, and denote by $\widehat T$ the first return map of $T$ in $\widehat I$. The return time is $1$ or $2$. 
One checks that $\widehat T$ is an i.e.m on $\widehat I$ whose combinatorial data $\widehat \pi$ are canonically labeled by the same alphabet $\A$ than $\pi$ (cf.~\cite{MMY1}, p.~829). Moreover $\widehat T$ has no connection; this allows to iterate the algorithm.
We say that $\widehat T$ is deduced from $T$ by an elementary step of the Rauzy--Veech algorithm. We say that the step is of {\it top} (resp. {\it bottom}) {\it type} if $u^t_{d-1} < u^b_{d-1}$ (resp. $ u^t_{d-1} > u^b_{d-1}$). One then writes $\widehat \pi = R_t(\pi)$ (resp. $\widehat \pi = R_b(\pi)$).

\smallskip

A {\it Rauzy class} on the alphabet $\A$ is a nonempty set of irreducible combinatorial data
which is invariant under $R_t,R_b$ and minimal with respect to this property.
A {\it Rauzy diagram} is a graph whose vertices are the elements of a Rauzy class and whose arrows
connect a vertex $\pi$ to its images $R_t(\pi)$ and $R_b(\pi)$. Each vertex is therefore the origin of two arrows. As
$R_t,R_b$ are invertible, each vertex is also the endpoint of two arrows. An arrow connecting $\pi$ to $R_t(\pi)$ (respectively $R_b(\pi)$) is said to be of {\it top type} (resp.\ {\it bottom type}). The {\it winner} of an arrow of top (resp.\ bottom) type starting at $\pi=(\pi_t,\pi_b)$ with $\pi_t(\alpha_t)=
\pi_b(\alpha_b)=d$ is the letter $\alpha_t$ (resp.\ $\alpha_b$) while the {\it loser} is $\alpha_b$ (resp.\ $\alpha_t$). 

\smallskip
To an arrow $\gamma$ of a Rauzy diagram $\mathcal D$ starting at $\pi$ of top (resp.\ bottom) type, is associated
the matrix $B_\gamma\in SL(\Zset^{\A})$ defined by $$B_\gamma = \Iset + E_{\alpha_b\alpha_t} \qquad 
\text{(resp.\ } B_\gamma = \Iset +E_{\alpha_t\alpha_b}),$$
 where $E_{\alpha\beta}$ is the elementary matrix whose only nonzero coefficient is $1$ in position $\alpha\beta$. 

A path $\underline{\gamma}$ in a Rauzy diagram is {\it complete} if each letter in $\A$ is the winner of at least one arrow in $\underline{\gamma}$; it is $k$--{\it complete} if $\underline{\gamma}$ is the concatenation of $k$ complete paths. An infinite path is $\infty$--{\it complete} if it is the concatenation of infinitely many complete paths.


For a path ${\gamma}$ in $\mathcal D$ consisting of successive arrows $\gamma_1, \ldots, \gamma_n$, we set
$$ B_{\gamma} = B_{\gamma_n} \ldots B_{\gamma_1}\, . $$
It belongs to $SL(\Zset^{\A})$ and has nonnegative coefficients.
Writing $\pi$ for the origin of $\gamma$ and $\pi'$ for the endpoint of $\gamma$, one has 
\begin{equation}\label{omegarel} B_{\gamma}\; \Omega_{\pi} \;^t \negthinspace B_{\gamma} = \Omega_{\pi'}.
\end{equation}
Notice that from this relation it follows that 
$$ B_{\gamma}\, {\rm Im}\, \Omega_{\pi} = {\rm Im} \,\Omega_{\pi'}, \quad ^t B_{\gamma}^{-1} \ker  \Omega_{\pi} = \ker \Omega_{\pi'}.$$
Setting, for $v,w \in Im \Omega_{\pi}$,
$$
< \Omega_{\pi} v, \Omega_{\pi} w >:= {}^t v \, \Omega_\pi \, w
$$
defines a symplectic structure on ${\rm Im} \, \Omega_{\pi}$, and similarly on ${\rm Im} \, \Omega_{\pi'}$. Thus, \eqref{omegarel} shows that the cocycle $B_\gamma$ is indeed symplectic relative to the (degenerate) 
symplectic structure induced by the matrices $\Omega_\pi$ and $\Omega_{\pi'}$. 

\subsection{Iterations of the Rauzy-Veech map and the Zorich algorithm}\label{sec:algorithms}

Let $T=T^{(0)}$ be an i.e.m.\ with no connection. We
denote by $\A$ the alphabet for the combinatorial data $\pi^{(0)}$ of $T^{(0)}$ and by ${\mathcal D}$ the Rauzy diagram on $\A$ having $\pi^{(0)}$ as a vertex.

\smallskip
The i.e.m.\ $T^{(1)}$, with combinatorial data $\pi^{(1)}$, deduced from $T^{(0)}$ by the elementary step of the
Rauzy--Veech algorithm has also no connection. It is therefore possible to iterate this elementary step indefinitely
and get a sequence $T^{(n)}$ of i.e.m.\, with combinatorial data $\pi^{(n)}$, acting on a decreasing sequence $I^{(n)}$ of
intervals. Remark that for $m,n \in \Zset$ with $n \leq m$, $I^{(m)}$ is a subinterval of $I^{(n)}$ with the same left endpoint, and $T^{(m)}$ is the first return map into $I^{(m)}$ under iteration of $T^{(n)}$.  

We also define a sequence $\gamma_n$, $n \in \mathbb{N}$,  
of arrows in ${\mathcal D}$ associated to the successive steps of the algorithm, so that $\gamma_n$ is the arrow  from $\pi^{(n-1)}$ to $\pi^{(n)}$ corresponding to the $n^{th}$ step. For $n<m$, we also
write $\gamma (n,m)$ for the path from $\pi^{(n)}$ to $\pi^{(m)}$ composed by
\[
\gamma (n,m) =  \gamma_{n+1} \star  \cdots \star \gamma_{m-1} \star \gamma_{m} ,
\]
where $\star$ denotes the juxtapposition of paths. In particular $\gamma(n-1,n)$ coincides with the $n^{th}$ arrow $\gamma_n$. 

\smallskip
We write $\underline{\gamma}=\underline{\gamma}(T)$ for the infinite path $\gamma_1 \star \gamma_2  \star \cdots $ starting from $\pi^{(0)}$ formed by the $\gamma_n$, $n\geq 1$. It is $\infty$-complete (cf.~\cite{MMY1}, p.~832).
We call $\underline{\gamma}(T)$ the {\it rotation number} of $T$.
Conversely, an $\infty$-complete path is equal to $\underline{\gamma}(T)$ for some  i.e.m. with no connection.

\smallskip 

Given $T$ with no connection, let $\gamma(T)$ be its rotation number. 
For any intergers $m>n$, let

$$B(n,m):= B_{\gamma(n,m)} = B_{\gamma_m} B_{\gamma_{m-1}} \cdots B_{\gamma_{n+1}}$$
be the matrix associated to the path $\gamma (n,m)= \gamma_{n+1} \star \cdots \gamma_{m-1}\star \gamma_m$ from $\pi^{(n)}$ to $\pi^{(m)}$. 
In particular, for $n\geq 1$, $B(n-1,n)$ is the elementary matrix associated to the $n^{th}$ arrow $\gamma_n$ of $\gamma(T)$ connecting  $\pi^{(n-1)}$ to $\pi^{(n)}$.  For any $p\leq n \leq m$ the following cocycle relation holds:
\begin{equation}\label{matrix_cocycle_eq}
B(p,m) = B(n,m) B(p,n).
\end{equation}
We also denote by
\begin{equation}\label{columsum}
B_{\alpha}(n,m):=  \sum_{\beta \in \A} B(n,m)_{\alpha \beta }, \quad m>n, \quad \alpha, \beta \in \A.
\end{equation}

\smallskip
The entries of the matrix  $B(n,m)$, $m>n$, have the following dynamical interpretation. For any $\alpha, \beta \in \A$,   the entry $B(n,m)_{\alpha, \beta}$ gives the number of times  in the orbit $x,  T^{(n)}(x), \dots  $ of  any point $x\in I_{\alpha}^{(m)}$ under $T^{(n)}$ visits $I^{(n)}_\beta$ up to the first return time of $x$ to $I^{(m)}$.  Hence, $B_{\alpha}(n,m)$ gives the first return time to $I^{(m)}$ of any point $x\in I_{\alpha}^{(m)}$ under $T^{(n)}$.

\smallskip 
Following Zorich \cite{Zo2}, it is often convenient to group together in a single
Zorich step successive elementary steps of the Rauzy--Veech algorithm whose 
corresponding arrows have the same type (or equivalently the same winner); 
we therefore introduce a sequence $0=n_0<n_1<\ldots $ such that for each $k$ all 
arrows in $\gamma (n_k,n_{k+1})$ have the same type and this type is alternatively
top and bottom. For $n\geq 0$, the integer $k$ such that $n_k\leq n<n_{k+1}$ is called
the {\it Zorich time} and denoted by $Z(n)$.

\subsection{Dynamics of the continued fraction algorithms}\label{sec:RVdynamics}

Let $\mathcal R$ be a Rauzy class on an alphabet $\A$. The elementary step of the Rauzy--Veech algorithm, 
$$
(\pi, \lambda) \mapsto (\hat\pi, \hat\lambda )\, , 
$$
considered up to rescaling, defines a map from $\mathcal R\times \Pset ((\Rset^+)^{\A})$ to itself, denoted by $Q_{{\rm RV}} $. There exists a unique absolutely continuous measure ${m}_{{\rm RV}}$ invariant under these dynamics (\cite{Ve2}); it is conservative and ergodic but has infinite total mass, which does not allow all ergodic--theoretic machinery to apply. Replacing a Rauzy--Veech elementary step by a Zorich step gives a new map $Q_{{\rm Z}}$ on 
$\mathcal R\times \Pset ((\Rset^+)^\A)$. This map has now a {\it finite} absolutely continuous invariant measure ${m}_{{\rm Z}}$, which is ergodic (\cite{Zo2}).

\smallskip
It is also useful to consider the \emph{natural extensions} of the maps $Q_{{\rm RV}}$ and $Q_{{\rm Z}}$, defined through the suspension data which serve to construct translation surfaces from i.e.m.
 For $\pi\in\mathcal R$, let $\Theta_{\pi}$ be the convex open cone of suspension vectors  in $\Rset^{\A}$, see \eqref{def:suspdatum}.   
Define also 
$$
\Theta_\pi^t 
= \{ \tau\in\Theta_\pi\, , \, \sum_\alpha\tau_\alpha <0\}\; ,\qquad 
\Theta_\pi^b 
= \{ \tau\in\Theta_\pi\, , \, \sum_\alpha\tau_\alpha >0\}\; . 
$$

Let $\gamma\, : \, \pi\rightarrow\hat\pi$ be an arrow in the Rauzy diagram $\mathcal D$ associated to $\mathcal R$. Then $^tB_\gamma^{-1}$ sends $\Theta_\pi$ isomorphically onto $\Theta^t_{\hat\pi}$ (resp.\ $\Theta^b_{\hat\pi}$)
when $\gamma$ is of top type (resp.\ bottom type). 

The natural extension $\hat Q_{{\rm RV}}$ is the defined on 
$\sqcup_{\pi\in\mathcal R}\{\pi\}\times\Pset ((\Rset^+)^\A)\times\Pset (\Theta_\pi)$ by 
$$
(\pi,\lambda, \tau)\, \mapsto (\hat\pi, ^tB_\gamma^{-1}\lambda , ^tB_\gamma^{-1}\tau )
$$
where $\gamma$ is the arrow starting at $\pi$, associated to the map $Q_{{\rm RV}}$ at $(\pi,\lambda )$. 
The map $\hat Q_{{\rm RV}}$ has again a unique absolutely continuous invariant measure $\hat{m}_{{\rm RV}}$; it is ergodic, conservative but infinite. One defines similarly a natural extension $\hat Q_{{\rm Z}}$ for $ Q_{{\rm Z}}$; it has a unique absolutely continuous invariant measure $\hat{m}_{{\rm Z}}$, which is finite and ergodic.

\smallskip
 The sequences $(\pi^{(n)})_{n \in \Zset}$, $(\lambda^{(n)})_{n \in \Zset}$, $(\tau^{(n)})_{n \in \Zset}$ defined by the Rauzy-Veech algoritm satisfy, for $m \leq n$
$$ \lambda_\alpha^{(n)} = \sum_{\beta} B(m,n)^*_{\alpha\, \beta} \;\lambda^{(m)}_\beta, \qquad \tau_\alpha^{(n)} = \sum_{\beta} B(m,n)^*_{\alpha\, \beta} \;\tau^{(m)}_\beta,\qquad \text{where}\quad B^* = ^t\negmedspace B(m,n)^{-1}. $$
The associated sequences $\delta^{(n)}= \Omega_{\pi^{(n)}} \lambda^{(n)}$ and $q^{(n)} = \Omega_{\pi^{(n)}} \tau^{(n)}$, for $n \geq 0$, satisfy 
$$ \delta_\alpha^{(n)} = \sum_{\beta} B(m,n)_{\alpha\, \beta}\; \delta^{(m)}_\beta, \qquad  q_\alpha^{(n)} = \sum_{\beta} B_{\alpha\, \beta}.$$


\subsection{Special Birkhoff sums and the extended Kontsevich-Zorich cocycle}\label{ssKZ}

For a real number $r\ge 0$ we denote by $C^r(\sqcup I_\alpha)$   the product $\prod_{\alpha\in\A} C^r(\overline{I_\alpha})$. If $r=0$ we sometimes drop the 
exponent in $C^r(\sqcup I_\alpha)$.

Let $(T^{(n)})_{n \in \Zset}$ be a sequence of i.e.m. acting on intervals $I^{(n)}$ related by the Rauzy-Veech algorithm as in \S~\ref{ssRV}. Let $m \leq n$. 
For a function $\varphi \in C( \sqcup I_\alpha^{(m)})$, we define the {\it special Birkhoff sum} $S(m,n) \varphi \in C( \sqcup \iatn)$ by

\begin{equation}\label{specialBS_def}
  S(m,n) \varphi (x) = \sum_{0\leq j< r(x)} \varphi((T^{(m)})^j (x)), \quad \forall \alpha \in \A, \; \forall x \in \iatn, \end{equation}

where $r(x)$ is the return time of $x$ into $I^{(n)}$ under $T^{(m)}$. One has 
$$ r(x) = \sum_{\beta \in \A} B_{\alpha \, \beta} , \quad \forall \alpha \in \A, \; \forall x \in \iatn,$$
with $B = B(m,n)$.

 For $n \in \Zset$, $\Gamma(n) :=\Gamma(T^{(n)})$ is the $d$-dimensional subset of $C( \sqcup \iatn)$ consisting of functions which are constant on each $\iatn$. The space $\Gamma(T^{(n)})$ is canonically isomorphic to $\RA$. One can check that the operator $S(m,n)$ sends $\Gamma(T^{(m)})$ onto $\Gamma(T^{(n)})$, and the matrix of this restriction is $ B(m,n)$.

This defines a $d$-dimensional cocycle over $Q_{{\rm RV}}$ (and also above $\hat Q_{{\rm RV}}$, $ Q_{{\rm Z}}$ and $\hat Q_{{\rm Z}}$) called the \emph{(extended) Kontsevich-Zorich cocycle}.

\smallskip
Special Birkhoff sums operators  satisfy the following cocycle relation: for any  $n\leq n'\leq n''$, one has
$$S (n, n'') = S (n',n'') \circ S (n,n').$$
Observe also that, if $\psi$ is integrable on $I^{(n)}$, then $S (n,n') \psi$ is integrable on $I^{(n')}$ for any $n'\geq n$ and we have
$$\int_{I^{(n')}} S (n,n') \psi = \int_{I^{(n)}} \psi .$$


\subsection{The boundary operator}\label{ssboundary}

Let $\pi$ be irreducible combinatorial data over the alphabet $\A$. Define a $2d$-element set ${\mathcal S}$ by 
 $$ \mathcal S = \{ U_0 = V_0, U_1, V_1, \ldots, U_{d-1}, V_{d-1}, U_d = V_d \} .$$ 
These symbols correspond to the vertices of the polygon produced by the suspension of an i.e.m $T$ with  combinatorial data $\pi$ 
(cf.~\S~\ref{ssSuspension}). 

Going anticlockwise around the vertices (taking the gluing into account) produces a permutation\footnote{We remark that the presentation of the permutation $\sigma$ is different from \cite{MMY2}.}  $\sigma$ of ${\mathcal S}$:
  $$
 \begin{array}{lll}
 \sigma (U_i) = V_j & {\rm with }\; \pi_b^{-1}(j+1)= \pi_t^{-1}(i+1) & {\rm for}\; 0 \leq i < d \\
 \sigma (V_{j'}) = U_{i'} & {\rm with }\; \pi_t^{-1}(i')= \pi_b^{-1}(j') & {\rm for}\; 0 < j' \leq d.
 \end{array}
 $$ 
The cycles of $\sigma$ in $\mathcal S$ are canonically associated to the marked points on the translation surface $M_T$ obtained by suspension of $T$. We will denote by $\Sigma $ the set of cycles of $\sigma$, by $s$ the cardinality of $\Sigma$ (cf.~\S~\ref {ssSuspension}).


Let $\varphi \in C^0(\sqcup I_\alpha)$.  We write $\varphi(u_i^t-0)$ (resp. $\varphi (u_i^t +0)$) for its value at $u_i^t$ considered as a point in $[u_{i-1}^t,u_i^t]$ (resp.
 $[u_i^t, u_{i+1}^t]$). We will use the convention that $\varphi (u_0 -0) = \varphi (u_d +0) =0$.  
Let $T$ be an i.e.m. with combinatorial data $\pi$, and let $\Sigma$ be the set of cycles
of  the associated permutation $\sigma$ of ${\mathcal S}$. The {\it boundary operator} $\partial$ is the linear operator from $C^0(\sqcup I^t_\alpha)$ to $\RS$ defined by
\begin{equation}\label{eqbound}
(\partial \varphi)_C := \sum _{0 \leq i \leq d, \; U_i \in C} (\varphi (u_i^t - 0) - \varphi (u_i^t + 0)),
\end{equation}
for any $\varphi \in C^0(\sqcup I_\alpha)$, $C \in \Sigma$.

\begin{remark}
\label{remhomology}
The name boundary operator is due to the following homological interpretation. The space $\Gamma(T)\subset  C^0(\sqcup I_\alpha)$ is naturally isomorphic to the first relative homology group $H_1(M_T,\Sigma,\Rset)$ of the translation surface $M_T$: the characteristic function of $\iat(T)$ corresponds to the homology class $[\zeta_{\alpha}]$ of the side $\zeta_{\alpha}$ (oriented rightwards) of the polygon giving rise to $M_T$ (see \S~\ref{sec:transl}). Through this identification, the restriction of the boundary operator to $\Gamma (T)$ is the usual boundary operator
$$\partial: H_1(M_T,\Sigma,\Rset) \longrightarrow H_0(\Sigma,\Rset) \simeq \RS .$$
\end{remark}

\smallskip

We will denote by  $\Gamma_\partial (T)$ the kernel of the restriction of $\partial$ to $\Gamma (T)$. 
We note in the following proposition several useful properties of the cycles and of the boundary operator, for which we refer to \cite{MMY2} (and in particular to Proposition 3.2, p. 1597). 
\begin{proposition}\label{propboundary}
Let $T$ be an  i.e.m.\ with combinatorial data $\pi$: 
\begin{enumerate}
\item Let $\psi\in C^0([u_0,u_d])$. One has $\partial\psi = \partial (\psi\circ T)$.
\item If $\varphi\in C^0(\sqcup I_\alpha)$ satisfies $\varphi (u_i^t+0)=\varphi(u_i^t-0)$ for $0\le i\le d$ then $\partial\varphi =0$. 
\item The boundary operator $\partial\, :\, C^0(\sqcup I_\alpha) \to \RS$ is onto.
\item The restriction of $\partial$ to $\Gamma (T)$ has as kernel  $\Gamma_\partial (T)$ the image of 
$\Omega_{\pi}$.
\item The image of the restriction of $\partial$ to $\Gamma (T)$ is  
$\RSO:= \{x \in \RS,\; \sum_C x_C =0\}$.
\item Let $n\ge 0$, let $T^{(n)}$ be the i.e.m.\ obtained from $T$ by $n$ steps of the Rauzy--Veech algorithm. 
For $\varphi\in C^0(\sqcup I_\alpha )$ one has 
$$ \partial (S(0,n)\varphi )=\partial\varphi\, , $$
where the left--hand side boundary operator is defined using the combinatorial data $\pi^{(n)}$ of $T^{(n)}$.
\end{enumerate}
\end{proposition}

\subsection{Summary of notation}\label{secnot}
For convenience of the reader, in particular the reader who is already familiar with Rauzy-Veech induction, we summarize in this section the notation used in the rest of the paper.

\subsubsection*{Interval exchange maps}\label{ssRV}

\begin{itemize}
\item $\A$ is the alphabet used to index the intervals.
\item $d:= \# \mathcal A$.
\item $\pi = (\pi_t, \pi_b)$ are the combinatorial data of an i.e.m $T$, acting on an interval $I$.
\item $I_\alpha:= \iat=\iat(T)$, $\iab=\iab(T)$, for $\alpha \in \A$  are the subintervals in the top and bottom subdivisions of $I$.
\item $\lambda = (\lambda_\alpha)_{\alpha \in \A}$ are the length data, i.e.~$\lambda_\alpha=|\iat|=|\iab|$ for all $\alpha \in \A$.
\item $\delta = (\delta_{\alpha})_{\alpha \in \A}$ is the translation vector of $T$, i.e.~$T : I_\alpha^t \to I_\alpha^b $ is  given by $x \mapsto x + \delta_{\alpha}$.
\item $\Omega_{\pi}  = (\Omega_{\alpha,\beta})_{\alpha,\beta \in \A}$ is the antisymmetric matrix associated by the combinatorial data $\pi$, given by \eqref{Omega_def} (the length vector $\lambda$ and the translation vector $\delta$  are related by $\delta = \Omega \lambda$).
\item $u_0=u_0(T), u_1=u_1(T)$ endpoints of $I=I(T) = (u_0,u_d)$.
\item $u^t_1<\cdots<u^t_{d-1}$  singularities of $T$. We also write $u^t_i(T)=u_i^t(T)$, $1\leq i \leq d-1$.
\item  $u^b_1<\cdots<u^b_{d-1}$ singularities of $T^{-1}$. We also write $u^b_i(T)=u_i^b(T)$, $1\leq i \leq d-1$.
\item $T$ has \emph{no connection} if there is no $m$ nonnegative integer, such that $ T^m(u_j^b)=u_i^t$ for some $1\leq i,j <d$. 
\end{itemize}

\subsubsection*{Translation surfaces and suspension data}\label{ssSuspension}

\begin{itemize}
\item $\tau = (\tau_{\alpha})_{\alpha \in \A}$ is a \emph{suspension vector} for the combinatorial data $\pi$. 
\item $M(\pi,\lambda,\tau)$ translation surface obtained from the zippered rectangle construction with data $(\pi,\lambda, \tau)$. 
\item $M_T = M(\pi,\lambda,\tau^{can})$ canonical suspension  associated to the canonical suspension vector $\tau^{can}$ in \eqref{cansusp}. 
\item $\zeta_{\alpha} := \lambda_{\alpha} +i\tau_\alpha$, for $ \alpha \in \A$ are the periods of $ M(\pi,\lambda,\tau)$.
\item $\omega$ is the canonical holomorphic $1$-form on $M$.  
\item $g = g(\pi)$ is the genus of $M$; it is also the half-rank of $\Omega$.
\item $\Sigma$ is the set of marked points on $M$.
\item $s := \# \Sigma$ depends only on $\pi$. One has $d = 2g + s-1$.
\item $ \mathcal S := \{ U_0 = V_0, U_1, V_1, \ldots, U_{d-1}, V_{d-1}, U_d = V_d \}$ vertices of the suspension (singualarities of $M(\pi,\lambda,\tau)$).
\item $T$ is the return map to $I$ for the upwards vertical flow of $M$.
\item $q = (q_{\alpha})_{\alpha \in \A}$ given by $q = - \Omega \tau$ is the vector of heights of the rectangles in the zippered rectangle construction of $M(\pi,\lambda,\tau)$; $q_{\alpha}$ is the return time of $\iat$ to $I$ under the the vertical flow on  $M(\pi,\lambda,\tau)$.

\end{itemize}

\subsubsection*{The Rauzy-Veech renormalization algorithm}\label{ssec:RVren}

\begin{itemize}
\item $\mathcal R$ is the {\it Rauzy class} of the combinatorial data $\pi$ of $T$.
\item $\mathcal D$ is the associated {\it Rauzy diagram}. The set of vertices of 
$\mathcal D$ is $\mathcal R$. 
\item $\alpha_t$ is the element of $\A$ such that $\pi_t(\alpha_t) = d$; $\alpha_b$ is the element of $\A$ such that $\pi_b(\alpha_b) = d$.
\item Each arrow $\gamma$ in $\mathcal D$ has a {\it type  top} or {\it bottom}  and is associated to two elements of $\A$, called {\it winner} and {\it loser}; if $\gamma$ if of type {\it type  top} (resp.\ {\it bottom}), the winner is $\alpha_t$ and the loser $\alpha_b$ (resp.\ the winner is $\alpha_b$ and the loser $\alpha_t$).
\item $E_{\alpha\beta}$ is the elementary matrix whose only nonzero coefficient is $1$ in position $\alpha\beta$. 
\item $B_\gamma$, where  $\gamma$ is an arrow starting at $\pi$ of top (resp.\ bottom) type, is the matrix $B_\gamma = \Iset + E_{\alpha_b\alpha_t}$ (resp.\ $B_\gamma = \Iset +E_{\alpha_t\alpha_b})$).
\item $B_\gamma = B_{\gamma_n} \ldots B_{\gamma_1} $ if  $\gamma$ is a path $\gamma$ in $\mathcal D$ consisting of successive arrows $\gamma_1, \ldots, \gamma_n$.


\medskip
\item $T^{(n)}:= (\lambda^{(n)}, \pi^{(n)} )$  for $n \in \mathbb{N}$ is the sequence of (non-renormalized) i.e.m. obtained by applying the Rauzy-Veech algorithm starting from $T^{(0)}:= (\lambda^{(0)}, \pi^{(0)} )$ with no connection.
\item  $(\pi^{(n)},\lambda^{(n)},\tau^{(n)})$ for $n \in \mathbb{Z}$ is the sequence  
produced by the Rauzy-Veech induction from  $\left( \pi^{(0)}, \lambda^{(0)}, \tau^{(0)}\right):= (\pi,\lambda,\tau) $  
s.t. $M(\pi,\lambda,\tau)$ has {\it neither horizontal nor vertical connections}.
\item the interval $I^{(n)}$ is the domain of the i.e.m.~$T^{(n)}= \left( \pi^{(n)}, \lambda^{(n)}\right)$.
\item $I^{(n)}_\alpha = I_\alpha^t (T^{(n)})$, $\alpha \in \mathcal{A}$, are the subintervals exchanged by $T^{(n)}$.
\item  $\gamma_n = \gamma(n-1,n)$ for $n \in \mathbb{N}$ (or $n \in \Zset$)  is the  arrow  of $\mathcal D$ from $\pi^{(n-1)}$ to $\pi^{(n)}$. 
\item  $\gamma (n,m) $, for $n<m$, is the path in   $\mathcal R$ from $\pi^{(n)}$ to $\pi^{(m)}$ given by $ \gamma_{n+1} \star  \cdots \star \gamma_{m-1} \star \gamma_{m} $.
\item $\gamma(T)$ for $T=(\pi,\lambda)$ with no connection, is the \emph{rotation number} of $T$; $\gamma(T)$ is the juxtapposition $\gamma_1 \ast \gamma_2 \ast$ of the arrows $\gamma_1, \gamma_2, \dots $ 
\item $B(m,n)$, where $m<n$ is the matrix product $B(m,n):= B_\gamma(m,n) = B_{\gamma_{n}}B_{\gamma_{n-1}} \cdots B_{\gamma_{m+1}}$.
\item $B_\alpha (m,n): =  \sum_{\beta \in \A} B(m,n)_{\alpha \, \beta} $ is the return time of any $x\in I^{(n)}_\alpha$ into $I^{(n)}$ under $T^{(m)}$. 

\medskip
\item  $Q_{{\rm RV}} : \mathcal R\times \Pset ((\Rset^+)^{\A})  \to \mathcal R\times \Pset ((\Rset^+)^{\A}) $ \emph{Rauzy-Veech map} (rescaled step of Rauzy-Veech algorithm). 
\item  ${Q}_{{\rm Z}}: \mathcal R\times \Pset ((\Rset^+)^{\A}) \to \mathcal R\times \Pset ((\Rset^+)^{\A}) $ \emph{Zorich map}, acceleration of the Rauzy-Veech map; 
\item $Z(n)$ \emph{Zorich time}, i.e.~the integer  $k$ such that $n_k\leq n<n_{k+1}$ where $n_k$ are such that  $Q_{{\rm Z}}^k(T)= Q^{n_k}_{{\rm RV}}(T)$.
\item   ${m}_{{\rm RV}}$ (resp.\ ${m}_{{\rm Z}}$) unique absolutely continuous measure  invariant under  $Q_{{\rm RV}} $ (resp.\ ${Q}_{{\rm Z}}$).
\item  $\hat Q_{{\rm RV}}$  (resp.\ $\hat Q_{{\rm Z}}$), defined on $\mathcal R \times\Pset ((\Rset^+)^\A)\times\Pset (\Theta_\pi)$: \emph{natural extension} of $ Q_{{\rm RV}}$ (resp.\ $Q_{{\rm Z}}$).
\item $\hat{m}_Q$ (resp.\ $\hat{m}_Z$) invariant measures for $\hat Q_{{\rm RV}}$ (resp.\ $\hat Q_{{\rm Z}}$);  
\end{itemize}

\subsubsection*{Functional spaces, special Birkhoff sums  and boundary operator}\label{ssec:boundary}
\begin{itemize}
\item $C(\sqcup I_\alpha)$   space of functions on $\sqcup I_\alpha$ which belong to $C(\overline{I_\alpha})$ for every $\alpha \in \mathcal{A}$.
\item  $C^r(\sqcup I_\alpha)$ for $r\ge 1$  space of functions on $\sqcup I_\alpha$ which belong to  $C^r(\overline{I_\alpha})$ for every $\alpha \in \mathcal{A}$.
\item  $\Gamma(T)\subset C(\sqcup I_\alpha)$ $d$-dimensional space of functions which are constant on each $I_\alpha^{(n)}$.
\item Given $T$ with no connection,  $\Gamma(n) :=\Gamma(T^{(n)})$, where $T^{(n)} = \mathcal{R}^n(T)$, for $n \in \Zset$.

\medskip
\item  $S(m,n): C( \sqcup I_\alpha^{(n)}) \to C( \sqcup I_\alpha^{(m)})$ \emph{special Birkhoff sums} operator; for  $\varphi \in C( \sqcup I_\alpha^{(n)})$, for $\alpha \in \A$ and $ x \in I_\alpha^{(n)}$,  $S(m,n) \varphi (x) = \sum_{0\leq j< B(m,n)} \varphi((T^{(m)})^j (x))$; 
\item the restriction $S(m,n): \Gamma(m) \to \Gamma(n)$ is given by $B(m,n)= B_{\gamma(m,n)}$.

\medskip
\item $\sigma: \mathcal S \to \mathcal S$ permutation which identifies the vertices of $\mathcal S$ going clockwise around singularities.
\item  $\Sigma $ set of cardinality $s$ idenfified with  \emph{cycles} of $\sigma$.
\item  $\varphi(u_i^t\pm 0)$   value of $\varphi \in C^0(\sqcup I_\alpha)$ at $u_i^t$ considered as a point in $[u_{i}^t,u_{i+1}^t]$ or $[u_{i-1}^t,u_i^t]$ respectively.
\item  $\partial: C^0(\sqcup I_\alpha) \to \RS$ {\it boundary operator},    defined for any $\varphi \in C^0(\sqcup I_\alpha)$ and $C \in \Sigma$ by $(\partial \varphi)_C := \sum _{0 \leq i \leq d, \; U_i \in C} (\varphi (u_i^t - 0) - \varphi (u_i^t + 0))$.
\end{itemize}


\section{Absolute Roth type and applications}\label{sec:absoluteRoth}
In this section we define the absolute Roth type condition for i.e.m.\ We first recall the definition of (restricted) Roth type i.e.m. which was give in \cite{MMY1, MMY2}. 
The adjective \emph{absolute} refers here to \emph{absolute homology}, since as we will see this condition is expressed in terms of cones which are isomorphic to the absolute homology of the suspended surface; in contrast, the usual Roth type condition concerns \emph{relative} homology (we could have called it in this paper \emph{relative Roth type condition}, but we chose not to do so for consistency with the literature). We then state the extension of the main results from  \cite{MY} on  the cohomological equation to i.e.m. of (restricted) absolute Roth type (see Theoreom~\ref{thmHolder1} below, as well as  Theorem~\ref{thmcohom2} and Theorem~\ref{thmHolder2} in the Appendix~\ref{app:proofs}).  In \S~\ref{sec:crucialestimates} we prove two crucial estimates  which are needed to adapt the corresponding proofs to deal with absolute Roth type i.e.m. 

\subsection{Roth Type i.e.m.}\label{ssRoth}
We first recall the Roth type diophantine condition on an i.e.m.\ introduced in \cite{MMY1} and then slightly modified in \cite{MMY2}.

\smallskip
For a matrix $B = (B_{\alpha \beta})_{\alpha,\beta \in \A}$, we define 
\begin{equation}\label{def:norm} ||B || := \max_{\alpha} \sum_{\beta} | B_{\alpha \beta}|,\end{equation}
which is the operator norm for the $\ell_{\infty}$ norm on $\RA$.

Let $T$ be an i.e.m, with combinatorial data $\pi$.  In the space $\Gamma(T) \simeq \RA$ of piecewise constant functions on $I(T)$, we define the hyperplane
$$ \Gamma_0(T) := \{ \chi \in \Gamma(T) ,\; \int \chi(x) \; dx = 0\}.$$
Assuming that  $T$ has no connection, we also define the \emph{stable subspace}
$$ \Gamma_s(T) := \{ \chi \in \Gamma(T),\quad  \exists \,\sigma >0,\; ||B(0,n)\chi|| = \mathcal {O} (||B(0,n)||^{-\sigma}) \}.$$

As $\Gamma_s(T)$ is finite-dimensional, there exists an exponent $\sigma >0$ which works for every $\chi\in\Gamma_s(T)$. We fix such 
an exponent in the rest of the paper. 
The subspace $ \Gamma_s(T)$ is contained in ${\rm Im}\,\Omega (\pi )$, and is an isotropic subspace of this symplectic space.
We have 
$$\Gamma_s(n):= \Gamma_s(T^{(n)}) = B(0,n)\Gamma_s(T), \quad \text{for}\ n \geq 0.$$

\smallskip

Let us first define a suitable \emph{acceleration} of the Rauzy-Veech algorithm. 
 Consider a {\it rotation number} $\ug = \gamma_0 \star \gamma_1 \star \ldots $, i.e an infinite $\infty$-complete path in $\D$ obtained concatenating successive arrows $\gamma_0, \gamma_1, \ldots$. Define $\hat n_0 =0$. Define inductively $\hat n_k$ for $k\geq1$ as the smallest integer $n> \hat n_{k-1}$ such that the path 
$ \gamma_{\hat n_{k-1}+1} \star \ldots \star \gamma_{\hat n}$ is complete (see \ref{secRV}).

\smallskip

Following \cite{MY}, \emph{Roth type i.e.m.} are those i.e.m.\ whose rotation number satisfy three conditions (a) (\emph{matrix growth}), (b) (\emph{spectral gap}), (c) (\emph{coherence}). Furthermore, i.e.m. which in addition satisfy also a fourth condition (d) (\emph{hyperbolicity}), are called of  \emph{restricted Roth type}. 
  We  now recall the definition of these four conditions. The sequence  $(\hat n_k)$ that appears in Condition (a) is the sequence of accelerated times  defined 
just above. 


\smallskip
\begin{definition}\label{defRoth2}
An i.e.m. $T$ with no connection is of  \emph{restricted Roth type} if the following four
conditions are satisfied: 


\begin{itemize}
\item[(a)]
 ({\it matrix growth})
$$ \forall \tau >0, \; \exists \, C_{\tau} >0, \; \forall k> 0, \; \Vert B(\hat n_{k-1},\hat n_k) \Vert \leq C_{\tau} \Vert B(\hat n_0, \hat n_{k-1}) \Vert^{\tau}.$$


\item[(b)] ({\it spectral gap}) 
There exists $\theta >0$ such that
$$||B(0,n)_{|\Gamma_0(T)}|| = \mathcal {O} (||B(0,n)||^{1-\theta})\,.$$

\item[(c)] 
 ({\it coherence})  \hspace{3mm}  For $0 \leq m\leq n$, let 
\begin{eqnarray*}
B_s(m,n)&:& \Gamma_s(T^{(m)}) \to \Gamma_s(T^{(n)}) ,\\
  B_{\flat}(m,n)&: &\Gamma(T^{(m)}) /\Gamma_s(T^{(m)}) \to \Gamma(T(n)) /\Gamma_s(T^{(n)})  
  \end{eqnarray*}

be the operators  induced by $B(m,n)$.
We ask that, for all $\tau >0$,
$$||B_s(m,n)||= \mathcal {O} (||B(0,n)||^{\tau}), \quad ||(B_{\flat}(m,n))^{-1}|| = \mathcal {O}
(||B(0, n)||^{\tau})\,.$$

\item[(d)] ({\it hyperbolicity}) 
$\dim \Gamma_s (T) =g$.
\end{itemize}

The i.e.m. with no connection that satisfy conditions (a), (b) and (c) are called \emph{Roth type} i.e.m. Finally, we will say that i.e.m. with no connection that satisfy conditions (a) and (b) only are i.e.m. of \emph{weak} Roth type. We will sometimes in this paper refer to these conditions as \emph{direct}  Roth type conditions to distinguish it from the \emph{dual} Roth type condition which we will introduce later.

\end{definition}
In \cite{MMY1}, condition (a) had a slightly different (but equivalent) formulation, as explained in the remark below.
 


\begin{remark}\label{remRoth0} 
Condition (a) can be rephrased in the following way. Define $\wt n_0 =0$. 
Define inductively $\wt n_k$ for $k>0$ as the smallest integer $n> \wt n_{k-1}$ such that the matrix $B(\wt n_{k-1},n)$ has positive coefficients. Then condition (a) is equivalent to 
\begin{equation}\label{a_paper} 
 \forall \tau >0, \; \exists \, C_{\tau} >0, \; \forall k> 0, \; \Vert B(\wt n_{k-1},\wt n_k) \Vert \leq C_{\tau} \Vert B(\wt n_0, \wt n_{k-1}) \Vert^{\tau}.
\end{equation}
Indeed, let $\gamma$ be a finite path in $\D$. If $\gamma$ is not complete, at least one  of the coefficients of $B_{\gamma}$ is equal to $0$. Conversely, if $\gamma$ is 
$(2d-3)$-complete (or, when $d=2$, $2$-complete), all coefficients of $B_{\gamma}$  are positive, see \cite{MMY1} (more precisely, see the Lemma in $\S$~1.2.2, page 833). These two facts imply the equivalence of the two formulations, namely of condition (a) and  condition  \eqref{a_paper}. It is the condition above which was used in \cite{MY}.
\end{remark}
Let us remark that an i.e.m.\ $T$ with no connection satisfying condition (b)
  is uniquely ergodic. 
Let us also recall that for any combinatorial data, the set of i.e.m of relative restricted Roth type has \emph{full measure}.
First, it is obvious that almost all i.e.m. have no connection. That condition (c) is almost surely satisfied is a consequence of Oseledets theorem applied to the Kontsevich-Zorich cocycle. Condition (d) follows from the hyperbolicity of the (restricted) Kontsevich-Zorich cocycle, proved by Forni {\Green (\cite {For2})}. 
A proof that condition (a) has full measure is provided in \cite{MMY1}, but much better diophantine
estimates were later obtained in \cite{AGY}. A simpler proof where full measure is deduced from the estimates in \cite{AGY} is sketched in \cite{MMM}.
Finally, the fact that condition (b) has full measure is a  consequence from the fact that the
larger Lyapunov exponent of the Kontsevich-Zorich cocycle is simple (see \cite{Ve4}).



\subsection{The cones $\C(\pi)$}

We give in this section some preliminary definitions which we need in order to define  \emph{absolute} Roth type i.e.m.

\smallskip
We denote by $\C$ the positive open cone of $\RA$. Remark first that for $m<n$, the matrix $B(m,n)$ has positive coefficients iff
$$ \overline{B(m,n)(\C)} \subset \C \cup \{ 0 \}.$$ 
For $\pi \in \R$, 
let $\Omega_{\pi}$ be the antisymmetric matrix associated to $\pi$ (see \S~\ref{sec:iem}) and let $H(\pi) \subset \RA$ the image of $\Omega_{\pi}$. The dimension of $H(\pi)$ is $2g$ and the codimension of $H(\pi)$ is $s-1$, where $g$ is the genus of the translation surfaces associated to $\R$ and $s$ is the number of marked points (see \S~\ref{sec:transl}). We are interested in the case $s>1$ where $H(\pi) \subsetneqq \RA$. 

We denote by  $\C(\pi)$ the intersection
$\C \cap H(\pi)$.  Let us recall that, from the properties of the extended KZ-cocycle, 
for any oriented path $\gamma: \pi \to \pi'$ in $\D$, one has 
\begin{equation}\label{p1} 
B_{\gamma} (H(\pi)) = H(\pi'),\qquad B_{\gamma} (\C(\pi)) \subset \C(\pi').
\end{equation}

\smallskip
Let us explain the connection between $H(\pi)$ and absolute homology. 
Let $M$ be a translation surface with combinatorial data $\pi$, namely a surface of the form $M=M(\pi, \lambda, \tau)$ where $\lambda \in {(\mathbb{R}^+)}^\A$ and $\tau \in \Theta_\pi$. 
Denote by  $\Sigma$ the set of marked points of $M$. 
Recall from \S~\ref{sec:transl} that the homology classes $[\zeta_\alpha]$, $\alpha \in \A$ of the vectors $\zeta_\alpha$  defined in \eqref{def:zeta} provide a base of relative homology $H_1(M,\Sigma,\Zset)$ (see Remark~\ref{relativeHidentification}).  
When we identify $\RA$ to $H_1(M,\Sigma,\Rset)$ as explained in Remark~\ref{relativeHidentification}, 
the subspace $H(\pi)$ is identified with the absolute homology group $ H_1( M,\Rset) \subset H_1(M,\Sigma,\Rset) $, see for example \cite{Y2}. This homological interpretation will be discussed further in Section \ref{sec:homological}.


\smallskip
The following result is well known, we recall its proof for convenience.


\begin{proposition}\label{p2}
For any $\pi \in \R$, the open cone $C(\pi) \subset H(\pi)$ is not empty.
\end{proposition}
\begin{proof}
Recall that an element $\pi = (\pi_t,\pi_b)$ is {\it standard} if there exist letters $\alpha_t, \alpha_b \in \A$ such that $\pi_t(\alpha_t) = \pi_b(\alpha_b) =d$ and 
$\pi_t(\alpha_b) = \pi_b(\alpha_t)=1$. Every Rauzy class contains at least one standard vertex. 
By \eqref{p1}, it is sufficient to prove that $C(\pi)$ is not empty when $\pi$ is standard. 


Consider the (absolute) homology class $\zeta \in H_1( M,\Zset) $ associated to the loop obtained by concatenation of paths in the polygon joining $U_0$ to $U_d$
and $U_1$ to $U_{d-1}$ (this is a loop because $\pi$ is standard). One has
$$
\zeta = 2 \sum_{\alpha \in \A} \zeta_{\alpha} - \zeta_{\alpha_t} - \zeta_{\alpha_b} =: \sum x_\alpha \zeta_\alpha .
$$
As $\zeta$ is an absolute homology class, one has $x \in H(\pi)$. Obviously, $x$ belongs also to $\C$.
\end{proof}

For later reference, we also state the 
\begin{proposition}\label{p4}
There exists a constant $C = C(\R) >1$ such that, for any path $\gamma: \pi \to \pi'$ in $\D$, one has, for $B:= B_{\gamma}$
$$ \Vert B_{\vert H(\pi)} \Vert \leq \Vert B \Vert \leq C \Vert B_{\vert H(\pi)} \Vert .$$
\end{proposition}
\begin{proof}
Recall that we choose as norm the operator norm for the $\ell^{\infty}$-norm on $ \RA$, see \eqref{def:norm}. The first inequality is trivial. The second is a consequence of Proposition \ref{p2}.
Indeed, fix a vector $v$ in $\C(\pi)$; one has $ \Vert B \Vert = \max_\alpha \sum_\beta B_{\alpha\,\beta}$. Hence, writing $w = Bv$,
\[
\min_\beta v_\beta\ \Vert B \Vert \leq 
\max_\alpha \sum_\beta B_{\alpha\,\beta} \; v_\beta 
= \max_\alpha w_\alpha \leq    \Vert B_{\vert H(\pi)} \Vert  \max_\beta v_\beta,
\]
which gives the required estimate with $C = \max_\pi (\frac {\max_\beta v_\beta}{\min_\beta v_\beta})$. 

\end{proof}

\subsection{Absolute Roth type i.e.m.}\label{ssabs}
To define the (restricted) \emph{absolute} Roth type Diophantine condition for  i.e.m. we will modify the definition of restricted Roth type recalled in the previous section. 
We will not change the last three conditions (b), (c) and (d),  but replace Condition (a) by a weaker Condition (a)'. 

Let $\underline{\gamma}$ be the rotation number of an i.e.m. with no connection and let $\pi^{(n)}$, $n\geq 0$ be the vertices of the paths and $B(m,n)$, $n>m$, the associated Rauzy-Veech matrices (see \S~\ref{secRV}). 
To introduce condition (a)', we set $n_0 =0$ and define $n_k$ for $k>0$ as the 
 smallest integer $n>  n_{k-1}$ such that the matrix $B( n_{k-1},n )$ satisfies

$$  \overline{B(n_{k-1},n)(\C(\pi^{(n_{k-1})}))} \subset \C(\pi^{(n)}) \cup \{ 0 \}.$$ 

Then we ask that
 $$ {(\rm a)'} \qquad \forall \tau >0, \; \exists \, C_{\tau} >0, \; \forall k> 0, \; \Vert B( n_{k-1}, n_k) \Vert \leq C_{\tau} \Vert B( n_0,  n_{k-1}) \Vert^{\tau}.$$

\begin{definition}\label{d1} 
An i.e.m with no connection is of  {\it absolute Roth type} if its rotation number satisfies conditions (a)', (b), (c), (d).
\end{definition}

The following Lemma shows that the absolute Roth type condition is weaker than the Roth type condition. 
\begin{lemma}\label{lem:Roth_implies_absRoth}
A rotation number which satisfies condition {\rm (a)} also satisfies condition {\rm (a)'}. Thus, (restricted) Roth type i.e.m. are  (restricted) absolute Roth type i.e.m.
\end{lemma}

 \begin{proof}
 Observe first that, if $B(m,n)$, where $0\leq m < n$, has positive coefficients, we have
\begin{equation}\label{nested_inclusion}
  \overline{B(m,n)(\C(\pi^{(m)}))} \subset \C(\pi^{(n)}) \cup \{ 0 \}.
\end{equation}
Let $(n_k)_{k\geq 0}$ be the subsequence of times   used to defined absolute Roth type condition $(\rm a)'$ and let $(\wt n_k)_{k\geq 0}$ be the subsequence of times defined in Remark~\ref{remRoth0} so that the matrix $B(\wt n_{k-1},\wt n_k)$ has positive coefficients. Fix some $k\geq 0$. Let $j:= j(k)$ be the smallest $l$ such that $\wt n_l\geq n_k$. We claim that we have the inequalities
$$
\wt n_{j-1}< n_k\leq  \wt n_{j} , \qquad    n_{k+1} \leq \wt n_{j+1} .
$$
The first set of inequalities follow from the definition of $j=j(k)$; to prove the last inequality, notice that, since $B( \wt n_{j}, \wt n_{j+1})>0$, also $B( n_k, \wt n_{j+1})>0$
and thus by \eqref{nested_inclusion} we have that
$$
\overline{B( n_{k},\wt n_{j+1} )(\C(\pi^{(n_{k})}))} \subset \C(\pi^{(\wt n_{j+1})}) \cup \{ 0 \}.
$$
This, by definition of the sequence $(n_{k})_{k\geq 0}$, implies that $n_{k+1}\leq \wt n_{j+1}$. 

Recall that by Remark~\ref{remRoth0}, condition  $(\rm a)$ is equivalent to  asking that for all $ \tau >0$,  there exists $C_{\tau} >0$ such that $\Vert B(\wt n_{k-1},\wt n_k) \Vert \leq C_{\tau} \Vert B(\wt n_0, \wt n_{k-1}) \Vert^{\tau}$ for all positive integers $k$. Thus, using the inequalities among indexes and the cocycle identity, we get
\begin{align*}
 \Vert B( n_{k}, n_{k+1}) \Vert  \leq   \Vert B( \wt n_{j-1},\wt n_{j+1}) \Vert   &  
 \leq C^2_{\tau} \Vert B(0, \wt n_{j-1}) \Vert^{\tau} \Vert B(0, \wt n_{j}) \Vert^{\tau} 
\\
 & \leq C^2_{\tau} \Vert B(0, \wt n_{j-1}) \Vert^\tau \Vert B(0,  n_{k}) \Vert^{\tau} \Vert B(n_k, \wt n_{j}) \Vert^{\tau} 
  \\ & \leq C^2_{\tau}  \Vert B( 0,  n_{k}) \Vert^{2 \tau} \Vert B(\wt n_{j-1}, \wt n_{j}) \Vert^{\tau}  
\leq C'  \Vert B( 0,  n_{k}) \Vert^{2 \tau + \tau^2}  . 
\end{align*}
for some $C'=C'(\tau)$. Since $\tau>0$ was arbitrary, this shows that $(\rm a)'$ holds in this case. 
 \end{proof}
\begin{remark} On the other hand, the converse of Lemma \ref{lem:Roth_implies_absRoth} is in general not true  (for $s>1$), namely there exists i.e.m which are of absolute Roth type but \emph{not} of Roth type. For instance when $d=3$, one can construct i.e.m. which are obtained from  a Roth type rotation number by marking a point and there is a set of measure zero but full Hausforff dimension of marked points for which the $3$ i.e.m. fails to be Roth type.
\end{remark}

\smallskip

Finally, let us remark that since (restricted) Roth type i.e.m. have full measure, from Lemma \ref{lem:Roth_implies_absRoth} it also follows that (restricted) absolute Roth type i.e.m. also have full measure. 

\subsection{Two crucial estimates}\label{sec:crucialestimates}
We state and prove in this section  two crucial estimates on the lengths of induced {\Green subintervals} (see Corollary \ref{c1} and Proposition \ref{p6})  which  are needed in particular to prove results on the cohomological equation, but are also of independent interest (see for example {\cite{Kim,KM1}}).  
These estimates were proved for i.e.m. of  Roth type in \cite{MMY1}.
They require a different proof for i.e.m. of \emph{absolute} Roth type, which is given in this section. Once these estimates are proved, one can easily generalize all the main results on the cohomological equation proved in  \cite{MMY1, MY} for (restricted) Roth type i.e.m.  to (restricted) absolute Roth type i.e.m. (see \S~\ref{sec:results} and Appendix \ref{app:proofs}). 
  We present the proof of these estimates this section, since they show how the absolute Roth type condition is  used.


\smallskip
 Let $T$ be an i.e.m with combinatorial data in $\R$ and no connection. Let $\ug$ be the rotation number of $T$. Let $(T^{(n)})_{n \geq 0} $ be the sequence of i.e.m obtained from $T=: T^{(0)}$
by the Rauzy-Veech algorithm. Recall that, for $n\geq 1 $,  $T^{(n)}$ acts on $I^{(n)}\subset I^{(n-1)}$ exchanging the {\Green subintervals} $I_{\alpha}^{(n)}:= I_{\alpha}^t( T^{(n)})$, $\alpha \in \mathcal{A}$. Recall from \cite{MMY1} (see also Lemma 3.14 of \cite{MY}) that condition (a) in the Roth type definition implies the following estimates on  the lengths of induced {\Green subintervals}.

\begin{proposition}\label{propo1}
Assume that $\ug$ satisfies condition {\rm (a)}. Then, for any $\tau >0$, there exists $C_{\tau} >0$ such that, for any $\alpha \in \A$, any $n \geq 0$, one has
$$
C_{\tau}^{-1} \ \vert I^{(0)} \vert \ \Vert B(0,n) \Vert^{-1-\tau} \leq \vert I_{\alpha}^{(n)} \vert \leq C_{\tau} \ \vert I^{(0)}  \vert \ \Vert B(0,n) \Vert^{-1+\tau}.
$$
\end{proposition}

{
\begin{remark}\label{maxmin} One can also see (cf.~\cite{MMY2} Appendix C, Proposition C1, p. 1639) that condition {\rm (a)} of Roth type is equivalent to the following: for all $\tau>0$, there exists $C>0$ such that for all $n \in \mathbb{N}$, we have
$$
\max_{\alpha \in \A}  \vert I_{\alpha}^{(n)} \vert\leq C \left( \min_{\alpha \in \A}  \vert I_{\alpha}^{(n)}\right)^{1-\tau}.
$$ 
\end{remark}
}

We will find analogues  estimates to those provided by Prop.~\ref{propo1} when $\ug$ satisfies only the weaker condition {\rm (a')}. We deal first with the upper bound, which is the easiest one (and the same than in proposition \ref{propo1}). The upper bound will follow from the following Proposition (see Corollary \ref{c1} below). 

\begin{proposition}\label{p3}
Assume that $\ug$ satisfies condition (a'). Then, for any $\tau >0$, there exists $C_{\tau} >0$ such that, for any $\alpha \in \A$, any $n \geq 0$, one has
$$ \sum_{\beta \in \A} B_{\alpha \, \beta}(0,n) \geq C_{\tau}^{-1} \Vert B(0,n) \Vert^{1-\tau}.$$
\end{proposition}

Let us state and prove the corollary giving the upper bound  before giving the proof of Propostion \ref{p3}.
\begin{corollary}\label{c1}
For any $\tau >0$, there exists $C_{\tau} >0$ such that, for any $\alpha \in \A$, any $n \geq 0$, one has
$$ \vert I_{\alpha}^{(n)}  \vert \leq C_{\tau} \ \vert I^{(0)}  \vert \ \Vert B(0,n) \Vert^{-1 +\tau}.$$
\end{corollary} 

\begin{proof} The upper bound follows immediately from  Proposition \ref{p3} combined with the simple observation that 
\[ \sum_{\alpha,\beta \in \A}  B_{\alpha \, \beta}(0,n) \  \vert I_{\alpha}^{(n)}  \vert  =  \vert I^{(0)} \vert . \]
\end{proof}
For the proof of Proposition \ref{p3}, we need the following Lemma.  
\begin{lemma}\label{l1}
There exists a constant $C = C(\R)$ with the following property. Let $\gamma: \pi \to \pi'$ be a finite path in $\D$ such that 
$$ \overline {B_{\gamma} (\C(\pi))} \subset \C(\pi') \cup \{0\}.$$
For any  vector $w \in \overline {B_{\gamma} (\C(\pi))}$, one has
$$ \max_\alpha w_\alpha \leq C \Vert B_\gamma \Vert \min_\alpha w_\alpha.$$
\end{lemma}

\begin{proof}
Fix $\pi \in \R$. As the subspace $H(\pi)$ is rational, one can find primitive integral vectors $v^i \in H(\pi)$ which span the extremal rays of $\C(\pi)$. Let $C(\pi)$ be the supremum over $i$ of the $ \Vert v^i \Vert$. Let $\gamma: \pi \to \pi'$ be as in the lemma, and let $w^i := B_\gamma v^i$. Each $w^i$ is an integral vector with positive coordinates. We have therefore 
$$ \max_\alpha w_\alpha^i \leq \Vert B_\gamma \Vert \ \Vert v^i \Vert  \leq C(\pi)  \Vert B_\gamma \Vert \leq C(\pi)  \Vert B_\gamma \Vert  \min_\alpha w_\alpha^i.$$

For  $v \in \overline{\C(\pi)} $, $w = B_\gamma v$, we write $v = \sum_i t_i v^i$ with 
$t_i \geq 0$. For $\alpha,\beta \in \A$, we obtain
$$ w_\alpha = \sum_i t_i w_\alpha^i \leq C(\pi) \Vert B_\gamma \Vert \sum_i t_i w_\beta^i = C(\pi) \Vert B_\gamma \Vert \  w_\beta.
$$
Therefore the conclusion of the lemma holds with $C(\R):= \max_{\pi \in \R} C(\pi)$.                
\end{proof}

\begin{proof}[Proof of Proposition \ref{p3}]
We first prove the inequality of the proposition when $n = n_k$. Fix $u \in \C(\pi^{(n_0)})$; let 
$$v = B(n_0,n_{k-1}) u,\quad  w = B(n_0,n_k) u = B(n_{k-1},n_k) v.$$ 
For $\alpha, \alpha' \in \A$, using Lemma~\ref{l1}, condition {\rm (a)'} and that $\Vert B(0,n) \Vert$ is non decreasing in $n$, one has that
\begin{eqnarray*}
\sum_{\beta \in \A}  B_{\alpha \, \beta}(0,n_k) &\leq&  (\min_\beta u_\beta)^{-1} \sum_{\beta \in \A}  B_{\alpha \, \beta}(0,n_k) u_\beta =  (\min_\beta u_\beta)^{-1}  w_\alpha \\
&\leq& C \Vert B(n_{k-1},n_k) \Vert (\min_\beta u_\beta)^{-1}  w_{\alpha'} \\
& \leq& C  C'_{\tau} \Vert B(0, n_{k-1}) \Vert ^{\tau}  (\min_\beta u_\beta)^{-1} \sum_{\beta \in \A}  B_{\alpha' \, \beta}(0,n_k) u_\beta \\
& \leq& C  C'_{\tau} \Vert B(0, n_{k}) \Vert ^{\tau} \ \frac { \max_\beta u_\beta}{ \min_\beta u_\beta} \sum_{\beta \in \A}  B_{\alpha' \, \beta}(0,n_k).
\end{eqnarray*}
As $\sum_{\alpha,\beta} B_{\alpha \, \beta}(0,n_k)$ is of the order of $\Vert B(0,n_k)  \Vert $, we obtain the inequality of the proposition for $n = n_k$.

\medskip

When $n_k < n < n_{k+1}$, we use that each $B_{\alpha,\beta}(0,n)$ is a non-decreasing function of $n$. For $\alpha \in \A$, we have

$$
\sum_{\beta \in \A}  B_{\alpha \, \beta}(0,n) \geq \sum_{\beta \in \A}  B_{\alpha \, \beta}(0,n_k) \geq C_{\tau}^{-1} \Vert B(0,n_k) \Vert^{1-\tau}.
$$ 
On the other hand, we have
$$ \Vert B(0,n) \Vert \leq \Vert B(0, n_k) \Vert \Vert B(n_k,n_{k+1}) \Vert \leq \wt C_\tau
\Vert B(0, n_k) \Vert^{1 + \tau}.
$$
As $\tau>0$ is arbitrary, we obtain the estimate of the proposition. 
\end{proof}

\bigskip

The estimate of $\vert I_\alpha^{(n)} \vert $ from below, as stated in Proposition \ref{propo1}, is \emph{not} a consequence of condition {$(\mathrm{a}')$} (even for $d=3$) . The   true statement is slightly more sophisticated:

\begin{proposition}\label{p6}
Assume that $\ug$ satisfies condition $(\mathrm{a}')$. Then, for any $\tau >0$, there exists $C_{\tau} >0$ such that, for any $\alpha \in \A$, any $\ell \geq 0$, one has
$$
\vert I_{\alpha}^{(n_\ell)} \vert \geq C^{-1}_{\tau} \ \vert I^{(0)} \vert \ \Vert B(0,n_k) \Vert^{-\tau-1}
$$
where $k$ is the largest integer greater or equal to $\ell$ such that 
$ \sum_\beta B_{\beta\,\alpha}(n_\ell,n_k) <s$.
\end{proposition}

\begin{proof}
Fix $\alpha \in \A$. Let $n_k$ be as in the proposition. We have 
$ \sum_\beta B_{\beta\,\alpha}(n_\ell,n_{k+1}) \geq s$.
It means that the interval $I_{\alpha}^{(n_\ell)}$ decomposes into at least $s$ subintervals
of level $n_{k+1}$. Therefore there are at least $s+1$ points (including the endpoints of $I_\alpha^{(n_\ell)}= I_{\alpha}^t (T^{(n_\ell)})$) bounding these subintervals and two of them must correspond to the \emph{same} marked point. The (horizontal) segment inside the transversal which has these two {\Green points} as endpoints hence give an \emph{absolute} homology class $\zeta$, which we express in the base $\zeta^{(n_{k+1})}_\alpha$, $\alpha \in \A$, as 
$$ \zeta := \sum_\beta x_\beta \zeta^{(n_{k+1})}_\beta.$$
The coefficients $x_\beta$ are integral, non-negative and the vector $x = (x_\beta)$ belongs to 
$H(\pi^{(n_{k+1})})$.

Let us also represent $\zeta$ at level $n_{k+2}$ as
$$  \zeta := \sum_\beta X_\beta \zeta^{(n_{k+2})}_\beta.$$
 The vectors $x$ and $X$ are related by
$$ X = B(n_{k+1},n_{k+2}) x.$$
As $x \in \overline{\C(\pi^{(n_{k+1})})}$ and the coefficients $X_\beta$ are integers, we get from the definition of $n_{k+2}$ that $X_\beta \geq 1$ for all $\beta \in A$. 
We therefore have
$$ \vert I_{\alpha}^{(n_\ell)} \vert  
\geq \int_\zeta {\rm Re} \, \omega \geq \sum_\beta \int_{\zeta^{(n_{k+2})}_\beta} {\rm Re}\, \omega  = \sum_\beta \lambda^{(n_{k+2})}_\beta 
   \geq \vert  I^{(n_{k+2})} \vert.$$
One has also 
$$
\vert I^{(0)}\vert = \sum_{\beta,\mu} B_{\beta\,\mu}(0,n_{k+2}) \vert I_{\beta}^{(n_{k+2})} \vert \leq   \vert  I^{(n_{k+2})} \vert \sum_{\beta,\mu} B_{\beta\,\mu}(0,n_{k+2}), 
$$
and we conclude from condition {\rm (a')},
$$ 
\vert I_{\alpha}^{(n_\ell)}  \vert \geq C^{-1} \vert I^{(0)} \vert \ \Vert B(0,n_{k+2}) \Vert^{-1} \geq C^{-1}_{\tau} \ \vert I^{(0)}  \vert \ \Vert B(0,n_k) \Vert^{-\tau-1}.
$$
\end{proof}

\subsection{Results on the cohomological equation for absolute Roth type i.e.m.}\label{sec:results}

We show in this paper that the main results on the cohomological equations for i.e.m. which were proved in \cite{MY} under the (restricted) Roth type diophantine condition, actually hold under the weaker  (restricted) {\it absolute Roth type} condition. One of the motivations for this weakening is shown in \S~\ref{sec:aedirection}.

As recalled in the introduction, motivated by Forni's  celebrated paper  \cite{For1} on the cohomological equation for area preserving flows on surfaces of higher genus, in \cite{MMY1} Roth type i.e.m. were introduced in order to provide  an \emph{explicit} diophantine condition sufficient to guarantee the existence of a solution to the cohomological equation up to finitely many obstructions. More precisely, in \cite{MMY1} it is proven that, given a sufficiently regular datum (namely, for functions which are of absolutely continuous on each of the intervals $I^t_\alpha$ with derivative of bounded variation and of zero mean on $I$),  after subtracting (to the datum) a correction function which is constant on each $I^t_\alpha$, the cohomological equation has a bounded solution.  
In \cite{MMY2} this result was used for proving a linearization theorem 
for generalized interval exchange maps with a  rotation number of restricted Roth type. More recently, for restricted Roth type i.e.m., in \cite{MY}  
a more refined regularity result was proved,   
namely the H\"older regularity of the solutions of the cohomological equation with $C^r$, $r>1$ data.
 
In this Section we limit ourselves to the statement of the generalization of the main theorem of \cite{MY} (Theorem 3.10, p.127). In Appendix B we will give the proof of this Theorem, as well as the statement and the proof of the generalizations of other two results of \cite{MY}: the growth estimate for special Birkhoff sums with $C^1$ data (Theorem 3.11, p. 128) and the regularity result for the solutions of the cohomological equation in higher differentiability (Theorem 3.22, p. 138). 


\smallskip

Let $T$ be an i.e.m. and let $r\geq 0$ be a real number. 
The boundary operator  $\partial$ which appears in the following result is the operator defined in  
 in \S~\ref{ssboundary}.  
 We denote by $C^r_{0} (\sqcup I_\alpha)$ the kernel of the boundary operator $\partial$ in $C^r (\sqcup I_\alpha)$. 
We recall from \S~\ref{ssboundary} that the coboundary of a continuous function belongs to the kernel of the boundary  operator (it follows from properties (1) and (2) of  Proposition~\ref{propboundary}).
%


\smallskip
The following result generalizes Theorem 3.10 of \cite{MY} to {\it absolute} restricted Roth type i.e.m. 

\begin{theorem}[H\"older  solutions to the cohomological equation] \label{thmHolder1} 
Let $T$ be an i.e.m. of absolute restricted Roth type. Let 
$\Gamma_u(T)$ be a subspace of $\Gamma_\partial (T)$
which is supplementing $\Gamma_s(T)$ in $\Gamma_\partial (T)$. Let $r >1$ be a real number. 
 There exist $\bar \delta >0$ 
and  bounded linear operators $L_0: \varphi \mapsto \psi$ from 
 $C^r_{0} (\sqcup I_\alpha)$ to the space $C^{\bar \delta}([u_0,u_d])$ of 
H\"older continuous functions of exponent $\bar \delta$ and  $L_1: \varphi \mapsto \chi$ from $C^r_{0} (\sqcup I_\alpha)$ to $\Gamma_u(T)$ such that any $\varphi \in C^r_{0} (\sqcup I_\alpha)$ satisfies
$$ \varphi = \chi + \psi \circ T - \psi\; , \;\; \int_{u_0}^{u_d}\psi (x)dx = 0\; . $$
The operators $L_0$ and $L_1$ are uniquely defined by the conclusions of the theorem.
\end{theorem}
 
\begin{remark}
The exponent $\bar \delta$ depends only  on $r$ and the constants $\theta$, 
$\sigma$ appearing in \S~\ref{ssRoth}. Thus, since in the difference of definition between absolute Roth type and Roth type does not concern the two  costants  
$\theta$ (spectral gap) e $\sigma$ (in the definition of stable space), the exponent $\bar \delta$ in this Theorem~\ref{thmHolder1}  is the same than the exponent in Theorem 3.10 of \cite{MY}. {\Green One can show that $\bar \delta$ tends to $0$ as $r >1$ tends to $ 1$ } {\Green and conjecture that one can take any $\bar \delta>0$ for any fixed $r>1$. This is indeed the case in the Sobolev scale, as proved by Forni in \cite{For3}. Very recently, this optimal loss of derivatives was also proved in the H{\"o}lder class in the pseudo-Anosov case (namely in the case of periodic Rauzy-Veech induction), see \cite{FGL}}. 
\end{remark}

The proofs of these Theorems for i.e.m. of restricted absolute Roth type follow essentially the same proofs given in the respective papers for restricted Roth type. The main difference is in the 
 proof of two crucial estimates on lengths of subintervals of i.e.m.  in a Rauzy-Veech orbit, presented above in \S~\ref{sec:crucialestimates}.  Once these estimates are proved, the proofs of the Theorems for restricted Roth type i.e.m. can be followed and adapted with minor modifications to the case of restricted absolute Roth type i.e.m. For completeness, in the Appendix \ref{app:proofs} we follow the original proof and indicate, for convenience of the interested reader, where and which modifications are needed.

\subsection{Results on the cohomological equation for translation surfeces in a.e. direction}\label{sec:aedirection}

The notion of (absolute) Roth type i.e.m. naturally leads to define (absolute) Roth type translation surfaces: 
the following definition of Roth type translation surfaces was given in \cite{MMY2} and we now extend it to absolute Roth type. 

\begin{definition}\label{absRothsurface}
Let $(M,\Sigma, \zeta)$ be a translation surface. It is of {\it absolute Roth type}  if there exists {\it some} open bounded horizontal segment $I$  in  good position such that the return map $T_I$ of the vertical flow on $I$ is an i.e.m. of {\it absolute Roth type}. 
Absolute restricted Roth type, restricted Roth type and Roth type translation surfaces are defined analogously.
\end{definition}              
 In Appendix C of \cite{MMY2} it is proved that for a (restricted) Roth type translation surface, $T_I$ will actually be of (restricted) Roth type for {\it any} horizontal segment $I$ in good position.  Analogously, one can show that for an absolute (restricted) Roth type translation surface, $T_I$ will be of absolute (restricted) Roth type for {\it any} horizontal segment $I$ in good position. 

\medskip

In \cite{ChaikaEskin} 
Chaika and Eskin prove that, for all translation surfaces $(M,\Sigma,\zeta)$ and for almost every angle $\theta$, the translation surface $(M,\Sigma,r_\theta\cdot\zeta)$ obtained by rotating by $r_\theta$ the translation structure $\zeta$ is generic for the Teichm\"uller geodesic flow and Oseledets generic for the Kontsevich-Zorich cocycle.  As an application of this result,  they prove the following Lemma, which we now state using the definition of absolute Roth type translation surface just given above.\footnote{In Section 1.2.2 of \cite{ChaikaEskin}, the authors define a diophantine conditon on the vertical flow of the translation surface, called in the paper \emph{Roth type} and consisting of three conditions (a), (b) and (c) (see Definition 1.13, \cite{ChaikaEskin}). Their definition is given in a geometric language, but one can check that it is equivalent to our notion of \emph{absolute} Roth type translation surface (Definition \ref{absRothsurface} above), and not to the definition of Roth types surface given in \cite{MMY1}. Lemma 1.16 in \cite{ChaikaEskin} shows that Condition (a) holds for a.e. $\theta$, while the proof that (b) and (c) also holds for a.e. $\theta$ is given at the end of Section 1.2.2.} 

\begin{lemma}[ \cite{ChaikaEskin}]\label{lemma:CE}
For all translation surfaces $(M,\Sigma,\zeta)$,  for almost every angle $\theta$, the translation surface  $(M,\Sigma,r_\theta\cdot\zeta)$ is of absolute Roth type.
\end{lemma}


\begin{remark}
Furthermore, it also follows from  \cite{ChaikaEskin} that\footnote{More precisely, it follows from Lemma 1.16 and Thm 1.4 of \cite{ChaikaEskin}, see the comment at the very end of $\S$ 1.2.2 in \cite{ChaikaEskin}.} if the orbit closure $\mathcal{M}:= \overline{SL(2,\mathbb{R})\cdot (M,\Sigma,\zeta)}$ under the linear action of $SL(2,\mathbb{R})$  is such that $\lambda_g>0$, where $\pm \lambda_i$, for $1\leq i\leq g$ here denote the Lyapunov exponents of the {\Green Kontsevich}-Zorich cocycle restricted to the locus $\mathcal{M}$, then,  for almost every angle $\theta$, the translation surface  $(M,\Sigma,r_\theta\cdot\zeta)$ is of \emph{restricted} absolute Roth type.
\end{remark}

Combining this Remark with the Theorem \ref{thmHolder1} proved in this paper, we get the following corollary\footnote{In \cite{ChaikaEskin},  a result on the solvability of the cohomological equation (under finitely many obstructions) for all translation surfaces in a.e. direction is stated (see the paragraph after Theorem 1.14 in \cite{ChaikaEskin}) and deduced from Lemma~\ref{lemma:CE} and \cite{MMY1}, but since the Definition 1.13 in \cite{ChaikaEskin} is \emph{not} equivalent to Roth type (but to absolute Roth type, see previous footnote), the results from \cite{MMY1} cannot be directly applied, but require the strenghthening proved in this paper. In private communications, though, Chaika and Eskin told us that their proof can be modified to actually show that all translation surfaces are of  Roth type for a.e.  $\theta$.}.

\begin{corollary}
Assume that $(M,\Sigma,\zeta)$ is  such that the orbit closure $ \overline{SL(2,\mathbb{R})\cdot (M,\Sigma,\zeta)}$ has $\lambda_g>0$. Then,  for almost every angle $\theta$,  any i.e.m. $T$ obtained as first return map of the linear flow in direction $\theta$ on $(M,\Sigma,\zeta)$ on a segment in good position satisfies the conclusion of Theorem \ref{thmHolder1}. Thus, for any $r>1$ there exists $\bar \delta >0$ such that for any  $\varphi \in C^r_{0} (\sqcup I^t_\alpha (T))$, there exists a H\"older continuous function $\psi \in C^{\bar \delta}([u_0(T),u_d(T)])$ with zero mean and a piecewise constant correction $\chi \in \Gamma_u(T)$ such that $ \varphi = \chi + \psi \circ T - \psi $. 
\end{corollary}

{\Green 
\begin{remark}
We remark that results on the cohomological equation for almost all directions on a fixed translation surface were first proved by Forni in \cite{For1, For3}. The result in this Corollary, compared with the latter results, improves on the loss of derivatives. Hower, a stronger assumption is assumed, namely hyperbolicity of the KZ cocycle\footnote{We already recalled that  the Masur-Veech mreasure is hyperbolic, as proved in {\Green \cite{For2}}. Results on the presence (or absence) of zero Lyapunov exponents and hyperbolicity for other $SL(2,\mathbb{R})$ invariant measures appear in several works in the literature, see for example \cite{For5, Fi}.} While this paper was under review, two new results on the cohomological equation have appeared: in  \cite{FMM} the Diophantine condition from \cite{MMY1} is weakened to allow the presence of zero Lyapunov exponents, but the loss of derivatives is worst than in our (or \cite{MMY1}) results; in \cite{FGL}, the optimal loss of derivatives without hyperbolicity assumptions is proved, but only  in the pseudo-Anosov (periodic Rauzy-Veech induction) special case (a measure zero class).
\end{remark}}

\section{Dual Roth Type}\label{sec:dualRoth}

We define in this section another  Diophantine condition for translation surfaces  which we call {\it dual Roth type}. 
In order to define it, we need first to introduce \emph{dual special Birkhoff sums} (see \S~\ref{sec:dualBS}). This is a notion which we believe is of independent interest, which is based on a form of \emph{duality} for Rauzy-Veech induction (which heuristically correspond to considering backward time in the Teichmueller geodesic flow and flipping the role of horizontal and vertical flow on translation surfaces). {\Blue It should be noticed that this type of duality was exploited also in the work of Bufetov \cite{Bu}, who discovered a duality between invariant distributions for the horizontal (vertical) flow and finitely additive functionals for the vertical (horizontal) flow. An analogous duality was later found by Bufetov and Forni for horocycle flows on hyperbolic surfaces in their work \cite{BuFo}.} 
Our motivation for introducing a dual special Birkhoff sums operator and dual Roth type conditions  in the combinatorial set up of Rauzy-Veech induction came from the results  proved in Section \ref{sec:centraldistributions}.

\subsection{Backward Rauzy-Veech induction and backward rotation numbers}\label{sec:backwardRauzyVeech}

%


Let ${\Blue M:}= M(\pi, \lambda, \tau)$ be a translation surface constructed from combinatorial data $\pi \in \R$, length data $\lambda \in \mathcal{C}$ and suspension data $\tau \in \Theta_\pi$. We assume in this section   that $M$ has \emph{no horizontal saddle connections}. 
Then, starting from $(\pi^{(0)},\lambda^{(0)},\tau^{(0)}):=(\pi,\lambda,\tau)$, we can iterate indefinitely the Rauzy-Veech \emph{backward} algorithm and (using the notation recalled in \S~\ref{sec:background}) construct the sequence as $\zeta_\alpha^{(n)}=\lambda^{(n)}_\alpha+i \tau_\alpha^{(n)}$  for $n\leq 0$ follows:

\begin{itemize}
\item when $\sum_\alpha \tau_\alpha^{(n)} <0$, we perform the inverse of an elementary \emph{top} step of the Rauzy-Veech algorithm
\begin{eqnarray*}
\pi^{(n-1)}&=& R_t^{-1} \left(\pi^{(n)}\right)\\
\zeta_{\alpha_w}^{(n-1)} &=&  \zeta_{\alpha_w}^{(n)} + \zeta_{\alpha_{\ell}}^{(n)}, \\
\zeta_\alpha^{(n-1)} &=& \zeta_\alpha^{(n)}, \qquad \forall \alpha \ne \alpha_w, 
\end{eqnarray*}
where $\alpha_w$, $\alpha_{\ell}$ are the letters such that 
$$
\pi_t^{(n)}(\alpha_w) = d, \qquad \pi_b^{(n)} (\alpha_\ell) = \pi_b^{(n)}(\alpha_w) +1;
$$
\item when $\sum_\alpha \tau_\alpha^{(n)} >0$, we perform the inverse of an elementary \emph{bottom} step of the Rauzy-Veech algorithm
\begin{eqnarray*}
\pi^{(n-1)}&=& R_b^{-1} \left(\pi^{(n)}\right)\\
\zeta_{\alpha_w}^{(n-1)} &=&  \zeta_{\alpha_w}^{(n)} + \zeta_{\alpha_{\ell}}^{(n)}, \\
\zeta_\alpha^{(n-1)} &=& \zeta_\alpha^{(n)}, \qquad \forall \alpha \ne \alpha_w, 
\end{eqnarray*}
where $\alpha_w$, $\alpha_{\ell}$ are the letters such that 
$$
\pi_b^{(n)}(\alpha_w) = d, \qquad \pi_t^{(n)} (\alpha_\ell) = \pi_t^{(n)}(\alpha_w) +1.
$$
\end{itemize}
 The case $\sum_\alpha \tau_\alpha(n) =0$ never occurs because there are no 
 horizontal saddle connections.
 
 \smallskip
 
We obtain in this way a {\it backward rotation number}, i.e an infinite  path $\ug$ in $\D$ with terminal point $\pi^{(0)}$ of the form
$$\ug = \ldots \star \gamma_{n-1}\star \gamma_{n}\star \ldots \gamma_{-1} \star \gamma_{0},$$ 
where $\gamma_{n}$, $n\geq 0$, is the arrow from $\pi^{(n-1)}$ to $\pi^{(n)}$.

\begin{remark}\label{remark1}
Remark that the equations for $\zeta_\alpha^{(n)}=\lambda^{(n)}_\alpha+i \tau_\alpha^{(n)}$ split into independent sets of equations for $\lambda^{(n)}$ and $\tau^{(n)}$. Since 
 the type of  backward move (top or bottom) depends only on $\tau^{(n)}$, it follows that 
 the backward rotation number $\ug$ only depends on $(\pi, \tau)$. 
 \end{remark}

%
\begin{remark}\label{remark2} 
The backward rotation number is defined as soon as the preferred rightwards separatrix in $M$ (issued from the leftmost vertex $O$ of the polygon) is not a saddle connection. Other 
 horizontal connections are apparently not detected by the algorithm (i.e the algorithm does not stop in this case).

In particular, if $\tau$ is such that its components $\tau_\alpha$, $\alpha \in \A$, are irrationally independent over $\mathbb{Q}$, then the backward Rauzy algorithm does not stop. Thus, to any $\pi$ and a.e. $\tau \in \Theta_\pi$, one can associate to the pair $(\pi, \tau)$ a backward rotation number. 
\end{remark}

\smallskip
Recall that a finite path in $\D$ is {\it complete} if every letter in $\A$ is the winner of at least one of its arrows and that an infinite path  is {\it $\infty$-complete}  if it is the concatenation of infinitely many finite complete paths. 

\begin{proposition}[Completeness of backward rotation numbers]\label{p7}
When $M= M(\pi, \lambda, \tau)$ has no horizontal connection, the backward rotation number $\ug$ associated to $(\pi, \tau)$ is $\infty$-complete.
\end{proposition} 
%

We give the  proof of this Proposition  in the Appendix \ref{secbac}, since the proof is quite long and rather involved combinatorially.  It would be interesting to know whether the converse to the above Proposition is true (see Remark \ref{remark2} above). 
%

\subsection{Dual special Birkhoff sums}\label{sec:dualBS}
Let $M(\pi, \lambda, \tau)$ be a translation surface as above that has \emph{no horizontal saddle connections}. Thus, we have well defined data $(\pi^{(n)},\lambda^{(n)},\tau^{(n)})$ for any $n \in \mathbb{Z}$. For any $n\in\mathbb{Z}$ we denote by $T^{(n)}$ the i.e.m. with data $(\pi^{(n)},\lambda^{(n)})$. Let $I^{(n)}=I(T^{(n)})$ be the sequence of intervals on which $T^{(n)}$ acts for $n\in\mathbb{Z}$. 

\smallskip
Let us recall that  special Birhoff sums (see \eqref{specialBS_def}) are operators $S(m,n)$, where $m \leq n$, which map from functions defined on $\sqcup I_\alpha^{(m)}$ to functions defined on $\sqcup I_\alpha^{(n)}$  and that, when acting on the space of piecewise constant functions, $S(m,n): \Gamma(m)  \to \Gamma(n) $ is represented by the KZ cocycle matrix $ B(m,n)$ (see \S~\ref{ssKZ}).


We will now define a dual operator, which we call \emph{dual special Birkhoff sums}. It is useful to remark that this operator plays for the \emph{horizontal} linear flow on a translation surface the role of the special Birkhoff sums play for the \emph{vertical} flow. 
%

\smallskip

Let $q^{(n)}$ be the vectors given by $q^{(n)}=-\Omega_{\pi^{(n)}} \, \tau^{(n)}$, so that the translation surface $M(\pi^{(n)}, \lambda^{(n)}, \tau^{(n)})$ can be represented as a zippered rectangle over $T^{(n)}$ with heights $q^{(n)}$ (see \S~\ref{sec:transl}). 
For $n \in \Zset$, $\alpha \in \A$, let 
$$L_{\alpha}^{(n)}:= \left(0,q_{\alpha}^{(n)}(\pi,\tau)\right), \qquad L^{(n)}:=\bigsqcup_{\alpha \in \A} L_{\alpha}^{(n)}.$$
The disjoint union  $L^{(n)}$ will be the domain of the functions on which the $n^{th}$ special Birkhoff sum operator will act. Geometrically, one should think of the intervals $L_{\alpha}^{(n)}$, $\alpha \in \A$, as \emph{vertical} intervals, one for each zippered rectangle representation with heights $q^{(n)}$ of the translation surface  $M(\pi^{(n)}, \lambda^{(n)}, \tau^{(n)})$. Remark that, for any $n\leq 0$, $q_{\alpha}^{(n)}$ depends only on $\pi$ and $\tau$ (see Remark \ref{remark1}). 

For $n' \leq n$, $\alpha, \in \A$, we can write  the interval $L_{\alpha}^{(n)}$  as a disjoint union (modulo $0$) of a number of translated copies of the  $L_{\beta}^{(n')}$ as follows (see also Figure~\ref{fig:pastsums}).
%
%


 Let us recall that  $B_{\alpha}(n',n):=  \sum_{\beta \in \A} B(n',n)_{\alpha \beta }$ gives the return time in $I^{(n)}$ of $I_{\alpha}^{(n)}$ under $T^{(n')}$ (see \S~ref{sec:algorithms}). For $0 \leq j < B_{\alpha}(n',n)$, let $\beta(\alpha,j)$ be the letter in $\A$ such that $$(T^{(n')})^j (I_{\alpha}^{(n')}) \subset I_{\beta(\alpha,j)}^{(n)}.$$ 
Then, for any $\alpha \in \mathcal{A}$, we can write 
$$
q_{\alpha}^{(n)} = \sum_{0\leq j < B_\alpha(n',n)} q_{\beta(\alpha,j)}^{(n')}. 
$$
Correspondingly, one can write the interval $L_{\alpha}^{(n)}$, which has length $q_{\alpha}^{(n)} $,  
as union of intervals of lengths $q_{\beta(\alpha,l)}^{(n')}$, each of which is  translated copy of the corresponding interval $L_{\beta(\alpha,l)}^{(n')}$, i.e. 
$$
L_{\alpha}^{(n)} = \bigcup_{0\leq j < B_\alpha(n',n)} \left( L_{\beta(\alpha,j)}^{(n')}(\pi, \tau) +  Sq(j) 
\right), \qquad \text{where} \quad S q (j)
:= \sum_{0\leq i < j} q_{\beta(\alpha,i)}^{(n')}(\pi, \tau) .
$$
Thus, collecting together all $j$ such that $\beta(\alpha,j)=\beta$ for a fixed $\beta \in \mathcal{A}$ (there are $ B(n',n)_{\alpha,\,\beta}$ of them)  and then denoting by $j_{\ell, \alpha,\beta}^{(n',n)} (L_{\beta}^{(n')})$ the $\ell^{th}$ shifted copy of $L_{\beta}^{(n')}$, for $0\leq \ell < B(n',n)_{\alpha,\,\beta}$, we can write 
\begin{equation} \label{decompL}
L_{\alpha}^{(n)}= 
 \bigcup_{\beta \in \A}\ \bigcup_{\substack{0 \leq j < B_\alpha(n',n), \\ \text{s.\ t.}\ \beta(\alpha,j) =\beta}}
 \left( L_{\beta(\alpha,j)}^{(n')} + Sq(j)\right)  
 = \bigcup_{\beta \in \A}\  \bigcup_{0 \leq \ell < B(n',n)_{\alpha,\,\beta}} j_{\ell, \alpha,\beta}^{(n',n)} (L_{\beta}^{(n')}).
\end{equation}



 We define for every $n'\leq n$, an operator $S^\sharp (n,n')$ that sends functions on  $L^{(n)}$ to functions on  $L^{(n')}$. Let $\psi$ be a function defined  on  $L^{(n)}$. Then the image $S^\sharp (n,n')\psi$ is the function $\Psi(x)$  on $L^{(n')}$ defined as follows: 
\begin{equation}\label{def:dualBS}
\begin{split}
S^\sharp (n,n')\psi = \Psi(x)&
 : =  \sum_{\alpha \in \A}\, \sum_{\substack{ 0 \leq j < B_\alpha(n',n) \\ \text{s.\ t.}\ \beta(\alpha,j) =\beta }} 
 \left( L_{\beta(\alpha,j)}^{(n')} + Sq(j)\right)   \\ &
  = \sum_{\alpha \in \A} \sum_{ 0 \leq \ell < B(n',n)_{\alpha,\,\beta}} \psi(j_{\ell, \alpha,\beta}^{(n',n)}(x)), \qquad \text{if }\,  x \in L_{\beta}^{(n)} .
\end{split}
\end{equation}
We will call $S^\sharp (n,n')\psi$ the \emph{dual special Birkhoff sums} and $S^\sharp (n,n')$ \emph{dual special Birkhoff sum operator}. 

\smallskip
Let $\Gamma^\sharp (n)$ be the space of functions on $L^{(n)}$ which are constant on each $L_{\alpha}^{(n)}$; it is canonically isomorphic to $\RA$. 
The operator $S^\sharp (n,n')$ sends $\Gamma^\sharp (n)$ to $\Gamma^\sharp (n')$ and the matrix w.r.t. the canonical bases is the transpose matrix
$^tB(n',n)$ (notice that here the order of $n,n'$ is reverted). This explains the name {\it dual Birkhoff sum operator} for $S^\sharp (n,n')$.

\begin{figure}
    \includegraphics[width=.8\textwidth]
		{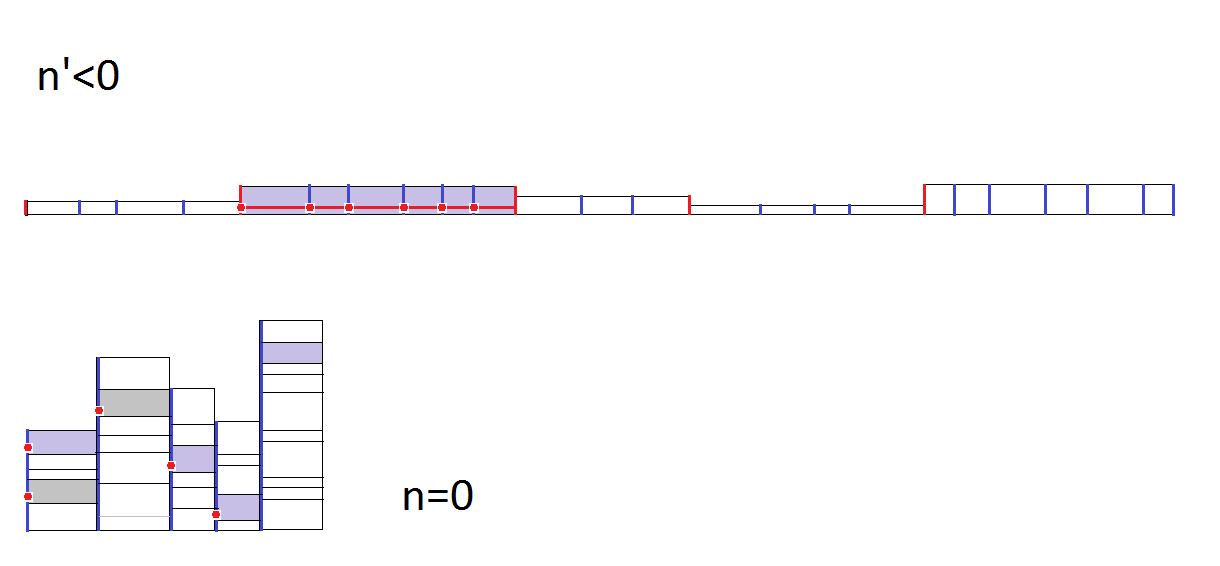}
		\caption{Geometric interpretations of the \emph{dual} special Birkhoff sums $ S^\sharp(n,n')$, $\quad n'<n$. \label{fig:pastsums} }
	\end{figure}
	\begin{figure}
    \includegraphics[width=.8\textwidth]
		{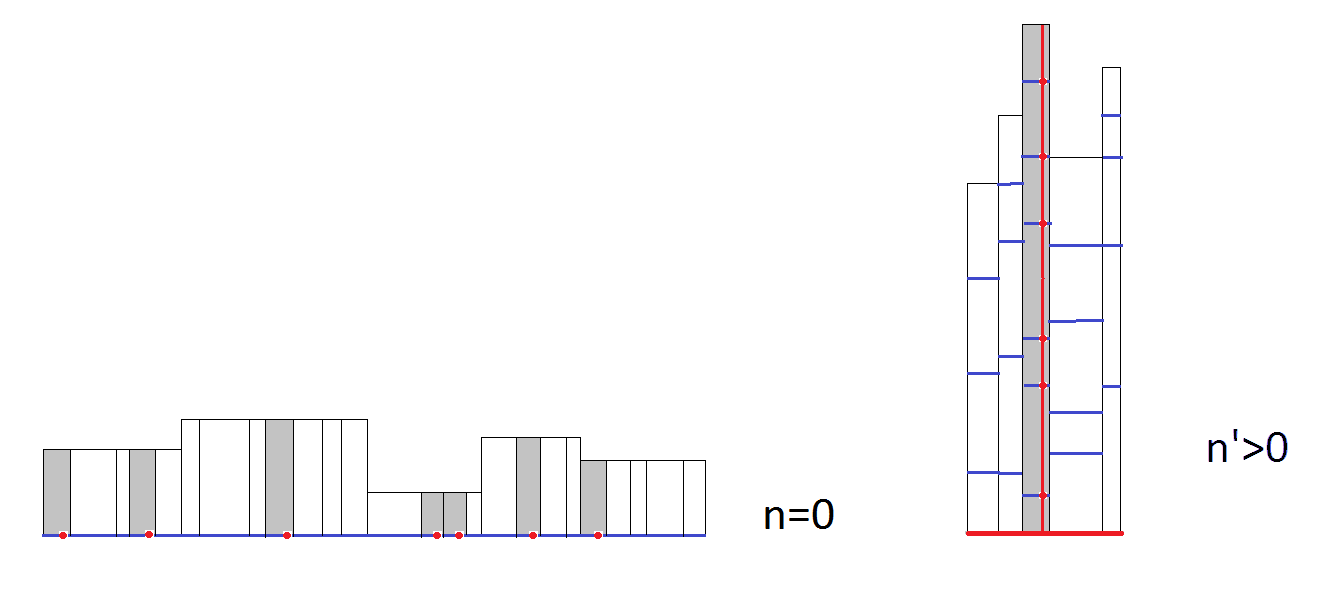}
		\caption{Geometric interpretations of the (\emph{direct}) special Birkhoff sums $ S(n,n')$, $\quad n'>n$.\label{fig:futuresums}}
\end{figure}

\begin{remark}
One can interpret
this definition as a \emph{special Birkhoff sum} for a Poincar{\'e} map of the \emph{horizontal} translation flow as follows. For each $n'\leq n, \alpha \in \A$, {\Green identify} each interval $L^{(n')}_\alpha$ with the left vertical side of the rectangle $R^{(n')}_\alpha = I^{(n')}_\alpha \times [0,q^{(n')}_\alpha] $ of the zippered rectangle presentation of $M(\pi^{(n')}, \lambda^{(n')}, \tau^{(n')})$. 
For $x\in L^{(n')}_\beta$, consider the leaf of the horizontal flow that starts at $x$, up to the first time it hits again the union $\cup_{\alpha\in \A}L^{(n'{\Blue)}}_\alpha$ (see the top part of Figure~\ref{fig:pastsums}, the horizontal leaf is contained in the coloured column). This piece of horizontal leaf hits the union $L^{(n)}$ (on which $\psi$ is defined)  
in exactly $\sum_\alpha B(n',n)_{\alpha,\,\beta}$ points (also shown in top part of Figure~\ref{fig:pastsums}, as dots along the horizontal leaf).  The dual special Birkhoff sum $S^\sharp (n,n')\psi (x) $ is obtained by adding the values of the function $\psi$ at these hitting points.


To highlight the duality with (direct) special Birkhoff sums, let us remark that the special Birkhoff sums can be defined similarly as follows (refer to Figure~\ref{fig:futuresums}). Given $n''>n $ and a function $\phi$ defined on $I^{(n)}$, $S^\sharp (n,n'') \phi(x)$  is obtained by considering a leaf of the \emph{vertical} flow starting from $x \in I^{(n'')}_\beta$ up to its first return to $I^{(n'')}$ and summing up the values of the function $\phi$ at the  hitting  points of this piece of a leaf with the interval $I^{(n)}$ (as shown in the {\Blue right picture} of Figure~\ref{fig:futuresums}). 
One can also give an equivalent definition of (direct) special Birkhoff sum similar (and dual to) the definition  \eqref{def:dualBS}. Namely, the interval $I^{(n)}_\alpha$ can be written as a union of translated copies of the intervals $I^{(n'')}_\beta$ (as one can see in the {\Blue left picture} of Figure~\ref{fig:futuresums})). Let $j_{\ell, \alpha,\beta}^{(n,n'')}\left( I^{(n'')}_\beta \right)$ be the $\ell$'s copy of the interval  $I^{(n'')}_\beta$ inside  $I^{(n)}_\alpha$, where $0\leq \ell < B (n,n'')_{\alpha,\beta}$. Then
$$
S(n,n'') \phi(x) = \sum_{\alpha \in \A} \, \sum_{0\leq \ell < B (n,n'')_{\alpha,\beta}} j_{\ell, \alpha,\beta}^{(n,n'')}(x) .
$$ 

Conversely, in order to give a definition of dual special Birkhoff sum formulated in a similar way to the classical definition of (direct) special Birkhoff sums (i.e. as a sum over an orbit of the induced  i.e.m. $T^{(n)}$ up to its first return to $I^{(n'')}$), one would need to describe the applications that arise as first return maps of the horizontal linear flow on the unions $L^{(m)}$, $m\leq 0$ of vertical intervals. One can describe these as piecewise isometries of a special type and the algorithm that produces the induced map on $L^{(m-1)}$ given the one on $L^{(m)}$ turns out to be related to the da Rocha algorithm \cite{daRocha}. We refer to \cite{dualRV} in which the authors prove that a variant of da Rocha induction is \emph{dual} (in the sense of Scheweiger \cite{Schweiger}) to Rauzy-Veech induction for i.e.m.\end{remark}

\smallskip
The dual Birkhoff sums operators $S^\sharp (n,n')$, $n'<n$, enjoy similar properties to the (direct) Birkhoff sums operators $S (n,n'')$, $n''>n$ (recalled in \S~\ref{ssKZ}). In particular, for $n''\leq n' \leq n$, one has:
$$S^\sharp (n, n'') = S^\sharp (n',n'') \circ S^\sharp (n,n').$$
\smallskip
Observe moreover that, if $\psi$ is integrable on $L^{(n)}$, then $S^\sharp (n,n') \psi$ is integrable on $L^{(n')}$ and we have
$$\int_{L^{(n')}} S^\sharp (n,n') \psi = \int_{L^{(n)}} \psi .$$
We will denote by $^tB_0(n',n)$ the restriction of $S^\sharp (n,n')$ to the hyperplane $\Gamma_0^\sharp (n) \subset \Gamma^\sharp (n)$ of functions of zero mean value.

\subsection{Dual Roth type condition}\label{sec:dualRothcondition}
We can now define the {\it dual Roth type} Diophantine condition. 

\smallskip
 Let $\underline \gamma^*= \cdots \star \gamma^*_{-1}\star \gamma^*_0$ be a backward rotation number, i.e. an infinite path in the Rauzy diagram $\mathcal D$ ending in $\pi$. 
$(\pi, \tau)$, where $\tau \in \Theta_\pi$, by the Rauzy-Veech algorithm.  
Let us assume that $\underline \gamma^*$ is infinitely complete. 
 Define $\hat n_0 =0$. Given $\hat n_k$ for some $k{\color{blue}\leq} 0$, define inductively $\hat n_{k-1}\leq  0 $ as the largest (negative) integer $n <  n_{k}$ such that the path 
$\gamma( n, n_{k})= \gamma_{ n+1} \star \ldots \star \gamma_{n_{k}}$ is $\infty$-complete (see \S~\ref{secRV}).

\smallskip
 As we saw in  Proposition \ref{p7}, $\underline \gamma^*$  is infinitely complete as long as $\underline \gamma^*$ is the backward rotation number of $M(\pi,\lambda, \tau)$ with no horizontal connection. In particular, since having no connection is a full measure condition and horizontal connections depend on the suspension datum only, for any combinatorial data $\pi$, for a.e. $\tau \in \Theta_\pi$ (with respect to the induced Lebesgue measure on $\Theta_\pi$), we can associate to $(\pi, \tau)$ an infinitely complete backward rotation number $\underline \gamma^*$. 


\smallskip
We can hence give the following definition, where  the sequence  $(\hat n_k)$ that appears in Condition (a) is the sequence of accelerated times defined above. 


\begin{definition}\label{def:dualRoth}
A pair of suspension data $(\pi, \tau)$, where  $\tau \in \Theta_\pi$, is of \emph{weak dual  Roth type} (or, often in this paper, simply \emph{dual Roth type}) if one can associate to $(\pi, \tau)$ an infinitely complete backward rotation number $\underline \gamma^*$  and  the following  conditions are satisfied:

\begin{itemize}
\item[(a)] ({\it dual matrices growth}) for every $\varepsilon >0$, there exists $C_0=C_0(\va)>0$ such that, for all $k\leq 0$
$$||B( \hat n_{k-1}, \hat n_{k})|| \leq C_0 ||B(\hat n_{k},0)||^{\va};$$

\item[(b)] ({\it dual spectral gap}) there exists $\theta >0$ and $C_1>0$ such that, for all $n \leq 0$
$$||{}^tB_0(n,0)||\leq C_1 ||\;^tB(n,0)||^{1-\theta} ,$$
where ${}^tB_0(n,0)$ is the restriction of $\;^tB(n,0)$ to the hyperplane $\Gamma^\sharp (0)$.
\end{itemize}

\end{definition}

These two conditions are dual to conditions (a) and (b) in the (direct) Roth type condition.   For the results in this paper, we do not require a third condition dual to the coherence condition (i.e. condition (c)) which appears in the (direct) Roth type condition, see Definition~\ref{defRoth2} in \S~\ref{ssRoth}). Hence we define this dual Diophantine condition with only (a) and (b), but we call it  \emph{weak} dual Roth type condition (and in parallel we call \emph{weak} Roth type the direct condition with (a) and (b) only) to leave room to the possibility of defining also a dual Roth type (or restricted dual Roth type) with additional conditions.

\begin{remark}
A stronger condition is to require that $(\pi,\tau)$ has {\it bounded type}, i.e that the matrices $B(n_{k}, n_{k-1})$ for $k< 0$ above stay uniformly bounded (equivalently, these matrices take their values in  a finite set). This clearly implies condition (a). Condition (b) follows from the uniform contraction of the Hilbert metric in the positive cone of $\RA$.
\end{remark}

{\Blue 

\begin{remark}\label{dualremRoth0} 
As in the case of Condition (a) of the (direct) Roth type condition (see Remark~\ref{remRoth0}), also 
Condition (a) of Definition~\ref{def:dualRoth} can be rephrased equivalently using different sequences of times, for example using the sequence ${(\tilde{n}_k)}_{k\leq 0}$ defined setting $ \tilde{n}_0 =0$ and,  for $k<0$ defining inductively $ \tilde{n}_{k-1}$ as the largest integer $n<  \tilde{n}_{k}$ such that the matrix $B(n, \tilde{n}_{k-1})$ has positive coefficients. 
\end{remark}
}


\smallskip
Definition \ref{def:dualRoth} clearly does not depend on the choice of the norm.
It is convenient for the following sections to take as the norm of a matrix the sum of all coefficients in absolute value, namely 
\[
\Vert B \Vert : = \sum_{\alpha, \beta \in \A} |B_{\alpha,\beta}| . 
\]
The absolute values will not appear  since the matrices
that we consider here all have nonnegative entries. In particular, notice that with this choice of norm we have that $\Vert B  \Vert = \Vert {}^t B \Vert$. We will use this observation freely throuought the next sections.

\subsection{Dual lengths control}\label{sec:dualcontrol}
The following Lemma shows that Condition (a) in  the dual Roth type condition imply a control on  lengths of the intervals $L_{\beta}^{(m)}$, $m<0$, which is analogous to the control of the lengths of the subintervals  $I_{\alpha}^{(n)}$, $n>0$ given by Condition (a) of the  Roth type condition or Condition (a) of the absolute Roth type condition (see \S~\ref{sec:crucialestimates}). 

\begin{proposition}\label{lengthsL}
For all $\va >0$ there exists $C_0= C_0(\va)>0$ and  $C_1= C_1(\va)>0$ such that, for all $\beta \in \A$, 
$$  C_0 \, ||B(m,0)||^{-1 -\va}  \leq  |L_{\beta}^{(m)}| \leq C_1\, ||B(m,0)||^{-1 +\va}. $$
\end{proposition}
We remark that the estimate from below is not needed for our immediate purposes, but useful in a more general context. 
We conclude this section by proving this Proposition. The proof using the following auxiliary Lemma.

\begin{lemma}\label{aux_lemma}
For every $\varepsilon >0$, there exists $C'=C'(\va)>0$ such that, for all $m\leq 0$, all $\alpha, \,\beta \in \A$
$$ \max \left(1,B_{\alpha,\,\beta}(m,0)\right) \geq C' ||B(m,0)||^{1-\va}.   $$
\end{lemma}
\begin{proof}
{
Let $(\hat n_{k})_{ k\geq 0}$ be the subsequence of induction times defined at the beginning of \S~\ref{sec:dualRothcondition}. } 
Without loss of generality, we can assume that $m \leq \hat n_{-4d+5}$, since if $\hat n_{-4d+5} < m\leq 0$, 
the conclusion is true as soon as we take $C' \leq || B(\hat n_{-4d+5},0)||^{-1}$. Therefore, if $k$
the unique integer  such that $\hat{n}_{k-1}< m \leq  \hat{n}_k$, we can assume that $k \leq -4d+5$.   
 By \cite{MMY1} (see Lemma in $\S$~1.2.2 of \cite{MMY1}), all coefficients of the matrix $X:=B(  \hat n_k, \hat n_{k+2d-3})$ are strictly positive, hence at least equal to $1$; the same is true for the coefficients of the matrix $Y:=B(\hat n_{-2d+3},0)$. Moreover, 
 Setting $V:=B( \hat n_{k+2d-3}, \hat n_{-2d+3} )$ (which is well defined since $k< -4d+6$), we can write (recalling the cocycle relation \eqref{matrix_cocycle_eq}) 
 $$U :=B(\hat n_k,0)  =  B(\hat n_{-2d+3},0) B( \hat n_{k+2d-3}, \hat n_{-2d+3} ) B(  \hat n_k, \hat n_{k+2d-3}) =    YVX.$$ 
Thus, we have that, for all $\alpha, \,\beta,\, \alpha', \,\beta' \in \A$ (using also that $ B (n,0)_{\alpha,\beta}$ is a non decreasing  function of $n$ and the definition of $k$), 
$$B_{\alpha,\,\beta}(m,0) \geq U_{\alpha,\,\beta} \geq Y_{\alpha,\,\alpha'}V_{\alpha',\,\beta'}X_{\beta',\,\beta} \geq V_{\alpha',\,\beta'}.$$ 
Choosing $\alpha', \beta'$ such that $V_{\alpha',\,\beta'} = \max_{\alpha,\beta} V_{\alpha,\beta}$, since 
 $\max_{\alpha,\beta} V_{\alpha,\beta}$ is a norm on $\mathbb{R}^\A$ and all norms on $\mathbb{R}^d$ are equivalent,  we get that there exists $c>0$ such that
$$B_{\alpha,\,\beta}(m,0) \geq c ||V||.$$ 
On the other hand, from condition $(a)$ in the Definition \ref{def:dualRoth} of dual Roth type and the choice of $k$, we have that for any $\va>0$ there exists $\tilde C_0 = \tilde C_0(\va)$ and $\tilde C_1 =\tilde C_1(\va)$ such that
$$||B(m,0)|| \leq ||B(\hat n_{k-1},0)||  \leq ||YV||\;||B(\hat n_{k-1}, \hat n_{k+2d-3} )|| \leq \tilde C_0 ||YV||^{1+\va} \leq \tilde C_1 ||V||^{1+\va},$$
giving the estimate of the lemma.
\end{proof}

\begin{proof}[Proof of {\Blue Proposition}~\ref{lengthsL}]
If follows from the geometric decomposition \eqref{decompL} that, for every $\alpha \in \A$, {
 for every $n'<n''$}
\begin{equation}\label{Lrel} |L_{\alpha}^{({
n''})}| =   \sum_{\beta \in \A}B_{\alpha,\,\beta}({
n',n''})|L_{\beta}^{({
n'})}|,\end{equation}
which gives that for every $\beta \in \A$, {
 for every $n'<n''$}
\begin{equation}\label{star} \quad \quad |L_{\beta}^{({\Blue n'})}| \leq  {\Blue \max_{\alpha \in \A} \left((B_{\alpha,\,\beta}({\Blue n',n''}))^{-1}  |L_{\alpha}^{({\Blue n''})}| \right)} \leq  \max_{\alpha \in \A} (B_{\alpha,\,\beta}({\Blue n',n''}))^{-1}  {\Blue \max_{\alpha \in \A}} |L_{\alpha}^{({\Blue n''})}|.
\end{equation}
The estimate from above follows from  Lemma \ref{aux_lemma} and the inequality \eqref{star} above {\Blue for $n':=m$ and $n'':=0$}. 

\smallskip
To prove the estimate from below, we define $k{\Blue \leq} 0$ as in the proof of the previous Lemma, i.e.\ so that $\hat n_{k-1} < m \leq \hat{n}_k$.
  Writing (recalling the cocycle equation \eqref{matrix_cocycle_eq}) 
$$V=:B( \hat n_{k-2d+2},0) =  B(m,0) B(\hat n_{k-2d+2}, m) = B(m,0) X , \qquad \text{for}\ X:= B(\hat n_{k-2d+2}, m),$$ 
 we know again (from $m> \hat n_{k-1}$, $k-1-(k-2d+2)=2d-3$ and the Lemma in $\S$~1.2.2 of \cite{MMY1}) that all coefficients of $X$ are at least equal to $1$. 
We have then that for all $\beta \in \A$ ({\Blue by \eqref{Lrel}} for $n':={\Blue \hat{n}_{k-2d+2}}$, $n'':=m$)
$$|L_{\beta}^{(m)}| \geq \max_{\alpha \in \A} |L_{\alpha}^{({\Blue \hat{n}_{k-2d+2}})}|.$$
From {\Blue \eqref{Lrel}} {\Blue for $n':={\Blue \hat{n}_{k-2d+2}}$, $n'':=0$ and the definition of $||\cdot ||$, }
 we get that {\Blue  there exists $c>0$ such that}
$$ \max_{\alpha \in \A} |L_{\alpha}^{({\Blue \hat{n}_{k-2d+2}})}| \geq c ||B({\Blue \hat{n}_{k-2d+2}},0)||^{-1}.$$
Finally,  {\Blue  using repeatingly} condition (a) of the dual Roth type definition~\ref{def:dualRoth}  {\Blue  (and then $m\leq \hat{n}_k$ and monotonicity of $n\mapsto  ||B(n,0)||$)}, we have that for every $\va$ there exists $C(\va)$ such that
$$||B({\Blue \hat{n}_{k-2d+2}},0)|| \leq ||B(m,0)||\; ||B( \hat{n}_{k-2d+2},  {\Blue \hat{n}_{k}})|| {\Blue  \leq C(\va) ||B(\hat{n}_{k},0)||^{1+\va}} \leq C(\va) ||B(m,0)||^{1+\va}.$$
The three previous inequalities  imply the required estimate.
\end{proof}

\subsection{Estimates on dual special Birkhoff sums of  H\"older functions}\label{sec:dualHolderestimates}

The main result in this section is a bound on the growth of dual special Birkhoff sums for H\"older continuous functions under the dual Roth type condition. The main result proved in this section is the following proposition. The analogous estimate for growth of (direct) special Birkhoff sums for H\"older continuous  was given in \cite{MMY1} (see the Proposition at the end of $\S$~2.3 in \cite{MMY1}).

\begin{proposition}\label{estimatedualBS}
Assume that $(\pi, \tau)$ is of dual Roth type. Then, for every function $\psi$ on $L^{(0)}$ which has mean value $0$ and is H\"older continuous for some exponent $\eta \in (0,1)$, the dual Birkhoff sums $S^\sharp (0,n) \psi$ satisfy
$$||S^\sharp (0,n) \psi||_{C^0} \leq C(\eta) ||\psi||_{C^{\eta}}\,  \;||B(n,0)||^{1- \frac 1{20}\theta \eta},$$
where $\theta$ is the exponent that appears in the definition of dual Roth type (see (b) in Definition \ref{def:dualRoth}). 
\end{proposition}
The rest of the section is devoted to the proof of this Proposition. We obtain an estimate on the $C^0$-norm of $S^\sharp (0,n) \psi$ by the same method used for (direct) Birkhoff sums. Let us remark that following the same methods in \cite{MMY1}, the Proposition \ref{estimatedualBS} can also be proved for  functions of bounded variation.
%

\begin{proof}[Proof of Proposition \ref{estimatedualBS}]
Let $(\pi, \tau)$ be of dual Roth type and let $\psi$ be a function on $L^{(0)}$ as in the assumptions of the Proposition. 
We first write $\psi= \psi_0 + \nu_0$, where $\nu_0 $ belongs to the space $ \Gamma_0^\sharp (0)$ (defined in $\S$~\ref{sec:dualBS}) and $\psi_0$ has mean value $0$ on \emph{each} interval $L_{\alpha}^{(0)}$.
For $0 \geq m >n$, we write
$$S^\sharp(m,m-1)\psi_m = \psi_{m-1} + \nu_{m-1},$$
where $\nu_{m-1} \in \Gamma_0^\sharp (m-1)$ and $\psi_{m-1}$ has mean value $0$ on  each interval $L_{\alpha}^{(m-1)}$.
We have then 
\begin{equation}\label{spadesuit} \quad \quad S^\sharp (0,n) \psi = \psi_n + \sum_{n\leq m\leq 0} S^\sharp (m,n) \nu_m.
\end{equation}
To estimate the terms in this sum, we first observe that, for $n\leq m \leq 0$ and each $\beta \in \A$ and  $x,y \in L_{\beta}^{(m)}$,  using that $\psi$ is H\"older-continuous, we have that
$$|\psi_m(x) -\psi_m(y)| = |S^\sharp (0,m) \psi (x) -S^\sharp (0,m) \psi (y)| \leq ||\psi||_{C^{\eta}} |L_{\beta}^{(m)}|^{\eta} \sum_{\alpha \in \A}B_{\alpha,\,\beta}(m,0).$$ 
As $\psi_m$ has mean $0$ in each $L_{\beta}^{(m)}$, this gives
$$||\psi_m||_{C^0}  \leq ||\psi||_{C^{\eta}} |L_{\beta}^{(m)}|^{\eta} \sum_{\alpha \in \A}B_{\alpha,\,\beta}(m,0).$$ 
The estimate from above in Lemma \ref{lengthsL} gives, for all $\va>0$ 
$$||\psi_m||_{C^0}  \leq C(\va)||\psi||_{C^{\eta}} ||B(m,0) ||^{1-\eta(1-\va)} ,$$ 
and thus also, as $S^{\sharp}(m,m-1)\psi_m=\psi_{m-1}+\nu_{m-1}$,  
$$||\nu_{m-1}||_{C^0}  \leq C(\va)||\psi||_{C^{\eta}} ||B(m,0) ||^{1-\eta(1-\va)} .$$ 
From the formula \eqref{spadesuit} above, we get, taking $\va = \frac 12$
$$ ||S^\sharp (0,n) \psi||_{C^0} \leq C ||\psi||_{C^{\eta}} \sum_{m=0}^n || {}^t B_0(n,m)||\;||B(m,0) ||^{1-\frac 12 \eta}.$$
In this sum, we distinguish two cases:

\smallskip
\noindent {\it Case 1.} When $||B(m,0) || \leq ||B(n,0) ||^{\frac {\theta}3}$, we write $B_0(n,m) = B_0(m,0)^{-1} B_0(n,0)$ and obtain, from condition (b) and the symplecticity of $B(m,0)$
\begin{eqnarray*}
|| {}^t B_0(n,m)||\;||B(m,0) ||^{1-\frac 12 \eta} &\leq & ||B_0(m,0)^{-1}||\;|| B_0(n,0)|| \;||B(n,0) ||^{\frac {\theta}3} \\
  &\leq &||B(m,0)^{-1}||\;|| B_0(n,0)|| \;||B(n,0) ||^{\frac {\theta}3} \\
   &\leq & C||B(m,0)||\;||B(n,0)||^{1-\theta} \;||B(n,0) ||^{\frac {\theta}3} \\
  &\leq & C||B(n,0)||^{1-\frac{\theta}3 }.
  \end{eqnarray*}

\smallskip
\noindent {\it Case 2.} When $||B(m,0) || > ||B(n,0) ||^{\frac {\theta}3}$, we use the trivial bound of $|| {}^t B_0(n,m)||$ by $||B(n,m)||$ and the following inequality, which is proved below: for all $\va >0$ 
$$ ||B(n,m)|| \;||B(m,0)|| \leq C(\va) ||B(n,0)||^{1+\va}.$$
We then obtain, taking $\va = \frac 1{12}\theta \eta $
\begin{eqnarray*}
 ||\,^t B_0(n,m)||\,||B(m,0)||^{1-\frac 12 \eta} &\leq & ||B(n,m)||\;||B(m,0) ||^{1-\frac 12 \eta} \\
 & \leq & C ||B(n,0)||^{1+ \frac 1{12}\theta \eta} ||B(n,0)||^{-\frac 16 \theta \eta} \\
 & \leq & C ||B(n,0)||^{1- \frac 1{12}\theta \eta}.
 \end{eqnarray*}
 To prove that $ ||B(n,m)|| \;||B(m,0)|| \leq C(\va) ||B(n,0)||^{1+\va}$, we may assume that $n\leq \hat n_{-2d+4}$. 
 Then, there exist integers $k$, $l$ with $n \leq \hat n_{k} \leq m \leq \hat n_l$ such that $l=k+2d-3$ so that  all coefficients of $B(\hat n_k, \hat n_l )$ are positive (by the Lemma in $\S$~2.3 of \cite{MMY1}) and hence at least equal to $1$. Furthermore,  by condition (a), there exists $C_1=C_1(\va)$ such that $||B(\hat n_k , \hat n_l)||\leq C_1 ||B(n,0)||^{\frac{ \va}{2}}$. This gives
 \begin{eqnarray*}
||B(n,m)|| \;||B(m,0)|| 
& \leq & C(\va)^2 ||B(n,\hat n_k)|| \;||B(\hat n_l,0)|| \; ||B(n,0)||^{\va}\\ 
  & \leq & C(\va)^2  ||B(n,0)||^{1+ \va},
\end{eqnarray*}
as required.

\smallskip

We conclude that
$$  ||S^\sharp (0,n) \psi||_{C^0} \leq C ||\psi||_{C^{\eta}}\, n \;||B(n,0)||^{1- \frac 1{12}\theta \eta}.$$

Since we also know that $||B(n,0)||$ grows exponentially fast with relation to Zorich time, this is sufficient to conclude that $n ||B(n,0)||^{-\va}$ converges to $0$ for any $\va>0$. This concludes the proof of the proposition.

\end{proof}

\section{Distributional limit shapes}\label{sec:centraldistributions}

Let $\mathcal R$ be a Rauzy class on an alphabet $\A$ of cardinality $d$, $\cD$ be the associated diagram. Let $g$ be the genus of the associated translation surfaces and $s$ the cardinality of the marked set on such surfaces. The rank of the matrices $\Omega_{\pi}$, $\pi \in \mathcal R$ is equal to $2g$ and we have $d = 2g+s-1$ (see $\S$~\ref{sec:transl}).  We assume throughout this section that $s>1$.
%

\smallskip
Consider a point $\underline T=(\pi, \lambda, \tau)$ in $\mathcal R \times (\Rset^+)^{\A} \times \Theta_{\pi}$ which is \emph{typical} for the dynamics of the natural extension 
$\hat Q_{{\rm RV}}$ (defined in \ref{sec:algorithms})  and typical (in the sense of Oseledets genericity) for the Kontsevich-Zorich cocycle. Let $T:= T_{\pi, \lambda}$ be the associated i.e.m. acting on the interval $I$. For $n \in \Zset$, 
let $\underline T^{(n)}=(\pi^{(n)}, \lambda^{(n)}, \tau^{(n)})$ be the image of  $\underline T$ under $\hat Q_{{\rm RV}}^n$, and let $T^{(n)}$ be the associated i.e.m acting on the interval $I^{(n)}$. For any integers $n,n'\in \mathbb{Z}$ with  $n \geq n'$, $T^{(n')}$ is the first return map of $T^{(n)}$ in $I^{(n')}$.

\smallskip

For each $n\in \Zset$, we denote respectively by $\Gamma_s^{(n)}$, $\Gamma_u^{(n)}$, $\Gamma_c^{(n)}$ the stable, unstable and central subspaces of the KZ-cocycle at $(\pi^{(n)}, \lambda^{(n)}, \tau^{(n)})$ (see $\S$~\ref{sec:absoluteRoth}). By hyperbolicity of the Kontsevich-Zorich cocycle (proved in {\Green \cite{For2}}), we have
$$\dim (\Gamma_s^{(n)}) = \dim (\Gamma_u^{(n)})=g, \quad \dim (\Gamma_c^{(n)})=s-1,$$
$$ \Gamma_s^{(n)} \oplus \Gamma_u^{(n)} = {\rm Im} \left(\Omega_{\pi^{(n)}}\right).$$

\smallskip
We will say that $\underline\chi:=(\chi^{(n)})_{n \in \Zset}$ is a sequence of \emph{central functions} if $\chi^{(0)}=\chi \in \Gamma_c^{(0)}$ and
\begin{align*} 
\chi^{(n)} & = 
 S(0,n) \chi ,  &  \text{for}\ n> 0, \\ 
 \chi^{(n)} & = 
 S(n,0)^{-1} \chi , &  \text{for}\ n< 0.
 \end{align*}
A particularly interesting example of central functions can be constructed from characteristic functions as shown in \S~\ref{sec:correctedchi}.

\subsection{Corrected characteristic functions}\label{sec:correctedchi}
Assume only for this subsection that $s=1$. 
%
%
Fix a point $\xi^{(0)} \in I^{(0)}$. 
\begin{definition}\label{def:past}
A sequence $\underline\xi:=(\xi^{(n)})_{n \in \Zset}$ is  \emph{compatible} with $\xi^{(0)} \in I^{(0)}$ if:
\begin{itemize}
\item for each $n \in \Zset$, $\xi^{(n)}$ is a point of $I^{(n)}$;
\item  for $n \geq n'$, $\xi^{(n)}$ is the first entry point in $I^{(n)}$ when iterating $(T^{(n')})^{-1}$ from $\xi^{(n')}$.
\end{itemize}
We will say in this case that $\underline\xi$ is a compatible sequence.
\end{definition}
\begin{remark}\label{rem:compatiblesequences}
It is clear that $\xi^{(n_0)}$ determines $\xi^{(n)}$ for all $n\geq n_0$. On the other hand, given $\xi^{(n)}$ for $n \geq n_0$, one can for instance choose $\xi^{(n)}= \xi^{(n_0)}$ for $n\leq n_0$. 
\end{remark} 
In particular, by Remark \ref{rem:compatiblesequences},  sequences compatible with $\xi^{(0)}$ do exist. For any such sequence, 
for any $n \in \Zset$, let $\wt\chi^{(n)}$ be the function defined on $\iin$ and equal to the characteristic function of the interval $(0,\xi^{(n)})$. 

\smallskip
Let $\Gamma^{(n)}$ be the $d$-dimensional vector space (canonically isomorphic to $\RA$) of functions defined on $\iin$ and constant on each $I_\alpha^{(n)}$. Finally, let $\Gamma^{(n)}(\xi)$ be the $d$-dimensional affine space given by
\[ \Gamma^{(n)}(\xi):= \wt\chi^{(n)} + \Gamma^{(n)}.\]

The following Proposition shows that any sequence ${(\wt\chi^{(n)})}_{n \in \Zset}$ of characteristic functions  as  above 
 can be \emph{corrected} by adding a function constant on the subintervals exchanged by $T^{(0)}$, to produce a sequence $(\chi^{(n)})_{n \in \Zset}$ of central functions:

\begin{proposition}[Projection to central sequences]\label{prop:correction}
For any  sequence $\underline\xi:=(\xi^{(n)})_{n \in \Zset}$ compatible with $\xi^{(0)} \in I^{(0)}$, for any intergers $n,n'$ with $n \geq n'$, the special Birkhoff sum operator $S(n',n)$ (cf.~\S~\ref{ssKZ}) is an isomorphism from $\Gamma^{(n')}$ onto $\Gamma^{(n)}$ and from $\Gamma^{(n')}(\xi)$ onto $\Gamma^{(n)}(\xi)$.

\smallskip
Furtheremore, there exists a unique sequence $(\chi^{(n)})_{n \in \Zset}$ such that $\chi^{(n)}$ belongs to  $\Gamma^{(n)}(\xi)$ for any $n \in \mathbb{Z}$, 
$$S(n',n)(\chi^{(n')})= \chi^{(n)}$$ 
for all $n, n' \in \Zset$ with $n\geq n'$ and
\begin{equation}\label{subexp_growth}
\lim_{n \rightarrow \pm \infty} \frac {1}{Z(n)} \log ||\chi^{(n)}||_{\infty} =0,
\end{equation}
where $Z(n)$ is the Zorich time (see \S~\ref{secRV}).
\end{proposition}

\begin{proof}
The first part of the assertion is already known (see \cite{MMY1}). For the second part, let $n \geq n'$ and consider $S(n',n)\wt\chi^{(n')}$. It is locally constant on $\iin$, except at $\xi^{(n)}$ and at the singularities of $T^{(n)}$. Moreover, the difference between the left and right limits at 
$\xi^{(n)}$ is equal to $1$. Therefore $S(n',n)\wt\chi^{(n')}$ belongs to $\Gamma^{(n)}(\xi)$ and the first part of the proposition follows.
\item Let $\Gamma^{(n)} = \Gamma_s^{(n)} \oplus \Gamma_u^{(n)}$ be the Oseledets decomposition into stable and unstable subspaces. 
Write then for $n \in \Zset$
$$S(n-1,n)\wt\chi^{(n-1)} = \wt\chi^{(n)} + \Delta \wt\chi_s^{(n)} + \Delta \wt\chi_u^{(n)}$$
with $\Delta \wt\chi_s^{(n)} \in \Gamma_s^{(n)}$ and $\Delta \wt\chi_u^{(n)} \in \Gamma_u^{(n)}$. Observe that, since $\wt\chi^{(n)}$ is bounded by $1$ and $n$ here indexes Rauzy-Veech elementary steps,   we have that $||S(n-1,n)\wt\chi^{(n-1)}||_{\infty} \leq 2$.

\smallskip

We look for $\chi^{(n)} = \wt\chi^{(n)}+ \Delta \chi_s^{(n)} + \Delta \chi_u^{(n)}$ satisfying the requirements of the Proposition. In particular, we must have
$$\Delta \chi_s^{(n)}=\Delta \wt \chi_s^{(n)}+ S(n-1,n)\Delta \chi_s^{(n-1)},$$ 
$$\Delta \chi_u^{(n)}=\Delta \wt \chi_u^{(n)}+ S(n-1,n)\Delta \chi_u^{(n-1)}.$$ 
This leads to the formulas
$$ \Delta \chi_s^{(n)} = \sum_{n' \leq n} S(n',n)\Delta \wt \chi_s^{(n')},$$ 
$$ \Delta \chi_u^{(n)} = -\sum_{n' > n} (S(n,n'))^{-1}\Delta \wt \chi_u^{(n')}.$$
We will now show that the convergence of these series follows from the hyperbolicity of the  Kontsevich-Zorich cocycle and this will conclude the proof. Let us first remark that, as $\underline T=(\pi, \lambda, \tau)$ is typical
and, as already observed, $||S(n-1,n)\wt\chi^{(n-1)}||_{\infty} \leq 2$, there is, for every $\va >0$, a constant $C:=C(\va, \underline T)$ such that, for all $n \in \Zset$
$$|| \Delta \wt \chi_s^{(n)}||_{\infty} + || \Delta \wt \chi_u^{(n)}||_{\infty} \leq C \exp (\va |Z(n)|).$$
As the cocycle is hyperbolic, there exists a constant $C':=C'(\va, \underline T)$ such that, for all $n \geq m$, all $v_s$ in $\Gamma_s^{(m)}$, all 
$v_u$ in $\Gamma_u^{(n)}$, denoting by $\theta_g$ the smallest positive exponent of the KZ-cocycle w.r.t the Zorich dynamics $\hat Q_{{\rm Z}}$, we have
$$||S(m,n) v_s||_{\infty} \leq C' \exp [(Z(m)-Z(n))\theta_g + (|Z(m)|+|Z(n)|)\va] ||v_s||_{\infty},$$
$$||(S(m,n))^{-1} v_u||_{\infty} \leq C' \exp [(Z(m)-Z(n))\theta_g + (|Z(m)|+|Z(n)|)\va] ||v_u||_{\infty}.$$
These equations show that the series above converge.
%
\smallskip
As $\va$ is arbitrary, the sequence $\chi^{(n)}$ is well-defined and satisfies the required properties. 

\smallskip
The uniqueness of the sequence with the required properties follows from the hyperbolicity of the KZ-cocycle: the difference between two solutions would grow subexponentially (in Zorich time)
and thus must be equal to $0$.
\end{proof}

%
\begin{remark}
This Proposition uses the assumption, assumed throughout this section, that $\underline T=(\pi, \lambda, \tau)$  is \emph{typical} for the dynamics of 
$\hat Q_{{\rm RV}}$ and Oseledets generic for the Kontsevich-Zorich cocycle. In $\S$~\ref{sec:homological} we will prove a similar result in the homological context under weaker assumptions on $(\pi, \lambda, \tau)$ than Oseledets genericity (see in particular the notion of KZ hyperbolic introduced in \S~\ref{ssKZhyp}). Furthermore, one can prove the existence of a sequence  $\Gamma^{(n)}(\xi)$  
as in the conclusion of Proposition \ref{prop:correction} but satisfying  \eqref{subexp_growth} only for $n \to + \infty$ (resp.~$n \to - \infty$) assuming only negative (resp.~positive) KZ hyperbolicity (see, in $\S$~\ref{ssKZhyp}, Propositions \ref{propKZpos} and \ref{propKZneg} respectively). In this case though uniqueness only holds modulo $\Gamma_s^{(n)}$ (resp.~$\Gamma_u^{(n)}$).
\end{remark}



\subsection{The functions $\Omega_{\alpha}^{(n)}(\pi,\tau,\chi)$}\label{sec:functions}
We will now define an object analogous to the \emph{limit shapes} defined in  \cite{MMY3},  but suited to study Birkhoff sums of central functions $\chi \in \Gamma_c^{(0)}$.  

\smallskip
The largest Lyapunov exponent of the Kontsevich-Zorich cocycle is simple and the associated $1$-dimensional eigenspace $F_1(\pi, \tau)$ does not depend on the $\lambda$ variable. Moreover, $F_1(\pi, \tau)$ is generated by a vector $q(\pi, \tau)$ contained in the positive cone  $(\Rset^+)^{\A}$.  
We normalize $q(\pi, \tau)$ by asking that its $\ell^2$-norm is equal to $1$.
For $n\leq 0$, we write
$$ (S(n,0))^{-1} q(\pi,\tau)=:q^{(n)}(\pi,\tau) = \Theta_1^{(n)} q(\pi^{(n)}, \tau^{(n)})$$
with $\Theta_1^{(n)} \in \Rset ^+$ satisfying
$$\lim_{n \rightarrow -\infty} \frac{1}{Z(n)} \log \Theta_1^{(n)} = \theta_1.$$
($\theta_1$ is the largest Lyapunov exponent of the KZ-cocycle.)


\smallskip 
Let $\chi \in \Gamma_c^{(0)}$.  Write $\chi^{(n)} := (S(n,0))^{-1}(\chi)$ for $n\leq 0$. For any $\alpha \in \A$ and any $n\leq 0$ , we define  a function $\Omega_{\alpha}^{(n)} =\Omega_{\alpha}^{(n)}(\pi,\tau,\chi)$ on the interval $[0,q_{\alpha}(\pi,\tau)]$ by the following requirements:
\begin{enumerate}
\item recall that $ B_\alpha (n,0)$ (defined in \eqref{columsum}) is the return time in $I^{(0)}$ of $I_{\alpha}^{(0)}$ under $T^{(n)}$; for $0 \leq j < B_\alpha(n,0)$, let $\beta(\alpha,j)$ be the letter in $\A$ such that $(T^{(n)})^j (I_{\alpha}^{(0)}) \subset I_{\beta(\alpha,j)}^{(n)}$; for $0 \leq \ell \leq B_\alpha(n,0)$, set
$$ Sq(\ell) = \sum_{0\leq j < \ell} q_{\beta(\alpha,j)}^{(n)}(\pi, \tau). $$
Observe that $Sq(0) = 0$ and $Sq(B_\alpha(n,0)) = q_{\alpha}(\pi,\tau)$. The function $\Omega_{\alpha}^{(n)}$
is continuous on $[0,q_{\alpha}(\pi,\tau)]$ and its restriction to each interval $[Sq(\ell), Sq(\ell +1) ] $ is affine.
\item For $0 \leq \ell < B_\alpha(n,0)$, one has
$$\Omega_{\alpha}^{(n)}(Sq(\ell +1)) - \Omega_{\alpha}^{(n)}(Sq(\ell))= \chi^{(n)}( (T^{(n)})^{\ell} (I_{\alpha}^{(0)}) )$$
\item The function $\Omega_{\alpha}^{(n)}$ is defined by the first two conditions up to the addition of a constant function. This constant is now determined by the requirement that $\Omega_{\alpha}^{(n)}$ should have mean $0$ on $[0,q_{\alpha}(\pi,\tau)]$.
\end{enumerate}

\begin{figure}
   \begin{subfigure}{0.45\textwidth}
    \includegraphics[width=\textwidth]{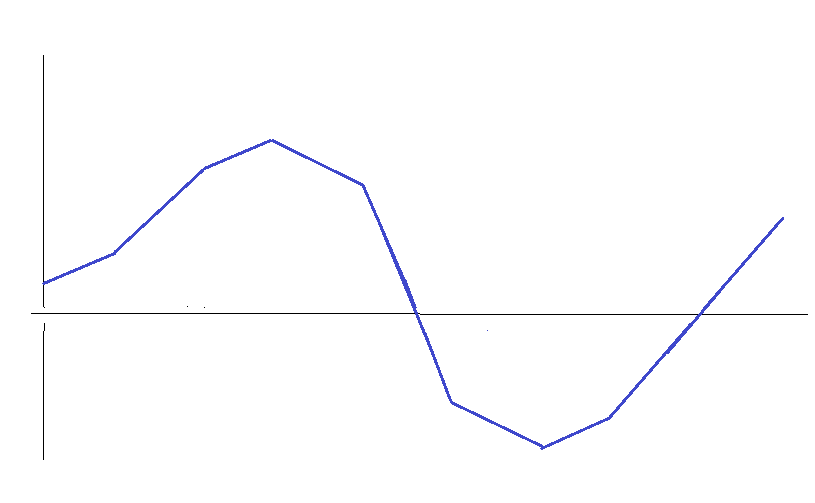}
    \caption{$\qquad \Omega_{\alpha}^{(n)}(\pi,\tau,\chi)$  \label{fig:limitshape1}}
   
  \end{subfigure}
 \vspace{12mm}
  \begin{subfigure}{0.45\textwidth}
    \includegraphics[width=\textwidth]{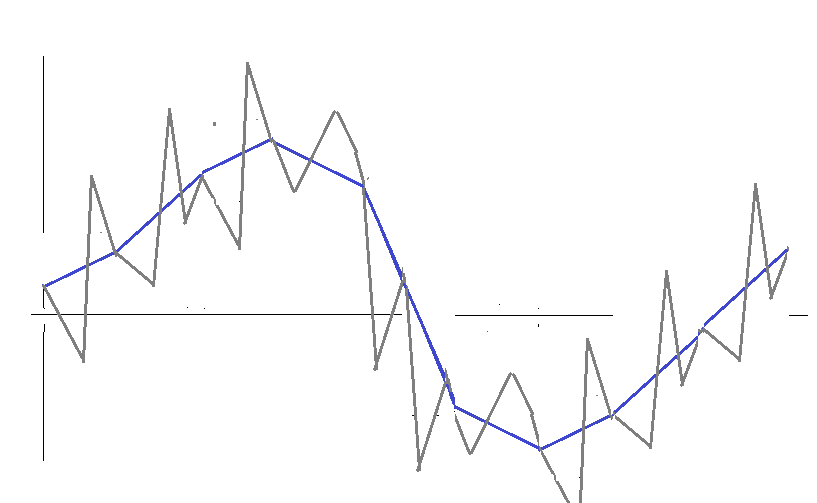}
    \caption{$\qquad \Omega_{\alpha}^{(n')}(\pi,\tau,\chi)$ \label{fig:limitshape2}}
  \end{subfigure}
	\caption{For $n' \leq n \leq 0$,  $\quad \Omega_{\alpha}^{(n)}(\pi,\tau,\chi)\ $ refines $\quad \Omega_{\alpha}^{(n')}(\pi,\tau,\chi)$.\label{fig:comparison}}
\end{figure}

\subsection{Comparing the functions $\Omega_{\alpha}^{(n)}(\pi,\tau,\chi)$ and $\Omega_{\alpha}^{(n')}(\pi,\tau,\chi)$}\label{sec:comparingfunctions}

\smallskip

Let $n' \leq n \leq 0$. We will compare the functions $\Omega_{\alpha}^{(n)}(\pi,\tau,\chi)$ and $\Omega_{\alpha}^{(n')}(\pi,\tau,\chi)$, which are defined on the same interval $[0,q_{\alpha}(\pi,\tau)]$.

\begin{lemma}\label{comparison}
 For every $\alpha \in \A$ and every $0 \leq \ell \leq B_\alpha(n,0)$, we have
$$\Omega_{\alpha}^{(n')}(Sq(\ell)) -\Omega_{\alpha}^{(n)}(Sq(\ell))= \Omega_{\alpha}^{(n')}(0) -\Omega_{\alpha}^{(n)}(0).$$

\smallskip
Moreover, for every $\alpha \in \A$ and every $0 \leq \ell < B_\alpha(n,0)$, the restriction of $\Omega_{\alpha}^{(n')}$ to  $[Sq(\ell), Sq(\ell +1) ] $  for 
$$0 \leq x \leq Sq(\ell +1) -Sq(\ell)= q_{\beta(\alpha,\ell)}^{(n)}(\pi, \tau)= \Theta_1^{(n)} q_{\beta(\alpha,\ell)}(\pi^{(n)}, \tau^{(n)})$$
satisfies 
$$ \Omega_{\alpha}^{(n')}(Sq(\ell) +x) = c(\ell) + \Omega_{\beta(\alpha, \ell)}^{(n'-n)}(\pi^{(n)}, \tau^{(n)}, \chi^{(n)})\left(\frac {x}{\Theta_1^{(n)}}\right),$$
where $c(\ell)$ is the mean value of $\Omega_{\alpha}^{(n')}$ on $[Sq(\ell), Sq(\ell +1) ] $.
\end{lemma}

\begin{proof}
The statements of the Proposition follow directly from the definition of the functions $\Omega_{\alpha}^{(n')}, \; \Omega_{\alpha}^{(n)}$ and from the relations
$$S(n',n)q^{(n')}(\pi,\tau) = q^{(n)}(\pi,\tau)= \Theta_1^{(n)} q(\pi^{(n)}, \tau^{(n)}), \quad S(n',n)\chi^{(n')} = \chi^{(n)}.$$
\end{proof}

\begin{figure}
   \begin{subfigure}{0.4\textwidth}
    \includegraphics[width=\textwidth]{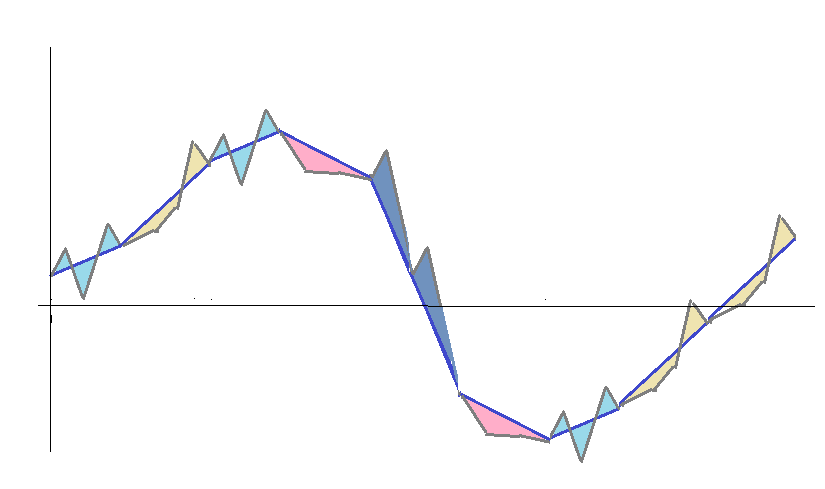}
    \caption{Copies of $\Omega_{\beta}^{(n'-n)}(\underline{T}^{(n)})$ in the difference (where $\underline{T}^{(n)}:= (\pi^{(n)}, \tau^{(n)}, \chi^{(n)})$) in different colours for different $\beta \in A$. \label{fig:difference1}}
   
  \end{subfigure}
 \hspace{20mm}
  \begin{subfigure}{0.4\textwidth}
    \includegraphics[width=\textwidth]{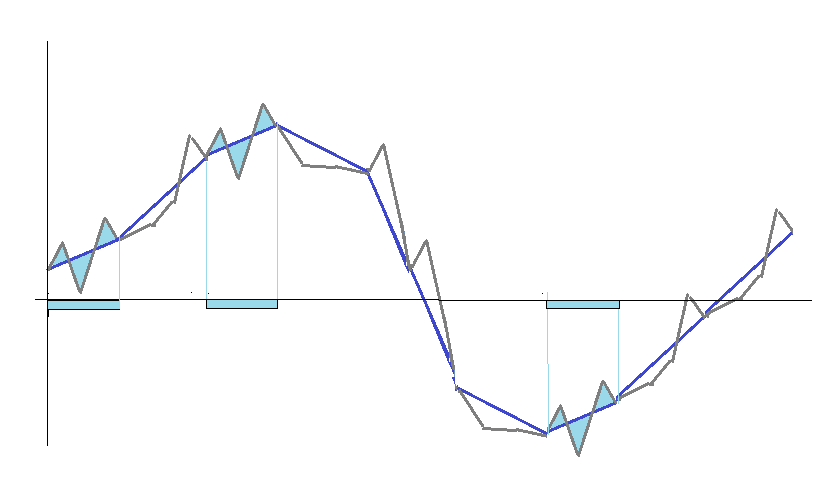}
    \caption{Occurrences of $\Omega_{\beta}^{(n'-n)}(\underline{T}^{(n)})$ with $\beta \in \A$ fixed (where $\underline{T}^{(n)}:= (\pi^{(n)}, \tau^{(n)}, \chi^{(n)})$) correspond to a dual Birkhoff sum.\label{fig:difference2}}
    
  \end{subfigure}
	\caption{The difference $\, \Omega_{\alpha}^{(n)}(\pi,\tau,\chi) - \Omega_{\alpha}^{(n')}(\pi,\tau,\chi)$, where $n' \leq n \leq 0$.\label{fig:difference}}
\end{figure}

\subsection{Distributional convergence of the sequence $\Omega_{\alpha}^{(n)}(\pi,\tau,\uvec{\xi})$ as $n \rightarrow - \infty$
}
\label{sec:convergence}

\smallskip
We saw in the previous section that, for  $n'<n$, the function $\Omega_{\alpha}^{(n')}$ refines a (translate of) the function $\Omega_{\alpha}^{(n)}$ (as illustrated in Figure~\ref{fig:comparison}).  
One can show that the sequence of functions $\Omega_{\alpha}^{(n)}$, as $n \to -\infty$, does not converge pointwise. Nevertheless, we will show that (under the dual Roth type condition) there exists a limit object, in the sense that the sequence converge in the space of distributions acting on H\"older functions. This is the content of the following result. 

\begin{theorem}\label{thm:convergence}
Assume that $(\pi, \tau)$ is of dual Roth type. For $\alpha \in \A$, let $\psi$ be a mean zero function defined  on the interval $[0,q_{\alpha}(\pi,\tau)]$. Assume that $\psi$ is H\"older continuous for some exponent $\eta \in (0,1)$. Then the sequence 
$$ \int_0^{q_{\alpha}(\pi,\tau)} \psi(x)\ \Omega_{\alpha}^{(n)}(x) dx$$
is convergent as $n \rightarrow - \infty$ . More precisely, there exists a  $\delta = \delta(\eta)$ such that, for $n' \leq n \leq 0$ 
$$  | \int_0^{q_{\alpha}(\pi,\tau)} \psi(x) (\Omega_{\alpha}^{(n')}(x)-\Omega_{\alpha}^{(n)}(x)) dx|< C(\eta) ||\psi||_{C^{\eta}}||B(n,0)||^{-\frac{\delta}4}. $$
\end{theorem}
\begin{proof}
Fix $n' \leq n \leq 0$. 
It is sufficient to prove the inequality of the proposition under the hypothesis
$$ ||B(n',n)|| \leq C ||B(n,0)||^{\frac {\delta}4}.$$
By decomposing the interval of integration into intervals of the form $[ Sq(\ell), Sq(\ell +1)]$ and  
$$
\int_0^{q_{\alpha}(\pi,\tau)} \psi(x) (\Omega_{\alpha}^{(n')}(x)-\Omega_{\alpha}^{(n)}(x)) dx = \sum_{0 \leq \ell < B_{\alpha}(n,0)} \int_{Sq(\ell)}^{Sq(\ell +1)} \psi(x) (\Omega_{\alpha}^{(n')}(x)-\Omega_{\alpha}^{(n)}(x)) dx $$
 $$= \sum_{\overline \beta \in \A} \sum_{0 \leq \ell < Q_{\alpha}^{(n)}, \overline \beta(\alpha, \ell) = \overline \beta} \int_{Sq(\ell)}^{Sq(\ell +1)} \psi(x) (\Omega_{\alpha}^{(n')}(x)-\Omega_{\alpha}^{(n)}(x)) dx $$
 $$= \sum_{\overline \beta \in \A} \int_0 ^{q^{(n)}_{\overline\beta}(\pi,\tau)} \widetilde\Omega^*_{\overline \beta}(\frac {x}{\Theta_1^{(n)}})
 \sum_{0 \leq \ell < Q_{\alpha}^{(n)}, \overline \beta(\alpha, \ell) = \overline \beta} \psi(x + Sq(\ell)) dx. $$
 
 Here, $\widetilde\Omega^*_{\overline \beta}$ is according to Lemma \ref{comparison}, the function on $[0,q_{\overline\beta}(\pi^{(n)}, \tau^{(n)})]$ which differs from $\Omega_{\overline\beta}^{(n'-n)}(\pi^{(n)}, \tau^{(n)}, \xi^{(n)} )$ by an affine function and satisfies
 $$ \widetilde\Omega^*_{\overline \beta} (0)=\widetilde\Omega^*_{\overline \beta}(q_{\overline\beta}(\pi^{(n)}, \tau^{(n)})) = \Omega_{\alpha}^{(n')}(0)-\Omega_{\alpha}^{(n)}(0).$$

 Let us consider $\psi$, which is defined on $L^{(0)}_\alpha=[0,q_\alpha]$, as a function on $L^{(0)}= \bigsqcup_{\beta \in \A} L^{(0)}_\beta$ which is zero on each $L^{(0)}_\beta$ for $\beta \neq \alpha$ and remark that $\psi$ is H\"older on the disjoint union $L^{(n)}$.  
Since we are assuming that $(\pi, \tau)$ is of dual Roth type, by the estimates of dual Birkhoff sums given in Proposition \ref{estimatedualBS}, we hence have that, for $x \in L^{(0)}_\alpha$, 
\begin{align*}
 \left| \sum_{0 \leq \ell < Q_{\alpha}^{(n)}, \overline \beta(\alpha, \ell) \overline \beta} \psi(x + Sq(\ell))\right|  =
 \left| \sum_{\alpha \in \A} \sum_{0 \leq \ell < Q_{\alpha}^{(n)}, \overline \beta(\alpha, \ell) \overline \beta} \psi(x + Sq(\ell))\right| 
 & = \left| S^\sharp(0,n) \psi(x) \right| \\ & 
 \leq C(\eta) ||\psi||_{C^{\eta}}||B(n,0)||^{1-\delta},
\end{align*} 
for some $\delta>0$ depending on $\eta$. 
It follows that
\begin{equation}\label{difference}
 |\int_0^{q_{\alpha}(\pi,\tau)} \psi(x)  (\Omega_{\alpha}^{(n')}(x)-\Omega_{\alpha}^{(n)}(x)) dx| 
    \leq C(\eta) ||\psi||_{C^{\eta}}||B(n,0)||^{1-\delta}
\sum_{\overline \beta \in \A} \int_0 ^{q^{(n)}_{\overline\beta}(\pi,\tau)} \left|\widetilde\Omega^*_{\overline \beta}\left(\frac {x}{\Theta_1^{(n)}}\right)\right|dx.
\end{equation} 
Changing integration variable,  we have
$$\int_0 ^{q^{(n)}_{\overline\beta}(\pi,\tau)} \left|\widetilde\Omega^*_{\overline \beta}\left(\frac {x}{\Theta_1^{(n)}}\right)\right|dx= 
\Theta_1^{(n)} \int_0 ^{q_{\overline\beta}(\pi^{(n)},\tau^{(n)})} |\widetilde\Omega^*_{\overline \beta}( x)|dx.$$
Since we also have that
$$ \Theta_1^{(n)} \leq C ||B(n,0)||^{-1+\frac 14 \delta},$$
and also that
$$ \left|\Omega_{\overline\beta}^{(n'-n)}(\pi^{(n)}, \tau^{(n)}, \xi^{(n)})\right|_{C^0}  \leq C ||B(n',n)|| ||\chi^{(n')}||_{\infty} ,
$$
with $||B(n',n)|| \leq C ||B(n,0)||^{\frac {\delta}4}$ by hypothesis and
$$ ||\chi^{(n')}||_{\infty} \leq C||B(n',0)||^{\frac {\delta}8} \leq C (||B(n',n)||\;\;||B(n,0)||)^{\frac {\delta}8} \leq C ||B(n,0)||^{\frac {\delta}4},$$
this gives that 
\[
\sum_{\overline \beta \in \A} \int_0 ^{q^{(n)}_{\overline\beta}(\pi,\tau)} \left|\widetilde\Omega^*_{\overline \beta}\left(\frac {x}{\Theta_1^{(n)}}\right)\right|dx \leq C  ||B(n,0)||^{-1+  \frac {3}{4}\delta}.
\]
Combining this with \eqref{difference} gives the required inequality.
\end{proof}

\section{Homological interpretation}\label{sec:homological}


In this section we give a different intepretation of the distributional limit shapes introduced in the previous section and we relate them to (relative) homology classes and sections of the boundary operator.  
In \S~\ref{ssHom} we  recall some homology classes in $H_1(M, \Sigma, \Zset)$ and $H_1(M \setminus \Sigma, \Zset)$ associated to the Rauzy-Veech algorithm.  In \S~\ref{sec:piecewise_affine_paths} we describe in the homological context a construction which is very similar to the constuction of central functions $(\xi^{(n)})_{n \in \Zset}$ correcting the {\Green characteristic} functions $(\widetilde{\xi}^{(n)})_{n \in \Zset}$ done in \S~\ref{sec:correctedchi}. There are however several differences ({\Green besides} the context). In particular, we use weaker assumptions than the typical Oseledets hyperbolic behaviour for the KZ-cocycle assumed in \S~\ref{sec:correctedchi}. In particular, we make separate assumptions on the backward and forward rotation numbers,  called respectively positive KZ-hyperbolicity (for future behaviour) and negative KZ-hyperbolicity (for past behaviour).   One of the points of introducing such assumptions is that there are explicit cases where it might be possible to check  these assumptions while Oseledets hyperbolic behaviour might be false or very difficult to check.  
Finally, in \S~\ref{ssOmega} and \S~\ref{sec:piecewise_affine_paths} we construct the distributional analogue of limit shapes in this homological context. 

%

\subsection{Bases of homology associated to the Rauzy-Veech algorithm}\label{ssHom}
Let us define bases for the punctured surface homology $H_1(M \setminus \Sigma, \Zset)$ and the relative homology $H_1(M, \Sigma,\Zset)$ naturally induced by the zippered rectangle presentation of a translation surface. The setting and notations here used are those of \S~\ref{ssRV}. We also refer the reader to the lecture notes \cite{Y4} (see in particular Section 4.5 in \cite{Y4}) for more details on the homological interpretation of the bases induced by zippered rectangles. 

\smallskip

\begin{figure}
 \begin{subfigure}{0.45\textwidth}
    \includegraphics[width=.7\textwidth]{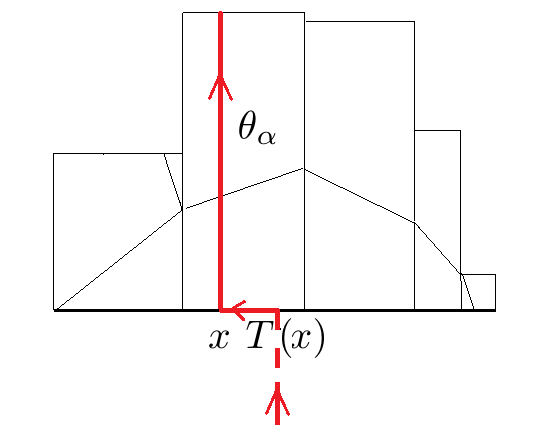}
    \caption{ The classes $\theta_{\alpha} \in H_1(M \setminus \Sigma, \Zset)$.\label{fig:theta}}
  \end{subfigure} \hspace{6mm}
  \begin{subfigure}{0.45\textwidth}
    \includegraphics[width=.7\textwidth]{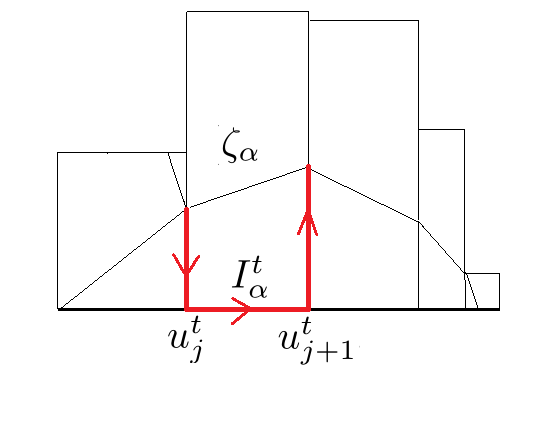}
    \caption{The class  $\zeta_{\alpha} \in H_1(M, \Sigma,\Zset)$.\label{fig:zeta}}

  \end{subfigure}
	\caption{The bases for $H_1(M \setminus \Sigma, \Zset)$ and $ H_1(M, \Sigma,\Zset)$ described in $\S$~\ref{ssHom}.\label{bases}}
\end{figure}

\noindent {\bf  The classes $\theta_{\alpha} \in H_1(M \setminus \Sigma, \Zset)$.} 

Let $\alpha \in \A$, $x \in \iat$. Consider the simple loop in $M \setminus \Sigma$ consisting of the upwards vertical segment (of length $q_{\alpha}$) joining $x$ to $T(x)$ followed by the horizontal segment contained in $I$ joining $T(x)$ to $x$ (as shown in Figure~\ref{fig:theta}). The homology class in $H_1(M \setminus \Sigma, \Zset)$ of this loop does not depend on $x$ and is denoted by $\theta_{\alpha}$. The classes $(\theta_{\alpha})_{\alpha \in \A}$ form a basis of $H_1(M \setminus \Sigma, \Zset)$. One has
$$ \int_{\theta_{\alpha}} \omega = -\delta_{\alpha} + i q_{\alpha}.$$


\medskip

\noindent {\bf The classes $\zeta_{\alpha} \in H_1(M, \Sigma,\Zset)$.} 

Let $\alpha \in \A$. Write $\iat = (u_j^t, u_{j+1}^t)$, with the convention that $u_0^t=u_0^b=u_0$, $u_d^t = u_d^b = u_d$. Consider the following path in $M$  (shown in Figure~\ref{fig:zeta}) with endpoints in $\Sigma$: the first part is a vertical separatrix segment from a point of $\Sigma$ to $u_j^t$ (this part is reduced to a point if $j=0$) , the second part is the horizontal segment $\iat \subset I$ from $u_j^t$ to $u_{j+1}^t$, the third part is a vertical separatrix segment from $u_{j+1}^t$ to a point of $\Sigma$. The homology class of this path in $H_1(M,\Sigma,\Zset)$ is denoted by $\zeta_{\alpha}$. The classes $(\zeta_{\alpha})_{\alpha \in \A}$ form a basis of $H_1(M, \Sigma,\Zset)$. One has 

$$ \int_{\zeta_{\alpha}} \omega = \lambda_{\alpha} + i \tau_{\alpha}.$$

Write $\iab = (u_{\ell}^b,u_{\ell +1}^b)$. An alternative path representing $\zeta_{\alpha}$ is made of a vertical separatrix segment connecting a point of $\Sigma$ to $u_{\ell}^b$, the horizontal segment $\iab$ connecting $u_{\ell}^b$ to $u_{\ell +1}^b$, and a vertical separatrix segment connecting $u_{\ell +1}^b$ to a point of $\Sigma$.

\medskip

{\bf Duality and relation with $\Omega_{\pi}$.} The intersection form $i(\cdot,\cdot)$ defines a non degenerate bilinear pairing between
 $H_1(M, \Sigma,\Zset)$ and $H_1(M \setminus \Sigma, \Zset)$ for which the bases   $(\zeta_{\alpha})_{\alpha \in \A}$ and $(\theta_{\alpha})_{\alpha \in \A}$ are dual: one has 
$$ i( \zeta_{\alpha}, \theta_{\beta}) = \delta_{\alpha\, \beta},$$
where $\delta_{\alpha\, \beta}$ is the Kronecker symbol.

\smallskip
Consider the homomorphism $\iota $ from $H_1(M \setminus \Sigma, \Zset)$ to $H_1(M, \Sigma,\Zset)$ obtained by composing the natural homomorphisms from $H_1(M \setminus \Sigma, \Zset)$ to $H_1(M,\Zset)$ and from $H_1(M,\Zset)$ to $H_1(M, \Sigma,\Zset)$. For $\alpha, \beta \in \A$, one has

$$ \iota(\theta_{\alpha}) = - \sum_{\beta}  \Omega_{\alpha \, \beta} \; \zeta_{\beta}.$$

\subsection{Change of bases associated to the Rauzy-Veech algorithm}

Let $(T^{(n)})_{n \in \Zset}$ be the sequence of i.e.m deduced from $(\pi, \lambda,\tau)$ by the Rauzy-Veech algorithm. Notice that all translation surfaces $M(\pi^{(n)}, \lambda^{(n)}, \tau^{(n)})$ are canonically isomorphic
to $M:=M^{(0)}$. Therefore, for each $n \in \Zset$, the previous constructions define a basis $(\theta_\alpha^{(n)})_{\alpha \in \A}$ of $H_1(M \setminus \Sigma, \Zset)$ and a basis $(\zeta^{(n)}_{\alpha})_{\alpha \in \A}$ of $H_1(M, \Sigma,\Zset)$.

\smallskip

These bases are related, for $m \leq n$,  by
\begin{align*}
\theta_\alpha^{(n)}& = \sum_{\beta} B_{\alpha\, \beta}\; \theta^{(m)}_\beta,    & \textrm{where}\ & B = B(m,n)\\ 
\zeta_\alpha^{(n)}& = \sum_{\beta} B^*_{\alpha\, \beta} \;\zeta^{(m)}_\beta,    & \textrm{where}\ & B = B^* = ^t\negmedspace B(m,n)^{-1}.
\end{align*}
Let $\varUpsilon $ be a homology class in $H_1(M, \Sigma,\Rset)$. For $n \in \Zset$, write
$$ \vU = \sum_\alpha \chi_{\alpha}^{(n)} \zeta_\alpha^{(n)}.$$
The coefficients $\chi^{(n)} = (\chi_{\alpha}^{(n)})_{\alpha \in \A}$ are related, for $m \leq n$, by
$$ \chi^{(n)} = B(m,n) \chi^{(m)},$$
which is the usual homological interpretation of the Kontzevich-Zorich cocycle. Notice that 
$$ \chi_{\alpha}^{(n)} = i(\vU, \theta_\alpha^{(n)}).$$

\subsection{KZ-hyperbolic translation surfaces}\label{ssKZhyp}

%

Consider the homology exact sequence
$$ 0 \lr  H_1(M,\Zset) \lr H_1(M, \Sigma,\Zset) \lr H_0(\Sigma,\Zset) \lr \Zset \lr 0,$$
and the corresponding sequence with real coefficients
$$ 0 \lr  H_1(M,\Rset) \lr H_1(M, \Sigma,\Rset) \lr H_0(\Sigma,\Rset) \lr \Rset \lr 0.$$

We identify $H_0(\Sigma,\Rset)$ with $\RS$. Then the image of $H_1(M, \Sigma,\Rset)$ is the hyperplane $\RSO$ consisting of vectors with vanishing coordinate sum. Under an appropriate hypothesis on $(\pi,\lambda,\tau)$, we will construct a preferred section $\nu: \RSO \to H_1(M, \Sigma,\Rset)$ of the boundary map $\partial: H_1(M, \Sigma,\Rset) \to H_0(\Sigma,\Rset)$.

\smallskip
In the following definition, we use the notation of the end of \S~\ref{ssHom}.

\begin{definition}
Assume that $M$ has no vertical connection. The {\it stable subspace} of \linebreak $H_1(M, \Sigma,\Rset)$, denoted by $\Gamma^s$, is 
$$ \Gamma^s = \{ \vU \in H_1(M, \Sigma,\Rset) \, \vert \, \exists \sigma >0, C>0, \Vert \chi^{(n)} \Vert \leq C \Vert B(0,n) \Vert^{-\sigma} \Vert \chi^{(0)} \Vert,\; \forall n>0\}. $$
Assume that $M$ has no horizontal connection. The {\it unstable subspace} of $H_1(M, \Sigma,\Rset)$, denoted by $\Gamma^u$, is 
$$ \Gamma^u = \{ \vU \in H_1(M, \Sigma,\Rset) \, \vert \, \exists \sigma >0, C>0, \Vert \chi^{(n)} \Vert \leq C \Vert B(n,0) \Vert^{-\sigma} \Vert \chi^{(0)} \Vert,\; \forall n<0\}. $$
\end{definition}

Notice that both $\Gamma^s$ and $\Gamma^u$ are subspaces of $H_1(M,\Rset)$ (more precisely of the image of $H_1(M,\Rset)$ in $H_1(M, \Sigma,\Rset)$). Moreover, when $H_1(M,\Rset)$ is equipped with the symplectic intersection form, both $\Gamma^s$ and $\Gamma^u$ are isotropic. Notice also that, as $H_1(M, \Sigma,\Rset)$ is finite-dimensional, one can choose an exponent $\sigma >0$ which works for all vectors in $\Gamma^s$ and all vectors in $\Gamma^u$.

When we identify $H_1(M, \Sigma,\Rset)$ with $\RA$ through the basis $(\zeta_{\alpha}^{(0)})$, The subspace $\Gamma^s$ only depend on $(\pi, \lambda)$ and the subspace $\Gamma^u$ on $(\pi,\tau)$.

\begin{definition}\label{defKZposhyp}
The translation surface $M = M(\pi, \lambda, \tau)$ is {\it positively  KZ-hyperbolic} if
$M$ has no vertical connection and the following properties hold
\begin{enumerate}
\item for any $\tau >0$, one has 
$$ \sum_{n>0} \Vert B(0,n) \Vert^{-\tau} < + \infty.$$
\item
the dimension of $\Gamma^s$ is equal to $g$; 
\item  (Coherence) For any $\tau >0$, there exists $C_\tau >0$ such that,
for any $\vU \in \Gamma^s$, any $0 \leq m \leq n$, one has 
$$ \Vert \chi^{(n)} \Vert \leq C_\tau \Vert B(0,n) \Vert^{\tau} \Vert \chi^{(m)} \Vert;$$
\end{enumerate}

\end{definition}

\begin{definition}\label{defKZneghyp}
The translation surface $M = M(\pi, \lambda, \tau)$ is {\it negatively  KZ-hyperbolic} if
$M$ has no horizontal connection and the following properties hold
\begin{enumerate}
\item for any $\tau >0$, one has 
$$ \sum_{n<0} \Vert B(0,n) \Vert^{-\tau} < + \infty.$$
\item
the dimension of $\Gamma^u$ is equal to $g$; 
\item  (Coherence) For any $\tau >0$, there exists $C_\tau >0$ such that,
  for any $\vU \in \Gamma^u$, any $0 \geq m \geq n$, one has 
$$ \Vert \chi^{(n)} \Vert \leq C_\tau \Vert B(0,n) \Vert^{\tau} \Vert \chi^{(m)} \Vert.$$
\end{enumerate}

\end{definition}

\begin{definition}\label{defKZhyp}
The translation surface $M = M(\pi, \lambda, \tau)$ is {\it  KZ-hyperbolic} if it is both positively and negatively KZ-hyperbolic and one has 
$$ H_1(M,\Rset) = \Gamma^s \oplus \Gamma^u.$$
\end{definition}
%

\begin{remark}\label{remKZ0}
By Forni's results in {\Green \cite{For2}}, for any combinatorial data $\pi$, almost all data $(\lambda,\tau)$ produce KZ-hyperbolic translation surfaces.
\end{remark}


\begin{remark}\label{remKZ1}
Assume that $M$ is positively KZ-hyperbolic. For $\vU \in H_1(M, \Sigma, \Rset)$ and $n \geq 0$, define 
$$ \Vert \vU \Vert_{n,\star s} := \min \{ \, \Vert \chi \Vert_{\infty}\,  \vert \, \quad  \vU - \sum \chi_\alpha \zeta_\alpha^{(n)} \in \Gamma^s\}.$$

It follows from the symplecticity of the restricted KZ-cocycle that one has, for $\vU \in H_1(M,\Rset)$, $n \geq m \geq 0$, and any $\tau >0$
\begin{eqnarray*}
 \Vert \vU \Vert_{0,\star s} &\leq& C \Vert B(0,n) \Vert^{-\sigma} \Vert \vU \Vert_{n, \star s} ;\\
 \Vert \vU \Vert_{m,\star s} &\leq& C_\tau \Vert B(0,n) \Vert^{\tau} \Vert \vU\Vert_{n,\star s}.
\end{eqnarray*}
{\Green The proof is an adaptation of the results discussed in Section 3.2 of  \cite{MY}}. 
\end{remark}

\begin{remark}\label{remKZ2}
Similarly, assume that $M$ is negatively KZ-hyperbolic. For $\vU \in H_1(M, \Sigma, \Rset)$ and $n \leq 0$, define 
$$ \Vert \vU \Vert_{n,\star u} := \min \{\,  \Vert \chi \Vert_{\infty}\, \vert \, \quad  \vU - \sum \chi_\alpha \zeta_\alpha^{(n)} \in \Gamma^u\}.$$

Then one has, for $\vU \in H_1(M,\Rset)$, $n \leq m \leq 0$, and any $\tau >0$
\begin{eqnarray*}
 \Vert \vU \Vert_{0,\star u} &\leq& C \Vert B(n,0) \Vert^{-\sigma} \Vert \vU \Vert_{n, \star u} ;\\
 \Vert \vU \Vert_{m,\star u} &\leq& C_\tau \Vert B(n,0) \Vert^{\tau} \Vert \vU\Vert_{n,\star u}.
\end{eqnarray*}
\end{remark}

\begin{proposition}\label{propKZpos}
Assume that $M$ is positively KZ-hyperbolic. Then, for every $\upsilon \in  \RSO$, there exists a  class $\vU= \sum \chi_{\alpha}^{(n)} \zeta_{\alpha}^{(n)} \in H_1(M, \Sigma,\Rset)$ such that $\partial \vU = \upsilon$ and
$$ \varlimsup_{n \to +\infty} \frac{\log (\Vert  \chi^{(n)} \Vert / \Vert \chi^{(0)} \Vert)}{\log  \Vert B(0,n) \Vert} =0.$$ 
Moreover, two such classes differ by an element in $\Gamma^s$. Therefore the Proposition defines a linear section from $\RSO$ to $H_1(M, \Sigma,\Rset) / \Gamma^s$.
\end{proposition}
The proof of this proposition is given below. 
Remark that  the content of the proposition is non-empty only when $s \geq 2$.

\smallskip 
We also have the following symmetric Proposition.
\begin{proposition}\label{propKZneg}
Assume that $M$ is negatively KZ-hyperbolic. Then, for every $\upsilon \in  \RSO$, there exists a  class $\vU = \sum \chi_{\alpha}^{(n)} \zeta_{\alpha}^{(n)}\in H_1(M, \Sigma,\Rset)$ such that $\partial \vU = \upsilon$ and
$$ \varlimsup_{n \to -\infty} \frac{\log (\Vert  \chi^{(n)} \Vert / \Vert \chi^{(0)} \Vert)}{\log  \Vert B(n,0) \Vert} =0.$$
Moreover, two such classes differ by an element in $\Gamma^u$. Therefore the proposition defines a linear section from $\RSO$ to $H_1(M, \Sigma,\Rset) / \Gamma^u$.
\end{proposition}
We omit the proof of Proposition \ref{propKZneg} since it is analogous to the proof of Proposition \ref{propKZpos} given below. 

Finally, combining the two Propositions, we have the following Corollary.

\begin{corollary}
Assume that $M$ is  KZ-hyperbolic. Then, for every $\upsilon \in  \RSO$, there exists a unique  class $\vU= \sum \chi_{\alpha}^{(n)} \zeta_{\alpha}^{(n)} \in H_1(M, \Sigma,\Rset)$ such that $\partial \vU = \upsilon$ and
$$ \varlimsup_{n \to +\infty} \frac{\log (\Vert  \chi^{(n)} \Vert / \Vert \chi^{(0)} \Vert)}{\log  \Vert B(0,n) \Vert} = \varlimsup_{n \to -\infty} \frac{\log (\Vert  \chi^{(n)} \Vert / \Vert \chi^{(0)} \Vert)}{\log  \Vert B(n,0) \Vert} = 0.$$ 
The section $\upsilon \mapsto \vU$ from $\RSO$ to $H_1(M, \Sigma,\Rset)$ is linear.
\end{corollary}

\begin{proof}[Proof of Proposition \ref{propKZpos}]
Uniqueness modulo $\Gamma^s$ is clear from the first  inequality in Remark \ref{remKZ1}. 
To prove existence, we choose, for each $n \geq 0$, a class  $\varXi(n) =  \sum_\alpha \wh\chi_{\alpha}^{(n)} \zeta_\alpha^{(n)} \in H_1(M, \Sigma,\Rset)$ such that $\partial \vX(n) = \upsilon$  and
$$ \Vert \wh\chi^{(n)} \Vert \leq C = C(\upsilon).$$
We have therefore (as the coefficients of $B(n-1,n)$ are equal to $0$ or $1$)
$$ \Vert \varXi(n) -\varXi(n-1) \Vert_{n,\star s} \leq C.$$
As $\varXi(n) -\varXi(n-1) $ belongs to $H_1(M,\Rset)$, we deduce from the first inequality in remark \ref{remKZ1} that
$$ \Vert \varXi(n) -\varXi(n-1)  \Vert_{0,\star s} \leq C  \Vert B(0,n) \Vert^{-\sigma}  .$$
We can therefore choose a class $\varXi(n-1,n) = \sum x_\alpha(n) \zeta_\alpha^{(0)} \in H_1(M,\Rset)$ such that $\varXi(n) -\varXi(n-1) - \varXi(n-1,n) \in \Gamma^s$ and 
 $$ \Vert x(n) \Vert_{\infty} \leq C  \Vert B(0,n) \Vert^{-\sigma}.$$
Therefore the series $ \sum_{n>0} \varXi(n-1,n) $ is convergent. We will see that 
$$ \vU := \varXi(0) + \sum_{n>0} \varXi(n-1,n)$$
satisfies the conclusions of the proposition. Clearly we have $\partial \vU = \upsilon$.

Next we estimate $\Vert \vU \Vert_{m,\star s}$. One has
$$ \vU = \varXi(m) + \sum_{n>m} \varXi(n-1,n)\quad {\rm mod.} \; \Gamma^s.$$
From the choice of $\varXi(m)$, one has $\Vert \varXi(m) \Vert_{m,\star s} \leq C$. To estimate $\Vert \varXi(n-1,n) \Vert_{m,\star s}$, there are two cases
\begin{itemize}
\item If $\Vert B(0,n) \Vert^{\sigma/2} \leq \Vert B(0,m) \Vert$, we use the second inequality in remark \ref{remKZ1} to get, for any $\tau >0$
$$ \Vert \varXi(n-1,n) \Vert_{m,\star s} \leq C_{\tau} \Vert B(0,m) \Vert^{\tau}.$$
\item If $\Vert B(0,n) \Vert^{\sigma/2} \geq \Vert B(0,m) \Vert $, we have
\begin{eqnarray*}
\Vert \varXi(n-1,n) \Vert_{m,\star s} &\leq& \Vert B(0,m) \Vert \Vert \varXi(n-1,n) \Vert_{0,\star s} \\
& \leq & C  \Vert B(0,m) \Vert  \Vert B(0,n) \Vert^{-\sigma} \\
& \leq & C \Vert B(0,n) \Vert^{-\sigma/2} .
\end{eqnarray*}
\end{itemize}
From the first condition in Definition \ref{defKZposhyp}, we have 
$$ \sum \Vert B(0,n) \Vert^{-\sigma/2} \leq C,$$
and 
$$ \sum_{\Vert B(0,n) \Vert^{\sigma/2} \leq \Vert B(0,m) \Vert} \Vert B(0,m) \Vert^{\tau} \leq C_{\tau}  \Vert B(0,m) \Vert^{2\tau}.$$
After renaming $\tau$, we get $\Vert \vU \Vert_{m,\star s} \leq C_{\tau} \Vert B(0,m) \Vert^{\tau}.$

Thus, for each $m \geq 0$,  we can write 
$$ \vU = \sum y_{\alpha}(m) \zeta_\alpha^{(m)}   + \vU(m),$$
with $\Vert y(m)\Vert_{\infty} \leq C_{\tau} \Vert B(0,m) \Vert^{\tau}$ and $\vU(m) \in \Gamma^s$. We can also take $\vU(0) =0$. As the coefficients of $B(m-1,m)$ are equal to $0$ or $1$, we have, for $m>0$ 
$$ \vU(m) - \vU(m-1) = \sum \wt y_{\alpha}(m) \zeta_\alpha^{(m)}, \quad \Vert \wt y(m)\Vert_{\infty} \leq C_{\tau} \Vert B(0,m) \Vert^{\tau}.$$
From the coherence property in Definition \ref{defKZposhyp}, we finally obtain
$\vU(m) = \sum \wh y_\alpha (m) \zeta_\alpha^{(m)} $ with
$$  \Vert \wh y(m)\Vert_{\infty} \leq C'_{\tau} \Vert B(0,m) \Vert^{\tau} \sum_{\ell \leq m} \Vert B(0,\ell) \Vert^{\tau} \leq C''_\tau \Vert B(0,m) \Vert^{3\tau}.$$
This concludes the proof of the proposition.

\end{proof}

\subsection{Piecewise-affine paths in $H_1(M \setminus \Sigma, \Rset)$}\label{sec:piecewise_affine_paths}

We assume that $M$ has no horizontal connection. Let $n < 0$; denote by $\alpha_w$ the winner of the arrow $\gamma(n,n+1)$, by $ \alpha_\ell$ its loser. Recall that we have
$ \theta_\alpha^{(n+1)} =  \theta_\alpha^{(n)}$ for $\alpha \ne \alpha_\ell$ and 
$$  \theta_{\alpha_{\ell}}^{(n+1)} =  \theta_{\alpha_{\ell}}^{(n)}  +  \theta_{\alpha_w}^{(n)} .$$

\smallskip
To the arrow $\gamma = \gamma(n,n+1)$, we associate a substitution transformation $\varsigma_\gamma$ on the alphabet $\A$ defined as follows:
\begin{itemize}
\item If $\gamma$ is of top type, $\varsigma_\gamma$ is defined by
$$ \alpha_\ell  \mapsto \alpha_\ell \, \alpha_w , \qquad \alpha \mapsto \alpha, \quad\forall \alpha \ne \alpha_\ell.$$

\item If $\gamma$ is of bottom type, $\varsigma_\gamma$ is defined by
$$ \alpha_\ell  \mapsto \alpha_w \ \alpha_\ell , \qquad \alpha \mapsto \alpha, \quad \forall \alpha \ne \alpha_\ell.$$
\end{itemize}
For $\alpha \in \A$, we then define inductively a sequence of words $(W_\alpha (n))_{n\leq 0}$ by 
$$ W_\alpha (0) = \alpha, \quad W_\alpha(n) = \varsigma_{\gamma(n,n+1)} (W_\alpha (n+1)), \quad \forall n<0.$$

One can see that for any $n\leq 0$ and $\alpha,\beta \in \A$, the number of occurrences of $\beta$ in the word $W_{\alpha}(n)$ is equal to $\left(B(n,0)\right)_{\alpha\, \beta}$.

\smallskip
The relation to the Rauzy-Veech algorithm is given by the following Proposition. 
\begin{proposition} \label{proptheta}
Write $W_{\alpha}(n) = \beta_0 \ldots \beta_{N-1}$. Then we have, for $0 \leq j <N$
$$ (T^{(n)})^j (I_\alpha^{(0)} ) \subset I_{\beta_j} ^{(n)}.$$
\end{proposition}
The proof is clear from the definition of the Rauzy-Veech algorithm.

\smallskip
{\Blue 
\begin{remark}\label{rk:leaveinterpretation}
Equivalently, for any $n<0$ and $\alpha,\beta \in \A$, we also have the following geometric interpretation of the word $W_{\alpha}(n) = \beta_0 \ldots \beta_{N-1}$. The rectangle $R^{(0)}_\alpha$ can be seen as a union of rectangles (of the same width, but shorter height) each fully contained in a rectangle $R^{(b)}_\beta$ for some $\beta \in A$ (refer to the bottom part of Figure~\ref{fig:pastsums}). These rectangles, in the order from the bottom to the top of $R^{(0)}_\alpha$, are contained in $R^{(n)}_{\beta_\ell}$ for $\ell=0,1,\dots, N-1$. In other words, if one considers a leaf of the vertical flow from a point $x \in I^{(0)}_\alpha$ of length $q_\alpha^{(0)}$, it crosses $I^{(n)}$ exactly $N$ times, and the intersections belong, in order, to $I^{(n)}_{\beta_0}, I^{(n)}_{\beta_1}, \dots, I^{(n)}_{\beta_{N-1}}$. 
\end{remark}}

In other words (cf.~Remark~\ref{rk:leaveinterpretation}), if one considers a leaf of the horizontal flow starting from a point $x \in L^{(0)}_\alpha$ and crossing all the rectangle $R^{(0)}_\alpha$ horizontally (see for example the top part of Figure \ref{fig:pastsums}), then its intersections with $L^{(n)}$ will be contained, in order, in the intervals $L^{(n)}_{\beta_\ell}$, for $\ell=0,1, \dots, N-1$.

\medskip
To the word $W_{\alpha}(n) = \beta_0 \ldots \beta_{N-1}$, we associate a broken line $\mathcal W_{\alpha}(n)$ in $H_1(M \setminus \Sigma, \Rset)$ as follows:
\begin{itemize}
\item for $0 \leq j \leq N$, let $\mathcal W_{\alpha,j}(n) := \sum_{0 \leq i <j} \theta_{\beta_i}^{(n)} \in H_1(M \setminus \Sigma,\Zset)$;
\item set $\mathcal W_\alpha (n)$ {\Green to be} the concatenation of the segments from $\mathcal W_{\alpha,j}(n) $ to $\mathcal W_{\alpha,j+1}(n)$ for $0 \leq j <N$. 
\end{itemize}
The points $\mathcal W_{\alpha,j}(n)$, for $0 \leq j \leq N$, are called the {\it vertices} of $\mathcal W_\alpha(n)$. 

\smallskip
The following properties hold.
\begin{itemize}
\item For $n\leq 0$, any vertex of $\mathcal W_\alpha(n)$ is also a vertex of  $\mathcal W_\alpha(n-1)$.
\item The first vertex $\mathcal W_{\alpha,0}(n)$ is equal to $0$.
\item The last vertex $\mathcal W_{\alpha,N}(n)$ is equal to $\theta_\alpha^{(0)}$. 
\end{itemize}

Let $p_v: H_1(M \setminus \Sigma, \Rset) \to \Rset$ be the linear form determined by integration of $\Im \omega$. The restriction of $p_v$ to $\mathcal W_\alpha(n)$ is a piecewise-affine homeomorphism  onto the interval $[0,q_\alpha^{(0)}]$. The image of the segment from $\mathcal W_{\alpha,j}(n) $ to $\mathcal W_{\alpha,j+1}(n)$ is a subinterval of length $q_{\beta_j}^{(n)}$.

\subsection{The functions $\Omega_\alpha^{(n)} (\vU)$} \label{ssOmega}

In this subsection, we assume that $M$ is negatively KZ-hyperbolic of \emph{dual Roth type}. Let $\upsilon \in \RSO$, and let $\vU \in H_1(M , \Sigma, \Rset)$ 
%
be a class satisfying the conclusion of Proposition \ref{propKZneg}. Let $p_{\vU}: H_1(M \setminus \Sigma, \Rset) \to \Rset^2$ be the linear map defined by
$$ p_{\vU} (\theta) = (p_v(\theta), i( \vU, \theta)).$$

From the above considerations (cf.~$\S$~\ref{sec:piecewise_affine_paths}), the image $p_{\vU} (\mathcal W_\alpha (n))$ is the graph of a piecewise-affine function $\Omega_\alpha^{(n)} (\vU) : [0,q_\alpha^{(0)}] \to \Rset$. This construction is analogous  in the homological context to the construction which was done in \S~\ref{sec:functions} (see the functions $\Omega^{(n)}_\alpha(\pi,\tau,\chi)$ associated in \S~\ref{sec:functions} to  a sequence $\chi=(\chi^{(n)})_n$ of central functions). Following the same arguments {\Green as in} \S~\ref{sec:comparingfunctions} and  \S~\ref{sec:convergence}, one hence has the following result (analogous in this context to Theorem \ref{thm:convergence}): 

\begin{proposition}\label{thm:homologicalconvergence}
Assume that $(\pi, \tau)$ is of dual Roth type. For any $\alpha \in \A$, $\Omega_\alpha^{(n)} (\vU) : [0,q_\alpha^{(0)}] \to \Rset$ converge as $n\rightarrow - \infty $ as a distribution on  H\"older functions, i.e.\ for any function $\psi$ on the interval $[0,q_{\alpha}(\pi,\tau)]$ which is mean zero  and  H\"older continuous   for some exponent $\eta \in (0,1)$,  the sequence 
$$ \int_0^{q_{\alpha}(\pi,\tau)} \psi(x)\ \Omega_{\alpha}^{(n)}(\vU) (x) dx$$
is convergent as $n \rightarrow - \infty$. 
\end{proposition}

Since in this section we only made assumptions on the backward rotation number $(\pi, \tau)$, the homology class 
$\vU \in H_1(M, \Sigma, \Rset)$ (and hence the limiting distribution) is only well-determined by its image under the boundary operator modulo the unstable subspace $\Gamma^u$ (cf.~Proposition \ref{propKZneg}). However, we also have the following result. 


\begin{proposition}\label{propmodgammau}
Let $\vU'$ be another class in $H_1(M,\Sigma,\Rset)$ such that $\vU' - \vU $ belongs to $\Gamma^u$. 
For any $n \leq 0$, any $\alpha \in \A$, the difference $\Omega_\alpha^{(n)} (\vU') - \Omega_\alpha^{(n)} (\vU)$ is H\"older continuous, with a H\"older exponent and a H\"older constant independent of $n$ and $\alpha$. Moreover, the sequence $\Omega_\alpha^{(n)} (\vU') - \Omega_\alpha^{(n)} (\vU)$ converges uniformly when $n$ goes to $-\infty$ to a H\"older continuous function on $ [0,q_\alpha^{(0)}] $.
\end{proposition}

\begin{proof}
We have en estimate
$$  \max_\alpha q_\alpha^{(n)} \leq C^{-1} \Vert B(n,0) \Vert^{-1}, \quad \forall n \leq 0.$$
As $(\pi, \tau)$ is of dual Roth type, we also have, for all $\tau >0$
$$  \min _\alpha q_\alpha^{(n)} \leq C_\tau^{-1} \Vert B(n,0) \Vert^{-1-\tau}.$$
 On the other hand, for two consecutive vertices  $\mathcal W_{\alpha,j}(n) $, $\mathcal W_{\alpha,j+1}(n)$ of $\mathcal W_\alpha (n)$, one has
 \begin{eqnarray*}
  \vert i(\vU' - \vU, \mathcal W_{\alpha,j+1}(n)) - i(\vU' - \vU, \mathcal W_{\alpha,j}(n))\vert & =&  i(\vU' - \vU, \theta_{\beta_j}^{(n)})\\
  &\leq & C  \Vert B(n,0) \Vert^{-\sigma}.
  \end{eqnarray*}
 
Using the space decomposition of \cite{MY} (see Section 3.8 in \cite{MY}),
 one gets the conclusions of the theorem from these estimates.
\end{proof}
Since by Proposition~\ref{thm:homologicalconvergence}, for any $\alpha \in \A$ the sequence $\Omega_\alpha^{(n)} (\vU)$ converge to a distribution on Holder continous functions, 
Proposition~\ref{propmodgammau}  shows that the singular nature of the limit distribution 
 only depends the projection of $\vU \in H_1(M,\Sigma,\Rset)$ to the stable space $\Gamma^s$.

%
%
%

{\Blue 

\subsection{Piecewise-affine paths in $H_1(M ,  \Sigma, \Rset)$}\label{sec:piecewise_affine_paths_dual}
In this section, we show  that exploiting the duality between horizontal and vertical flows and between $H_1(M \backslash  \Sigma, \Rset)$ and $H_1(M ,  \Sigma, \Rset)$, we present  a  construction and results dual to the one in the previous two sections, namely \S~\ref{sec:piecewise_affine_paths} and \S~\ref{ssOmega}, which in particular allow to construct distributional limit shapes for relative homology classes with subpolynomial deviations for the \emph{horizontal} flow. This dual result requires the KZ \emph{negative} hyperbolicity assumption, together with the \emph{direct} {Roth type} condition (and is based on the study of special Birkohff sums instead than dual ones).  

\medskip
Let us first construct piecewise-affine paths in $H_1(M ,  \Sigma, \Rset)$, in a dual way to the definition of piecewise-affine paths in $H_1(M \backslash \Sigma, \Rset)$ in \S~\ref{sec:piecewise_affine_paths}. 
\smallskip

We assume that $M$ has no \emph{vertical} connection. Let $n > 0$. To the arrow $\gamma = \gamma(n,n+1)$, we associate a substitution transformation $\varsigma_\gamma$ on the alphabet $\A$ defined as follows. 
Denote by $\alpha_w$ the winner of arrow $\gamma$, by $ \alpha_\ell$ its loser. 
\begin{itemize}
\item If $\gamma$ is of top type, $\varsigma_\gamma$ is defined by
$$ \alpha_w  \mapsto \alpha_w \, \alpha_\ell , \qquad \alpha \mapsto \alpha, \quad\forall \alpha \ne \alpha_\ell.$$

\item If $\gamma$ is of bottom type, $\varsigma_\gamma$ is defined by
$$ \alpha_w  \mapsto \alpha_w \, \alpha_\ell , \qquad , \alpha_\ell \mapsto \alpha_w \qquad \alpha \mapsto \alpha, \quad \forall \alpha \notin \{ \alpha_\ell, \alpha_w\}.$$
\end{itemize}
For $\alpha \in \A$, we then define inductively a sequence of words $(W^*_\alpha (n))_{n\geq 0}$ by 
$$ W^*_\alpha (0) = \alpha, \quad W_\alpha^*(n+1) = \varsigma_{\gamma(n,n+1)} (W^*_\alpha (n)), \quad \forall n \geq 0.$$

One can see that for any $n\geq 0$ and $\alpha,\beta \in \A$, the number of occurrences of $\beta$ in the word $W^*_{\alpha}(n)$ is equal to $\left(B(0,n)\right)_{\alpha\, \beta}$.

\smallskip
The relation to the Rauzy-Veech algorithm is the following. For any $n\geq 0$ and $\alpha \in \A$ the rectangle $R^{(0)}_\alpha$ of base $I^{(0)}_\alpha$ can be seen as union of rectangles (of same height but shorter width) fully contained in the rectangles $R^{(0)}_\beta$, $\beta \in A$ (refer to the right hand side of Figure \ref{fig:futuresums}). In particular, the interval $I^{(0)}_\alpha$ can be seen as a union of translates of the intervals $I^{(n)}_\beta$. If we write $$W^*_{\alpha}(n) = \beta_0 \ldots \beta_{N-1},$$ 
then we have that $I^{(0)}_\alpha$ is union, from left to right, of translates of the intervals $I^{(n)}_{\beta_0}$,  $I^{(n)}_{\beta_1}, \dots , I^{(n)}_{\beta_{N-1}}$, or, more precisely
$$
I_{\alpha}^{(0)}= 
 =  \bigcup_{0 \leq \ell \leq N-1} j_{\alpha,\ell}^{(0,n)} (I_{\beta_\ell}^{(n)}), \qquad \text{where}\quad j_{\alpha,\ell}^{(0,n)}(x) = x+ \sum_{0\leq i< \ell} \lambda^{(n)}_{\beta_i}.
$$
In other words (cf.~Remark~\ref{rk:leaveinterpretation}), if one considers a leaf of the horizontal flow starting from a point $x \in L^{(0)}_\alpha$ and crossing all the rectangle $R^{(0)}_\alpha$ horizontally (see for example the top part of Figure \ref{fig:pastsums}), then its intersections with $L^{(n)}$ will be contained, in order, in the intervals $L^{(n)}_{\beta_\ell}$, for $\ell=0,1, \dots, N-1$.

\medskip
To the word $W^*_{\alpha}(n) = \beta_0 \ldots \beta_{N-1}$, we associate a broken line $\mathcal W^*_{\alpha}(n)$ in $H_1(M, \Sigma, \Rset)$ as follows:
\begin{itemize}
\item for $0 \leq j \leq N$, let $\mathcal W^*_{\alpha,j}(n) := \sum_{0 \leq i <j} \zeta_{\beta_i}^{(n)} \in H_1(M , \Sigma,\Zset)$;
\item set $\mathcal W^*_\alpha (n)$ {\Green to be} the concatenation of the segments from $\mathcal W^*_{\alpha,j}(n) $ to $\mathcal W^*_{\alpha,j+1}(n)$ for $0 \leq j <N$. 
\end{itemize}
The points $\mathcal W^*_{\alpha,j}(n)$, for $0 \leq j \leq N$, are called the {\it vertices} of $\mathcal W^*_\alpha(n)$. 

\smallskip
The following properties hold.
\begin{itemize}
\item For $n\geq 1$, any vertex of $\mathcal W^*_\alpha(n)$ is also a vertex of  $\mathcal W^*_\alpha(n-1)$.
\item The first vertex $\mathcal W^*_{\alpha,0}(n)$ is equal to $0$.
\item The last vertex $\mathcal W^*_{\alpha,N}(n)$ is equal to $\zeta_\alpha^{(0)}$. 
\end{itemize}

Let $p_h: H_1(M , \Sigma, \Rset) \to \Rset$ be the linear form determined by integration of $\Re \omega$. The restriction of $p_h$ to $\mathcal W^*_\alpha(n)$ is a piecewise-affine homeomorphism  onto the interval $[0,\lambda_\alpha^{(0)}]$. The image of the segment from $\mathcal W^*_{\alpha,j}(n) $ to $\mathcal W^*_{\alpha,j+1}(n)$ is a subinterval of length $\lambda_{\beta_j}^{(n)}$.

\subsection{The functions ${\Omega^*}_\alpha^{(n)} (\vU)$} \label{ssOmega}

In this subsection, we assume that $M$ is \emph{positively} KZ-hyperbolic of (\emph{direct}) \emph{Roth type}. Let $\upsilon \in \RSO$.

Using the indentification between  $ H_1(M \backslash \Sigma, \Rset)$ and  $ H_1(M \backslash \Sigma, \Rset)$, similarly to Proposition \ref{propKZposg}, one can show the existence of classes $\vU \in H_1(M \backslash \Sigma, \Rset)$ 
%
be a class satisfying the conclusion of Proposition \ref{propKZposg}.  Let $\vU \in H_1(M \backslash \Sigma, \Rset)$ be any such class.

Let $p_{\vU}: H_1(M , \Sigma, \Rset) \to \Rset^2$ be the linear map defined by
$$ p_{\vU} (\zeta) = (p_h(\zeta), i( \vU, \zeta)).$$

From the above considerations (cf.~$\S$~\ref{sec:piecewise_affine_paths_dual}), the image $p_{\vU} (\mathcal W^*_\alpha (n))$ is the graph of a piecewise-affine function ${\Omega^*}_\alpha^{(n)} (\vU) : [0,\lambda_\alpha^{(0)}] \to \Rset$. With arguments similar to the proof of Proposition~\ref{thm:homologicalconvergence}, where the role of dual Birkhoff sums is replaced by special Birkhoff sums and 
  using the bounds for Birkhoff sums of H\"older under the dual Roth type condition proved in this paper,  
  one can prove the following.

\begin{proposition}\label{thm:homologicalconvergence_dual}
Assume that $(\pi, \lambda)$ is of dual Roth type. For any $\alpha \in \A$, ${\Omega^*}_\alpha^{(n)} (\vU) : [0,\lambda_\alpha^{(0)}] \to \Rset$ converge as $n\rightarrow + \infty $ as a distribution on  H\"older functions, i.e.\ for any function $\psi$ on the interval $[0,\lambda_{\alpha}^{(0)}]$ which is mean zero  and  H\"older continuous   for some exponent $\eta \in (0,1)$,  the sequence 
$$ \int_0^{\lambda_{\alpha}} \psi(x)\ \Omega_{\alpha}^{(n)}(\vU) (x) dx$$
is convergent as $n \rightarrow + \infty$. 
\end{proposition}

Finally, one can also prove a dual version of Proposition~\ref{propmodgammau} and thus show that the singular nature of the limit distribution 
 only depends the projection of $\vU \in H_1(M \backslash \Sigma,\Rset)$ to the unstable space $\Gamma^u \subset H_1(M \backslash \Sigma,\Rset)$.
}

\appendix

\section{Completeness of backward rotation numbers}\label{secbac}

In this section we prove Proposition 
\ref{p7} (see \S~\ref{completenessrotnumb}), which shows that when $M$ has no horizontal connection, the backward rotation number $\ug$ defined in \S~\ref{sec:backwardRauzyVeech} is $\infty$-complete.
The proof which we present exploits the combinatorics of Rauzy-Veech classes and is surprisingly involved. It would be interesting to know if it is possible to find a simpler, perhaps more geometric, proof. 

\smallskip

\subsection{prove Proposition  \ref{p7}}\label{completenessrotnumb}

Let $M$ be a translation surface constructed from combinatorial data $\pi \in \R$, length data $\lambda \in C$ and suspension data $\tau \in \Theta_\pi$. We assume throughout this Appendix that  $M$ has no horizontal saddle connections. We associate to $\tau$ the quantities 

$$
 H^t_k (\pi,\tau):= \sum_{\pi_t \alpha \leq k} \tau_\alpha, \quad H^b_k (\pi,\tau) := \sum_{\pi_b \alpha \leq k} \tau_\alpha , \quad\forall 1 \leq k <d.
$$

We also define $H^t_d(\pi,\tau) = H^b_d(\pi,\tau) = H_d(\pi,\tau) = \sum_\alpha \tau_\alpha$ and $H^t_0(\pi,\tau) = H^b_0(\pi,\tau) = H_0(\pi,\tau) = 0$.

The domain for the suspension data $\tau$ is   is the open cone $\Theta_\pi$ in $\RA$ defined by the  Veech conditions on suspension data:
$$
H^t_k (\pi,\tau) >0, \qquad H^b_k (\pi,\tau)  <0, \qquad\forall 1 \leq k <d.
$$

As usual, we write $\zeta_\alpha = \lambda_\alpha + i \tau_\alpha$.
 Since $M$ has no horizontal saddle connections, by iterating the \emph{backward} Rauzy Veech algorithm defined in \S~\ref{sec:backwardRauzyVeech}, we obtain a {\it backward rotation number}, i.e an infinite  path $\ug = \ldots \gamma_{-1} \star \gamma_0$ in $\D$ with terminal point $\pi^{(0)}$. We want to show that it is $\infty$-complete.

\subsection{The quantity $H(n)$}
Given $T$ with no connection, let us denote as usual by  $T^{(n)}$  the i.e.m. with data $(\pi^{(n)}, \lambda^{(n)})$ obtained via Rauzy-Veech induction.  
We define for $n \leq 0$ the quantities
\begin{align*}
H_k^t(n)&:= H^t_k (\pi^{(n)}, \lambda^{(n)}), \qquad \text{for}\ 1 \leq k <d; \\
H_b^t(n)&:= H^b_k (\pi^{(n)}, \lambda^{(n)}), \qquad \text{for}\ 1 \leq k <d;\\
  H(n)&:= \sum_{1 \leq k <d} H_k^t(n) - H_k^b (n) .
 \end{align*}
This is the sum of $(2d-2)$ positive numbers. We have

$$
H(n-1) = H(n) + H^b_{k-1}(n) - H^t_{d-1}(n),
$$
 
 with $k:= \pi_b^{(n)}(\alpha_w)$, when $\gamma(n,n+1)$ is of top type. Similarly
 
$$
H(n-1) = H(n) - H^t_{k-1}(n) + H^b_{d-1}(n),
$$
 
 with $k:= \pi_t^{(n)}(\alpha_w)$, when $\gamma(n,n+1)$ is of bottom type.  
 
Therefore we always have $H(n-1)<H(n)$; the sequence $H(n)$ thus converges to a limit $H(-\infty) \geq 0$ when $n$ goes to $-\infty$. 
 
\begin{lemma}\label{l2}
 For any $\alpha \in \A$, any $n \leq0$, one has 
$$ \vert \tau_\alpha^{(n)} \vert \leq H(n).$$ 
\end{lemma}
\begin{proof} Indeed, if for instance $\pi_t(n)(\alpha)=:k<d$, one has
$$  \vert \tau_\alpha^{(n)} \vert \leq \vert H_k^t(n) - H_{k-1}^t(n) \vert,$$
with $H_k^t(n), \,H_{k-1}^t(n) \in [0,H(n)]$.
\end{proof}

\subsection{Switching times}

\begin{lemma}\label{l3}
The path $\ug$ contains infinitely many arrows of each type.
\end{lemma}
 \begin{proof}
If the conclusion of the lemma does not hold, we may assume, after truncating $\ug$, that all arrows $\gamma(n,n+1)$ have the same type. Assume for instance that all arrows are of top type. Then they have all the same winner $\alpha_w$. Let $\pi_b^{(0)}(\alpha_w) =k$. We must have $k=1$ because $H_{k-1}^b(n)$ does not depend on $n$ and therefore must vanish (as $H(n)$ is convergent). The $\tau_\alpha^{(n)}$ with $\alpha \ne \alpha_w$ do not depend on $n$. One also has 
 $$ \tau_{\alpha_w}^{(n-d+1)} =  \tau_{\alpha_w}^{(n)}+ \sum_{\alpha \ne \alpha_w} \tau_\alpha, \quad \forall n \leq 0.$$
 However, we have $\sum _{\alpha \ne \alpha_w} \tau_\alpha = H^t_{d-1} >0$, hence $\tau_{\alpha_w}$ cannot be negative for all $n \leq 0$, a contradiction.
 \end{proof}
 
 Denote by $\T$ (resp. $\B$) the set of $n\leq 0$ such that $\gamma(n,n+1)$ is of top type (resp. bottom type). 
 
 \begin{lemma}\label{l4}
 One has 
 $$
 \lim_{n \to -\infty,  n \in \T} H_{d-1}^t(n) =   \lim_{n \to -\infty, n \in \B} H_{d-1}^b(n) =0. 
 $$ 
 \end{lemma}
 
 \begin{proof}
 This follows from the formulas relating $H(n)$ and $H(n-1)$.
 \end{proof}
 
\begin{lemma}\label{l5}
If $\gamma(n,n+1)$ is of top type, one has 
$$H_{d-1}^t(n-1) = H_{d-1}^t(n),\qquad H_{d-1}^b(n-1)=H_{d}(n).$$

 If $\gamma(n,n+1)$ is of bottom type, one has 
 $$H_{d-1}^b(n-1) = H_{d-1}^b(n), \qquad H_{d-1}^t(n-1)=H_{d}(n).$$  
\end{lemma}

\begin{proof}
Clear from the formulas of an elementary Rauzy-Veech step.
\end{proof}

\begin{definition}\label{defin1}
A negative integer $n$ is a {\it switching time} if $\gamma(n,n+1)$ and $\gamma(n+1,n+2)$ are of different types. Let $\mathcal {SW}$ be the set of switching times. It is infinite by lemma \ref{l3}.
\end{definition}
 
 \begin{lemma}\label{l6}
 One has 
 $$
 \lim_{n \to -\infty,  n \in \SW} H_{d-1}^t(n) =   \lim_{n \to -\infty, n \in \SW} H_{d-1}^b(n) =0. 
 $$ 
 \end{lemma}
\begin{proof}
This follows from lemmas \ref{l4} and \ref{l5}.
\end{proof}

 \begin{lemma}\label{l7}
 One has also
 
 $$
 \lim_{n \to -\infty,  n \in \SW} H_{d}(n)  =0. 
 $$ 

 \end{lemma}
\begin{proof}
Let $n \in \SW$. Assume for instance that $\gamma_{n+1}=\gamma(n,n+1)$ is of top type and $\gamma_{n+2}=\gamma(n+1,n+2)$ is of bottom type. Let $\alpha_w$ be the winner of  $\gamma_{n+2}$  and define $k:= \pi_b(n+1)(\alpha_w) \in [1,d)$. We have 
$$ H_d(n) = H(n+1) - H(n) + H^b_{k+1}(n+1) . $$
When $n$ is large, $H(n)$ and $H(n+1)$ are close,
 $H_d(n)$ is positive (as $\gamma(n)$ is of bottom type) and $H^b_{k+1}(n+1)$ is negative (even when $k=d-1$ as $\gamma_{n+2}$ is of top type). Therefore $H_d(n)$ is small.
\end{proof}

\subsection{The subset $\A' \subset \A$}

Let $\A'$ be the set of letters in $\A$ which are the winners of at most finitely many arrows of $\ug$. After cutting a terminal subpath off $\ug$, we may assume that 
none of the letters  in $\A'$ is the winner of an arrow of $\ug$. Therefore the quantities
$\tau_\alpha^{(n)}$, for $\alpha \in \A'$, do not depend on $n$. Observe that, as $M$ has no horizontal saddle-connection, these quantities are \emph{different from $0$}.

\smallskip

\begin{lemma}\label{l8}
For any $\alpha \in \A \setminus \A'$, we have 
$$ \lim_{n \to -\infty, n\in \SW} \tau_\alpha^{(n)} =0.$$
\end{lemma}

\begin{proof}
 Let $\alpha \in  \A \setminus \A'$. Between consecutive switching times, all arrows have the same winner. Let $0 >  n_1>\ldots >n_k> \ldots$ be the switching times such that $\alpha$ is the winner of $\gamma(n_k,n_{k+1})$, and let $n'_k <n_k$ be the switching 
 time immediately inferior to $n_k$. The value of  $\tau_\alpha^{(n)}$ stays the same for
 $n_{k+1} \leq n \leq n'_k$.  On the other hand, it follows from Lemmas \ref{l6} and \ref{l7} that 
 
 $$
 \lim_{k \to +\infty} \tau_\alpha^{(n_k)} =0.
 $$
\end{proof}

\begin{remark}
At this stage of the proof we \emph{do not claim} that $\tau_\alpha^{(n)}$ is small for $n'_k <n<n_k$.
\end{remark}
 
We have to prove that $\A'$ is empty. \emph{We assume by contradiction that $\A'$ is not empty}. 

 \begin{lemma}\label{l9}
One has 
$$
\sum_{\alpha \in \A'} \tau_\alpha =0.
$$
\end{lemma}

\begin{proof}
Indeed, as the switching time $n$ goes to $-\infty$, the quantity $H_d(n)$ is alternately positive and negative and the $\tau_{\alpha}^{(n)}$ with $\alpha \in \A \setminus \A'$ converge to zero.
\end{proof}

\subsection{Decompositions of $\A \setminus \A'$}

Let $n$ be a switching time. Define

$$ 
\A^t_0(n) = \{ \alpha \in \A \setminus \A' \ \vert \sum_{{\Small \begin{array}{cc}\beta \in \A'\\ \pi_t^{(n)} (\beta) < \pi_t^{(n)} (\alpha)\end{array}}} \tau_\beta =0\},
$$
 
 $$ 
\A^t_{>0}(n) = \{ \alpha \in \A \setminus \A' \ \vert \sum_{{\Small \begin{array}{cc}\beta \in \A'\\ \pi_t^{(n)} (\beta) < \pi_t^{(n)} (\alpha)\end{array}}} \tau_\beta >0\},
$$
$$ 
\A^b_0(n) = \{ \alpha \in \A \setminus \A' \ \vert \sum_{{\Small\begin{array}{cc}\beta \in \A'\\ \pi_b^{(n)} (\beta) < \pi_b^{(n)} (\alpha)\end{array}}} \tau_\beta =0\},
$$
$$ 
\A^b_{<0}(n) = \{ \alpha \in \A \setminus \A' \ \vert \sum_{{\Small\begin{array}{cc}\beta \in \A'\\ \pi_b^{(n)}) (\beta) < \pi_b^{(n)} (\alpha)\end{array}}} \tau_\beta <0\},
$$

\begin{lemma}\label{l10}
For $n$ negative enough, one has $\A^t_0(n) = \A^b_0(n)=\A \setminus \A'$
\end{lemma}

\begin{proof}
We first observe that, for $n$ negative enough, we must have

$$
\A \setminus \A' = \A^t_0(n) \sqcup \A_{>0}^t(n) =  \A^b_0(n) \sqcup \A_{<0}^b(n).
$$ 

Indeed, if some $\alpha \in \A \setminus \A'$ satisfies 

$$
\sum_{{\Small \begin{array}{cc}\beta \in \A'\\ \pi_t(n) (\beta) < \pi_t(n) (\alpha)\end{array}}} \tau_\beta <0
$$ 
with $n\in \SW$ large and negative, we would contradict the suspension data conditions
as $\tau_\beta(n)$ goes to zero  for $\beta \in \A \setminus \A'$  (Lemma \ref{l8}). Therefore we have $\A \setminus \A' = \A^t_0(n) \sqcup \A_{>0}^t(n)$ for $n$ negative enough.  We prove that  $\A \setminus \A' =  \A^b_0(n) \sqcup \A_{<0}^b(n)$ in the same way.
\smallskip

To prove that $\A_{>0}^t(n) $ is eventually empty, we prove that this set is shrinking
as $n\in \SW$ goes to $-\infty$. 
\smallskip

Consider $n_{k+1} <n_k$ two consecutive switching times.

\smallskip
If $\gamma_{n+1}=\gamma(n,n+1)$ is of top type for $n_{k+1} < n \leq n_k$, then $\pi_t^{(n_k)} = \pi_t^{(n_{k+1})}$ hence $\A_{>0}^t(n_k)= \A_{>0}^t(n_{k+1}) $.
Observe that, by Lemma \ref{l9}, the winner of the arrows $\gamma_{n+1}$, for $n_{k+1} < n \leq n_k$, belongs to $\A_0^t(n_k)$. 

\smallskip

Assume that $\gamma_{n+1}=\gamma(n,n+1)$ is of bottom type for $n_{k+1} < n \leq n_k$. 
Let $\alpha_w$ be the winner of the arrows $\gamma_{n+1}$, $n_{k+1} < n \leq n_k$. 
Define $m:= \pi_t^{(n_k)}(\alpha_w)$. The successive losers of the arrows 
$\gamma_{n_k -j }$, $j=1,\ldots, n_{k}-n_{k+1} $ are the letters $\alpha(\bar j)$ with $\pi_t^{(n_k)}(\alpha(\bar j)) = m  +\bar j$, $1 \leq \bar j \leq d-m$ and $j \equiv \bar j$ mod.$(d-m)$.
There are two cases.

\begin{itemize}
\item Assume that  the letter $\alpha_w$ belongs to $\A^t_0(n_k)$. Then the loser of 
$\gamma_{n_{k+1}+2}$, which is also the winner of $\gamma_{n_{k+1}+1}$, also belongs to $\A^t_0(n_k)$ (from Lemmas \ref{l6} and \ref{l7}). One has $\A_{>0}^t(n_k)= \A_{>0}^t(n_{k+1}) $.
\item Assume that  the letter $\alpha_w$  belongs to $\A^t_{>0}(n_k)$. Then we must have (from Lemmas \ref{l6} and \ref{l7}) that the loser $\alpha_\ell$ of $\gamma_{n_{k+1} +2}$, which is also the winner of $\gamma_{n_{k+1}+1}$, satisfies $\alpha_\ell \in \A \setminus \A'$, $\pi_t^{(n_k)}(\alpha_w) < \pi_t^{(n_k)}(\alpha_\ell)\leq d-2$ and
$$
\sum_{{\Small \begin{array}{cc}\beta \in \A'\\ \pi_t^{(n_k)}(\alpha_w)< \pi_t^{(n_k)} (\beta) < \pi_t^{(n_k)} (\alpha_\ell)\end{array}}} \tau_\beta =0
$$
Therefore $\alpha_\ell$ belongs to $\A^t_{>0}(n_k)$ and to $\A^t_0(n_{k+1})$. On the other hand, no element $\alpha$ of $\A_0^t(n_k)$ satisfies $ \pi_t^{(n_k)}(\alpha_w)< \pi_t^{(n_k)} (\alpha) < \pi_t^{(n_k)} (\alpha_\ell)$, therefore $\A_0^t(n_k) \subset \A_0^t(n_{k+1})$.

\end{itemize}

We now can conclude that $\A^t_{>0}(n)$, $n\in \SW$, is eventually empty. Indeed, a letter which would belong to $\A^t_{>0}(n)$ for all $n \in \SW$ is the winner of $\gamma_n$ for infinitely many $\gamma_n$, $n \in \SW$. But we have seen that it cannot be the winner of infinitely many arrows of top type, and that each time that it is the winner of an arrow of bottom type, the cardinality of $\A^t_{>0}(n)$ is decreasing.

\smallskip

We prove similarly that $\A_{<0}^b(n)$, $n \in \SW$, is eventually empty.
\end{proof}

\subsection{Proof of Proposition~\ref{p7}}

 Take $n\in \SW$ large enough such that $\A_0^t(n) = \A \setminus \A'$. Let $m$ be the smallest value of $\pi_t^{(n)}(\alpha)$, for $\alpha \in \A'$. From lemma \ref{l10}, there exists $m' >m$ such that $\beta \in \A'$ for $m \leq \pi_t^{(n)}(\beta) \leq m'$, and
 $$
 \sum_{m \leq  \pi_t(n)(\beta) \leq m'} \tau_\beta =0 , \quad \sum_{m \leq  \pi_t^{(n)}(\beta) \leq m_1} \tau_\beta >0, \quad \forall m \leq m_1 <m'.
 $$
  Then we have an horizontal saddle connection!

\section{Adaptation of the proofs of results on  cohomological equation for absolute  Roth type i.e.m.}\label{app:proofs}

 Our goal in this Appendix is to check that the conclusion of Theorem 3.10 of \cite{MY}, namely Theorem~\ref{thmHolder1}, as well as the conclusions of Theorem 3.11 and Theorem 3.22 of \cite{MY}, are still valid under the assumption that the i.e.m is of {restricted} \emph{absolute} Roth type (instead that under the stronger requirement of being of restricted Roth type). The generalization of  Theorem 3.10 of \cite{MY} was already stated as Theorem \ref{thmHolder1} in \S~\ref{sec:results}.  We begin the Appendix by  stating the generalizations of Theorems 3.11 and 3.22 of \cite{MY} (stated as Theorem~\ref{thmcohom2} and~\ref{thmHolder2} below), then review their proofs 
indicating which (small) modifications are needed.
\subsection{Estimates for special Birkhoff sums with $C^1$ data}\label{ssC1}
The following result extends Theorem 3.11 of \cite{MY}, which was proved for i.e.m.\ of restricted Roth type,  to i.e.m. of restricted \emph{absolute} Roth type.
\begin{theorem}\label{thmcohom2}({\bf Generalization of Theorem 3.11 of \cite{MY}})
Let $T$ be an i.e.m. which satisfies conditions  (a'), (c), (d) of sections \ref{ssRoth}  and \ref{ssabs}. Let  
$\Gamma_u(T)$ be a subspace of $\Gamma_\partial (T)$
which is supplementing $\Gamma_s(T)$ in $\Gamma_\partial (T)$. The operator $L_1: \varphi \mapsto \chi$ of Theorem \ref{thmHolder1} extends to a bounded operator
 from $C^1_{0} (\sqcup I_\alpha (T))$ to $\Gamma_u(T)$. For $\varphi \in C^1_{0} (\sqcup I_\alpha (T))$, the function $\chi= L_1(\varphi) \in \Gamma_u(T)$ is characterized by the property that the special Birkhoff sums of $\varphi -\chi$ satisfy, for any $\tau >0$

 $$ \Vert S(0,n) (\varphi -\chi) \Vert_{C^0} \leq C(\tau) \Vert \varphi \Vert_{C^1} \Vert B(0,n)\Vert^\tau .$$ 

\end{theorem}

 The inequality of the theorem for special Birkhoff sums implies a similar inequality for general Birkhoff sums:
\begin{corollary}\label{remtime} For any $T$ and $\varphi$ as in Theorem \ref{thmcohom2}, for all $\tau>0$, we have 
$$ \Vert \sum_{j=0}^{N-1} (\varphi -\chi) \circ T^j \Vert_{C^0} \leq C(\tau) \Vert \varphi \Vert_{C^1} N^\tau .$$
\end{corollary}

\subsection{Higher differentiability}\label{higherdiff}
To formulate a result allowing for smoother solutions of the cohomological equation which extend Theorem 3.22 in \cite{MY}, we introduce the same finite-dimensional spaces than in \cite{MMY2} (although the notations are slightly different). In the following definitions, $T$ is an i.e.m., {\Green $\kappa$} is a nonnegative integer, and  $r$ is a real number $\geq 0$.

\begin{definition}\label{defcrD} We denote by 

\begin{itemize}
\item $C^{{\Green \kappa}+r}_{{\Green \kappa}}(\sqcup I_\alpha (T))$ the subspace of 
$C^{{\Green \kappa}+r}(\sqcup I_\alpha (T))$ consisting of functions $\varphi$ which satisfy $\partial {\Green \kappa}^i \varphi =0$ for $0 \leq i \leq {\Green \kappa}$.
\item  $\Gamma({\Green \kappa},T)$ the subspace of $C^{\infty}_{{\Green \kappa}}(\sqcup I_\alpha (T))$ consisting  of functions $\chi$ whose restriction to each $[u_{i-1}(T), u_i(T)]$ is a polynomial of degree $\leq {\Green \kappa}$.
\item  $\Gamma_s ({\Green \kappa}, T)$ the subspace of $\Gamma({\Green \kappa},T)$ consisting of functions $\chi$ which can be written as $\chi = \psi \circ T - \psi$, for some function $\psi$ of class $C^{{\Green \kappa}}$ on the closure of $[u_0(T), u_d(T)]$.
\end{itemize}
\end{definition} 

The proof given in \cite{MMY2} (p.1602) that if  $T$ is of  restricted Roth type then  the dimension of $\Gamma ({\Green \kappa},T)$ is equal to
$2g + {\Green \kappa}(2g-1)$, and that the dimension of $\Gamma_s ({\Green \kappa},T)$ is equal to $g+{\Green \kappa}$, 
 also holds, essentially unmodified, under the assumption that $T$ is of  absolute restricted Roth type.

\begin{theorem}\label{thmHolder2} ({\bf Generalization of Theorem 3.22 of \cite{MY}})
Let $T$ be an i.e.m.  of absolute restricted Roth type. Let 
${\Green \kappa}$  be a nonnegative integer, and let $r$ be a real number $> 1$. Let 
$\Gamma_u({\Green \kappa},T)$ be a subspace of $\Gamma ({\Green \kappa},T)$ 
 supplementing $\Gamma_s ({\Green \kappa},T)$.  There exist a real number $\bar \delta >0$, and bounded linear operators $L_0: \varphi \mapsto \psi$ from 
 $C^{{\Green \kappa}+r}_{{\Green \kappa}} (\sqcup I_\alpha (T))$ to $C^{{\Green \kappa}+ \bar \delta} ([u_0(T),u_d(T)])$ and $L_1: \varphi \mapsto \chi$ from 
$C^{{\Green \kappa}+r}_{{\Green \kappa}} (\sqcup I_\alpha (T))$ to $\Gamma_u({\Green \kappa},T)$ such that any $\varphi \in C^{{\Green \kappa}+r}_{{\Green \kappa}} (\sqcup I_\alpha (T))$ satisfies
$$ \varphi = \chi + \psi \circ T - \psi.$$
The number $\bar \delta$ is the same than in Theorem \ref{thmHolder1}.
\end{theorem}

The derivation of Theorem 3.22 from Theorem 3.10 (the case ${\Green \kappa} = 0$)   is easy and done in Section 3.2 of \cite{MMY2} (see p.1602-1603). Following the same arguments, Theorem~\ref{thmHolder2} follows from Theorem~\ref{thmHolder1}. We hence will now focus on the adaptation of the {\Green proofs} of Theorems 3.10 and 3.11, needed to prove Theorems~\ref{thmHolder1} and \ref{thmcohom2}.

\subsection{Growth of $B(0,n_\ell)$}
 
At many places in the proof of Theorems 3.10 and 3.11, one uses the following estimate, which we now prove under condition $\mathrm{(a')}$ of absolute Roth type  instead than condition $\mathrm{(a)}$ of Roth type. 

\begin{proposition}\label{propo3}
Assume that condition $\mathrm{(a')}$ of the definition of absolute Roth type is satisfied. Then, for any $\delta >0$, there exists a constant $C_\delta$ such that, for all $k \geq 0$, we have
$$
\sum_{\ell \geq k} \Vert B(0,n_\ell)\Vert^{-\delta} \leq C_\delta \Vert B(0,n_k)\Vert^{-\delta}.
$$
\end{proposition}

\begin{proof} {\Green The desired estimate will be a consequence of the exponential growth of $\Vert B(0,n_\ell)\Vert$,
implied by the absolute Roth type condition, which will guarantee that the whole sum is controlled by its first term.} 
Let $v$ be a fixed vector in $\C(\pi(n_0))$. The sequence $\vert B(0,n_\ell) v\vert_1$ (where $\vert \ \vert_1$ is the $\ell_1$-norm on $\RA$) is non decreasing and we will show that it grows fast enough to obtain the conclusion of the proposition. Let $v(\ell):= 
B(0,n_\ell) v$. By Lemma \ref{l1}, one has 
$$
\max_\alpha v(\ell)_\alpha \leq C \Vert B(n_{\ell -1},n_\ell) \Vert \min_\alpha v(\ell)_\alpha .
$$

We have therefore

$$
\vert v(\ell+1) \vert_1 \geq \vert B(n_\ell, n_\ell +1)v(\ell) \vert_1 \geq (1 + c \Vert B(n_{\ell -1},n_\ell) \Vert^{-1}) \vert v(\ell) \vert_1
$$

It follows that the smallest $\ell'$ such that $\vert v(\ell') \vert_1 \geq 2 \vert v(\ell)  \vert_1$ satisfies, for any $\tau >0$ (and an appropriate constant $C_\tau$)
$$ \ell' \leq \ell + C_\tau \Vert B(0,n_\ell) \Vert^\tau.$$

{\Green Since for any fixed vector  $v\in\C(\pi(n_0))$
the sequence $\vert v (\ell)\vert_1$ doubles in at most $C_\tau \Vert B(0,n_\ell) \Vert^\tau$ steps, the same holds for the 
norm $\Vert B(0,n_\ell) \Vert$. Let $\ell_0=n_k$ and define recursively $\ell_{j+1}$ equal to the the smallest 
$\ell'$ such that $\Vert B(0,n_{\ell'})\Vert\geq 2\Vert B(0,n_{\ell_j})\Vert$. The bound on $\ell' - \ell$ obtained above gives
$$
\sum_{\ell \geq k} \Vert B(0,n_\ell)\Vert^{-\delta} \leq \Vert B(0,n_k)\Vert^{-\delta} \sum_{j\ge 0} 
C_\tau \Vert B(0,n_{\ell_j)} \Vert^\tau 2^{-\delta}
\Vert B(0,n_{\ell_j})\Vert^{-\delta}.
$$
}
Taking $\tau = \frac 12 \delta$ gives the required inequality.
\end{proof}

Along the same lines, one also have (with a very similar proof):
\begin{proposition}\label{propo4}
Assume that condition (a') is satisfied. Then, for any $\delta >0$, there exists a constant $C_\delta$ such that, for all $k \geq 0$, we have
$$
\sum_{0 \leq \ell \leq k} \Vert B(0,n_\ell)\Vert^{\delta} \leq C_\delta \Vert B(0,n_k)\Vert^{\delta}.
$$
\end{proposition}

\subsection{Adaptation of the proof of Theorems 3.10 and 3.11 in \cite{MY}}
We here follow Section 3.5 of \cite{MY}.  
As mentioned above (Corollary \ref{c1}), the bound from above for $\vert I_\alpha^{(n)} \vert$ is still valid under condition (a'). The bound from below in Lemma 3.14 of \cite{MY} is not used in the proof of theorem 3.11.

\smallskip

The sequence $(n_k)$ which appears in the proof of theorem 3.11 (and was defined at the end of Section 2.5 of \cite{MY}) is the sequence that we have called  $(\wt n_k)$ in Section \ref{ssabs} of the present text. We have to replace it by the sequence $(n_k)$ of condition (a'). The only property satisfied by $(\wt n_k)$ under condition (a) which is not satisfied by $(n_k)$ under condition (a') is the lower bound for $\vert I_\alpha^{(n)}\vert$ in Lemma 3.14.

\subsubsection{Adaptation of the proof of Lemma 3.15 in \cite{MY}}

In \cite{MY}, we refer for the proof  to an annex in \cite{MMY2}. The proof in this paper depends on a claim (Proposition \ref{p5} below) which is still true, but the proof  has to be modified.

\medskip

\begin{proposition}\label{p5}
Assume that the rotation number $\ug$ satisfies condition  (a').
For every $\tau >0$, there exists $C(\tau)>0$ such that, for any $0 \leq \ell \leq k$, one has 
$$ \Vert B(n_0,n_k) \Vert \leq \Vert B(n_0,n_{\ell}) \Vert \Vert B(n_{\ell},n_k) \Vert \leq C(\tau) 
\Vert B(n_0,n_k) \Vert^{1+\tau}
$$
\end{proposition}

\begin{proof}
The first inequality is trivial. To prove the second, choose $v \in \C(\pi^{(n_0)})$, $v' \in \C(\pi^{(n_\ell)})$, and define $w:= B(n_0, n_\ell)\cdot v$. One has 

\begin{eqnarray*} 
\Vert B(n_0,n_k) \Vert &\geq& \Vert v \Vert^{-1} \Vert B(n_0,n_k) v \Vert = \Vert v \Vert^{-1} \Vert B(n_\ell,n_k) w \Vert , \\
\Vert B(n_0,n_\ell) \Vert & \leq&  C(\pi(n_0))  \Vert v \Vert^{-1} \Vert B(n_0,n_\ell) v \Vert =  C(\pi(n_0))  \Vert v \Vert^{-1} \Vert w \Vert, 
\end{eqnarray*}
and therefore
$$
\frac { \Vert B(n_0,n_\ell) \Vert  \Vert B(n_\ell,n_k) \Vert}{\Vert B(n_0,n_k) \Vert }
\leq  C(\pi^{(n_0)}) \frac {\Vert B(n_\ell,n_k) \Vert \Vert w \Vert }{ \Vert B(n_\ell,n_k) w \Vert}.
$$
\smallskip

From Lemma \ref{l1}, one has  
$$ \max_\alpha w_\alpha \leq C \Vert B(n_{\ell -1},n_\ell) \Vert \min_\alpha w_\alpha.$$
and therefore 
$$ \frac {\Vert B(n_\ell,n_k) \Vert \Vert w \Vert }{ \Vert B(n_\ell,n_k) w \Vert}  \leq C'  \Vert B(n_{\ell -1},n_\ell) \Vert  \leq C_{\tau} \Vert B(n_0,n_k) \Vert^{\tau},$$
giving the required estimate.

\end{proof}


 The proof of Lemma 3.16 in \cite{MY} is completely similar to the proof of Theorem 3.11 of \cite{MY} and the same remarks for its generalization apply.

\subsubsection{Space and Time decompositions}
Following sections 3.7 and 3.8 in \cite{MY},  
one has just to remember that the sequence $n_k$ in the present setting is not the same than in \cite{MY}. Otherwise, the same formalism works and yields the same type of space and time decompositions. 

\medskip

Section 3.9 also does not need any modification (except, as always, for the change of definition of the $(n_k)$).

\subsubsection{General H\"older estimates}
We now follow Section 3.10 of \cite{MY}.  The bound from below in Lemma 3.14 is used (for the first and last time) at the end of this section.  However, Proposition \ref{p6} above is an adequate substitute. 

As in Section 3.10 in \cite{MY}, we write any interval $(x_-,x_+)$ as a countable union of intervals
$J$ which belong to some $\mathcal P(\ell)$, $\ell \geq \ell_{min}$ (We use $\ell_{min}$ instead of $k$ because we are going to use Proposition \ref{p6}).

Let $J \in \mathcal P(\ell)$ be one of these intervals. It is a translated copy of some 
$I_\alpha^{(n_\ell)}$. We apply Proposition \ref{p6}. It gives an integer $k$ such that
\begin{itemize}
\item $J$ is the union of at most $(s-1)$ intervals of $\mathcal P(k)$;
\item $\vert J \vert \geq C^{-1}_{\tau} \ \vert I^{(0)} \vert \ \Vert B(0,n_k) \Vert^{-\tau-1}.$
\end{itemize}

From Lemma 3.20 of \cite{MY} and the first item above, we have that 
$$(\star) \qquad  \vert \Delta(J) \vert \leq C \Vert B(0,n_k) \Vert^{-\delta_2} \Vert \varphi \Vert_{C^r}.$$
(where $\Delta(J)$ is the difference between the values of $\psi$ at the endpoints of $J$)

Let $k_{min}$ be the minimal value of $k$ when $J$ runs among the intervals of the space decomposition of $(x_-,x_+)$. We have then
$$ \vert x_+ -x_- \vert \geq C^{-1}_{\tau} \ \vert I^{(0)} \vert \ \Vert B(0,n_{k_{min}}) \Vert^{-\tau-1}$$
For $\ell_{min} \leq m $, there are at most $\Vert B(n_m,n_{m+1}) \Vert$
intervals in $\mathcal P(m)$ in the space decomposition of $(x_-,x_+)$. 
For  intervals $J$ with $m \leq k_{min}$, we use $(\star)$. For the other intervals  with $m > k_{min}$, we just use Lemma 3.20. We get

\begin{align*}
\vert \psi(x_+) - \psi(x_-) \vert & \leq C  \Vert \varphi \Vert_{C^r} \lbrace  \Vert B(0,n_{k_{min}}) \Vert^{-\delta_2} \left(\sum_{\ell_{min} \leq m < k_{min}} \Vert B(n_m,n_{m+1}) \Vert \right)\\
&+ \sum_{m \geq k_{min}} \Vert B(n_m,n_{m+1}) \Vert  \ \Vert B(0,n_m)) \Vert^{-\delta_2} \rbrace \\
& \leq  C'  \Vert \varphi \Vert_{C^r}  \sum_{m \geq k_{min}}   \ \Vert B(0,n_m)) \Vert^{-\delta_2/2} \\
& \leq C''  \Vert \varphi \Vert_{C^r}   \Vert B(0,n_{k_{min}}) \Vert^{-\delta_2/2}.
\end{align*}

We have used Propositions \ref{propo3} and \ref{propo4} (with $\delta = \delta_2/2$). The proof of the required inequality is complete. 

\smallskip
This concludes the modifications required to prove Theorem~\ref{thmcohom2} and Theorem~\ref{thmHolder1}, and hence also Theorem~\ref{thmHolder2}.


%
%

\end{document}